\documentclass[11pt]{article}
\usepackage[T1]{fontenc} 
\usepackage{amsmath} 
\usepackage{amsthm}
\usepackage{amssymb}
\usepackage{mathtools}
\usepackage{etoolbox}

\usepackage[percent]{overpic}
\usepackage{tikz}
\usepackage{xcolor}

\usepackage{enumitem}
\usepackage{hyperref}
\hypersetup{
	colorlinks   = true, 
	urlcolor     = blue, 
	linkcolor    = blue, 
	citecolor   = red 
}
\usepackage[final
]{showlabels}

\DeclareFontFamily{U}{stix2bb}{}
\DeclareFontShape{U}{stix2bb}{m}{n} {<-> stix2-mathbb}{}

\NewDocumentCommand{\indicator}{}{\text{\usefont{U}{stix2bb}{m}{n}1}}

\setenumerate[2]{label=(\alph*)}
\setenumerate[1]{label=\textup{(\roman*)}}
\setenumerate[3]{label=(\arabic*)}

\newcommand{\normW}[1]{{\vert\kern-0.25ex\vert\kern-0.25ex\vert #1 
		\vert\kern-0.25ex\vert\kern-0.25ex\vert}}
\newcommand{\normWbig}[1]{{\Big\vert\kern-0.25ex\Big\vert\kern-0.25ex\Big\vert #1 
		\Big\vert\kern-0.25ex\Big\vert\kern-0.25ex\Big\vert}}

\newcommand{\C}{\mathbb{C}}
\newcommand{\N}{\mathbb{N}}

\newcommand{\R}{\mathbb{R}}

\renewcommand{\S}{\mathbb{S}}

\newcommand{\Z}{\mathbb{Z}}

\newcommand{\boA}{\mathcal{A}}

\newcommand{\boC}{\mathcal{C}}

\newcommand{\boF}{\mathcal{F}}
\newcommand{\boG}{\mathcal{G}}

\newcommand{\boL}{\mathcal{L}}
\newcommand{\boM}{\mathcal{M}}
\newcommand{\boN}{\mathcal{N}}

\newcommand{\boP}{\mathcal{P}}

\newcommand{\boR}{\mathcal{R}}

\newcommand{\boV}{\mathcal{V}}

\newcommand{\boX}{\mathcal{X}}

\newcommand{\gp}{\mathfrak{p}}

\newcommand{\ka}{\kappa}
\newcommand{\ve}{\varepsilon}

\newcommand{\loc}{\operatorname{loc}}

\renewcommand{\Re}{\operatorname{Re}}
\renewcommand{\Im}{\operatorname{Im}}

\newcommand{\atan}{\operatorname{atan}}
\renewcommand{\mod}{\textup{ mod }}
   
\newcommand{\wto}{\rightharpoonup}

\theoremstyle{plain}
\newtheorem{theorem}{Theorem}[section]
\newtheorem{proposition}[theorem]{Proposition}
\newtheorem{corollary}[theorem]{Corollary}

\newtheorem{lemma}[theorem]{Lemma}
\theoremstyle{definition}
\newcounter{claimcount}
\newtheorem*{step}{\refstepcounter{claimcount}\textbf{Step \arabic{claimcount}}}{}
\AtBeginEnvironment{proof}{\setcounter{claimcount}{0}}

\newtheorem{remark}[theorem]{Remark}

\theoremstyle{remark}

\newtheorem*{merci}{Acknowledgments}

\numberwithin{equation}{section}

\begin{document}

	\renewcommand{\thefootnote}{\fnsymbol{footnote}}
	\title{Orbital stability of black solitons for quasilinear Schr\"odinger equations with nonzero conditions at infinity}
		\renewcommand{\thefootnote}{\arabic{footnote}}\author{\renewcommand{\thefootnote}{\arabic{footnote}}
		 Erwan Le Quiniou\footnotemark[1]}
	
	\footnotetext[1]{
		Univ.\ Lille, CNRS, Inria, UMR 8524, Laboratoire Paul Painlevé, F-59000 Lille, France.\\
		E-mail: {\tt erwan.lequiniou@univ-lille.fr}}
	\date{}
	\maketitle

	%
	%
	\begin{abstract} 
We investigate the orbital stability of black solitons for a broad class of quasilinear Schr\"odinger equations in one space dimension, with nonzero boundary conditions at infinity. Namely, our framework handles general defocusing semilinear nonlinearities and focusing or defocusing quasilinear nonlinearities. First, we establish sufficient conditions on the quasilinear nonlinearities ensuring the existence of a local branch of finite-energy solitons parameterized by their speed. Within this branch, the black soliton, also called kink, corresponds to the stationary solution.

Our main result is the orbital stability of the black soliton in the energy space, provided that the Vakhitov--Kolokolov (VK) slope condition holds; namely, that the derivative of the momentum with respect to the speed is negative at zero. Moreover, we derive an explicit formula for verifying this VK condition.

The proof relies on the analysis of a carefully designed variational problem, which allows us to control the sup-norm of the evolution of a perturbation of the kink in terms of the energy and momentum, both of which are conserved by the flow. A delicate part of the argument is the analysis of minimizing sequences for this variational problem, since the infimum is not attained.

	\end{abstract}
	
	\medskip
	\noindent{{\em Keywords:}
		Quasilinear Schr\"odinger equation, Gross--Pitaevskii equation, traveling waves, black soliton, kink, nonzero conditions at infinity, orbital stability.
		
		\medskip
		\noindent{2020 {\em MSC}}:
		35Q55, 
		35C07, 
		35C08, 
		35J62,  
		35B35,  
		35Q60, 
		82D50, 
			35Q56, 
		   37K40,  
		34B40.  	
		\medskip
        \section{Introduction and main results}
        We are interested in quasilinear Schr\"odinger models arising in nonlinear optics \cite{krolikowski2000,krowlikowskiMI}, quantum hydrodynamics \cite{rutledgethirdsound,Kurihara,horikisWnonloc} and magnetohydrodynamics \cite{MHDQGP}. The equation writes, for $\Psi:\R\times\R\to\C$, and $\ka\in\R$
  \begin{equation}\label{QGP}\tag{QLS}
      i\partial _t\Psi+\partial_{xx}\Psi+\Psi f(|\Psi|^2)+\kappa \Psi h'(|\Psi|^2)\partial_{xx} h(|\Psi|^2)=0,\quad\text{ for all  }(x,t)\in\R\times\R.
  \end{equation}
  We assume that $h$ and $f$ are smooth real-valued functions with $f(r_0^2)=0$ for some $r_0>0$, and that $\Psi$ satisfies the nonzero conditions at infinity:
  \begin{equation}\label{eq:nonzero}
      \lim_{|x|\to\infty}|\Psi(x,t)|=r_0,\quad\text{ for all }t\in\R.
  \end{equation}
  The local well--posedness (LWP) of equation \eqref{QGP}--\eqref{eq:nonzero} can be handled by adopting the general framework of quasilinear Schr\"odinger equations developed in the celebrated work of C.~Kenig, G.~Ponce and L.~Vega \cite{Kenig}. 
  In the setting of \eqref{QGP}, their results yields
  \begin{theorem}[\cite{Kenig}]\label{thm:lwp}
    Let $v\in L^\infty(\R)\cap \dot{H}^\infty(\R)$ satisfy $|v|^2-r_0^2\in L^2(\R)$.
    There exists $s_0>4$ such that for any $s>s_0$, if $\Psi_0\in v+ H^s(\R)$ satisfies $1+2\ka|\Psi_0|^2h'(|\Psi_0|^2)^2>0$, then there exist $T>0$ and a unique solution $\Psi\in\boC([0,T];v+H^s(\R))$ to \eqref{QGP}. Moreover, the energy and momentum of $\Psi$ are well-defined and conserved in time. 
  \end{theorem}
  In a work in preparation~\cite{LQCauchyQGP}, we show that \eqref{QGP} is well suited for energy estimates: we lower the regularity threshold to LWP in $H^{\frac52^+}(\R)$ for data of arbitrary size and refine the blow-up criterion. This result is on par with modern LWP results whose proofs involve more educated tools such as Paley--Littlewood decompositions~\cite{tataruIII}.
 \subsection{ Solitons and their stability}
          
          To our knowledge, the first interest of the mathematical community in \eqref{QGP} goes back to works from H.~Lange and L.~Brull~\cite{LangeVirial,brullLangeQGP} in 1986. They obtained a sufficient condition for the existence of classical standing
and traveling-wave solutions in $L^\infty(\R\times\R)$ and derived formal (pseudo)-conservation laws of equation \eqref{QGP}. Among them, let us recall the energy conservation law
          \begin{align}\label{def:energy}
    E_\ka(\Psi(\cdot))=\int_{\R}|\nabla\Psi(x)|^2dx+\int_{\R}F(|\Psi(x)|^2)dx+\frac\ka2\int_\R|\nabla h(|\Psi(x)|^2)|^2dx,
          \end{align}
              where $F(\sigma)=\int_{\sigma}^{r_0^2}f(w)dw$ is the antiderivative of $-f$ vanishing at $r_0^2$;
  and the (renormalized) momentum is given by
    \begin{align*}
           P(v)=\int\Re(i\partial_xv\bar v)\Big(1-\frac{r_0^2}{|v|^2}\Big)dx.
    \end{align*}
    To ensure that the energy is well-defined, it is standard to work in the so-called energy space
    \begin{equation*}\label{def:boX}
    \boX(\R)\coloneqq\{v\in H^1_{\loc}(\R):v'\in L^2(\R),\text{ and } |v|^2-r_0^2\in L^2(\R)\}\subset\boC_b(\R;\C),
    \end{equation*}
    which coincides with the domain of $E_\ka(\cdot)$ in $H^1_{\loc}(\R)$ whenever $\ka\geq 0$. If $v\in\boX(\R)$, then H\"older's inequality and the Sobolev embedding $H^1(\R)\hookrightarrow C^{\frac{1}{2}}(\R)$ provide $\lim_{|x|\to\infty}|v(x)|^2=r_0^2$. Besides, defining the momentum of $v\in\boX(\R)$, requires $\inf_{x\in\R}|v(x)|>0$, that is, $v$ belongs to the nonvanishing energy space given by
    \begin{equation*}
        \boN\boX(\R)=\{v\in\boX(\R):\inf_{x\in\R}|v(x)|>0\}.
    \end{equation*}
    In what follows, we endow $\boX(\R)$ with the distance \begin{equation}\label{def:dX}
        d_{\boX}^2(v,w)=\int_\R|\partial_xv-\partial_xw|^2dx+\int_\R(|v|-|w|)^2dx+|v(0)-w(0)|^2.
    \end{equation}

          \paragraph{Dark and black solitons} In the setting of nonvanishing conditions at infinity, we call a dark soliton a finite energy traveling-wave solution to \eqref{QGP} given by
          $\Psi(x,t)=u_c(x-ct)$ for some $c\in\R$.
          The traveling-wave equation writes
 \begin{equation}\label{GTWc}\tag{TW$(c,\ka)$}
     icu_{c,\ka}'=u_{c,\ka}''+u_{c,\ka}f(|u_{c,\ka}|^2)+\ka u_{c,\ka}h'(|u_{c,\ka}|^2)(h(|u_{c,\ka}|^2))'',\text{ in }\R,
 \end{equation}
 together with the condition $u_{c,\ka}\in\boX(\R)$, so that $|u_{c,\ka}|\to r_0$ as $|x|\to\infty$.
 Linearizing \eqref{QGP} around the constant solution with the plane wave ansatz $\Psi(x,t)=r_0+\ve e^{i(\omega t-kx)}$, we find the dispersion relation $$\omega(k)=(1+2\ka r_0^2h'(r_0^2)^2)|k|^4-2|k|^2f'(r_0^2),$$ which highlights two main features of \eqref{QGP}. First, the speed of sound, corresponding to the speed threshold for the existence of nontrivial traveling-wave solutions, is given by $$c_s^2\coloneqq\lim_{k\to0}w(k)/k^2=-2f'(r_0^2).$$
 Then, \eqref{QGP} can be seen as an intensity-dependent dispersive equation; thus, as explained in Proposition~\ref{prop:illvar}, one cannot obtain energy estimates for variational problems unless the dispersive part in \eqref{QGP} satisfies the ellipticity condition 
\begin{equation}\label{cond:nondeg}
    1+2\ka |\Psi|^2h'(|\Psi|^2)^2>0, \quad\text{ in }\R\times\R.
\end{equation}
We say that a solution to \eqref{QGP} satisfying \eqref{cond:nondeg} is nondegenerate. In that setting, a kink or black solution $u_{c,\ka}\in\boX(\R)$ is a nondegenerate solution to \eqref{GTWc} that vanishes at some point, say $u_{c,\ka}(0)=0$ by translation invariance. This definition motivates the introduction of our assumptions on the functions $f$ and $h$ and on the parameter $\ka\in\R$. For the rest of this introduction, we assume the following
\begin{align}\label{def:hyp1}
  h\in\boC^\infty([0,\infty);\R), \quad\text{ and }\quad  f\in\boC^\infty([0,\infty);\R),\tag{H1}\\
    1+2\ka \sigma (h'(\sigma))^2>0,\quad\text{ for all }\sigma\in[0,r_0^2],\tag{H2}\label{def:hyp2}
\end{align}
and, denoting $F(\sigma)=\int_{\sigma}^{r_0^2}f(w)dw,$ we suppose that
\begin{align}\label{def:hyp3}
 \quad F(\sigma)>0,\text{ for all }\sigma\in[0,r_0^2), \quad\text{ and } F''(r_0^2)>0.\tag{H3}
\end{align}
The assumption \eqref{def:hyp2} imposes $\ka>\tilde\ka$ where 
   \begin{equation}\label{def:katil}
         \tilde\ka=\sup_{\sigma\in[0,r_0^2]}\Big(-\frac{1}{2\sigma (h'(\sigma))^2}\Big)<0.
           \end{equation}
           Going back to the expression of the energy~\eqref{def:energy}, we see that our framework handles defocusing semilinear nonlinearities and focusing or defocusing quasilinear nonlinearities.
Under Assumption~\eqref {def:hyp2}, we will see in Theorem~\ref{thm:soli} that Assumption~\eqref{def:hyp3} is sufficient for the existence of a smooth, rapidly decaying nondegenerate black soliton. Up to a translation and a constant phase change, this black soliton corresponds to an odd real-valued stationary solution $u_{0,\ka}\in\boX(\R)$ to \eqref{QGP} given by the implicit formula
\begin{equation}\label{eq:implikink}
    \int_{0}^{|u_{0,\ka}(x)|^2}\sqrt{\frac{1+2\ka \sigma h'(\sigma)^2}{4\sigma F(\sigma)}} dr=|x|.
\end{equation}
 In \eqref{def:hyp1}, we assume smoothness of the nonlinearities for the sake of simplifying the presentation of the results; we refer to Remark~\ref{rem:regu} for possible extensions in this regard. Since $c_s>0$ by Assumption~\eqref {def:hyp2}, we show in Proposition~\ref{prop:implisol} that this well-behaved black soliton belongs to a unique continuous branch of finite energy traveling-wave solutions $u_{c,\ka}$ for $c\in[0,\delta)$ and all $\ka>\tilde\ka$.
We also show that the mapping $c\mapsto\inf_{x\in\R}|u_{c,\ka}(x)|$ is increasing. In that setting, we call a gray soliton any nonvanishing traveling-wave solution to \eqref{GTWc} with $c>0$. Using this, we can define the function
\begin{equation}\label{def:Pka}
    (0,\delta)\ni c\mapsto P_\ka(c)\coloneqq P(u_{c,\ka}).
\end{equation}
In Proposition~\ref{prop:moment0}, we show that $P_\ka(\cdot)$ can be extended to a $\boC^1$-function in $[0,\delta)$ and obtain the crucial formula 
\begin{equation*}
    P_\ka'(0)=-\frac{8r_0^3}{3\sqrt{F(0)}}+\int_{0}^{r_0^2}\frac{(r^2-r_0^2)^2}{r^{2}}\left(\sqrt{\frac{1+2\ka r^2 (h'(r^2))^2}{F(r^2)}}-\frac{1}{\sqrt{F(0)}}\right)dr. 
\end{equation*}
We can now state our main result, a slope criterion on $P_\ka(\cdot)$ to ensure the orbital stability of the black soliton.
\begin{theorem}\label{thm:stab}
    Assume that \eqref{def:hyp1}--\eqref{def:hyp3} hold, and let $\ka>\tilde\ka$, so that there exists a unique nondegenerate kink $u_{0,\ka}\in\boX(\R)$. Let $s\geq1$ be such that LWP holds in $u_{0,\ka}+H^s(\R)$ in the sense of Theorem~\ref{thm:lwp}. If $P_\ka'(0)<0$, then there exist $K>0$ and functions $z(\cdot),$ $\varphi(\cdot)\in\boC^1([0,T);\R)$ such that
    \begin{equation}
        |z'(t)|+|\varphi'(t)|+d_{\boX}(\Psi(\cdot,t);u_{0,\ka}(\cdot-z(t))e^{i\varphi(t)})\leq K d_{\boX}^{{1}/{8}}(\Psi_0,u_{0,\ka}),\quad\text{ for all }t\in[0,T),
    \end{equation}
    where $0<T\leq\infty$ is the maximal time of existence of the solution $\Psi\in\boC([0,T);u_{0,\ka}+H^s(\R))$ to \eqref{QGP} starting from $\Psi(0)=\Psi_0$. 
\end{theorem}
\begin{remark}\label{rem:regu}
        We believe that Assumption~\eqref{def:hyp1} can be relaxed to $f\in\boC^4([0,\infty))$ and $h\in\boC^6([0,\infty))$ within our framework. With $f\in\boC^1(\R)$ and $h\in\boC^3(\R)$, one can show stability without estimations on the modulation parameters $(z,\varphi)\in\R^2$. Note, however, that the assumption of LWP in $H^s(\R)$ may prevent us from reaching such low regularities.
    \end{remark}

To exemplify our result, we consider three special cases for the nonlinearities in \eqref{QGP}, respectively used for the modeling of 
weakly nonlocal optics \cite{krolikowski2000}, optics in plasmas with relativistic effects~\cite{deBouardGWP}, 
and superfluids with surface tension \cite{Kurihara}:  
    \begin{enumerate}
\item \label{nonlin:1} $f(\sigma)=1-\sigma$, $h(\sigma)=\sigma,$ 
   \item \label{nonlin:2} $f(\sigma)=1-\sigma$, $h(\sigma)=\sqrt{1+\sigma},$
   \item \label{nonlin:3} $f(\sigma)= \Big(\frac{1+r_0^2}{1+\sigma}\Big)^3-1$, $h(\sigma)=\sigma.$
   \end{enumerate}

    For each case, Figure~\ref{fig:kink} shows the black soliton profile for several values of $\ka$, whereas Figure~\ref{fig:Pk0} displays the derivative of the momentum at $c=0$, with respect to $\ka$, close to $\tilde\ka$.
    As $\kappa$ becomes negative (focusing quasilinear terms), the black solitons become more localized with respect to the nonzero background, thereby reducing the magnitude of the stability criterion.
\begin{figure}[ht!]
		\begin{tabular}{cc}
		\resizebox{0.48\textwidth}{!}{
			\begin{overpic}
				[scale=0.5,trim=28 20 0 20,clip]{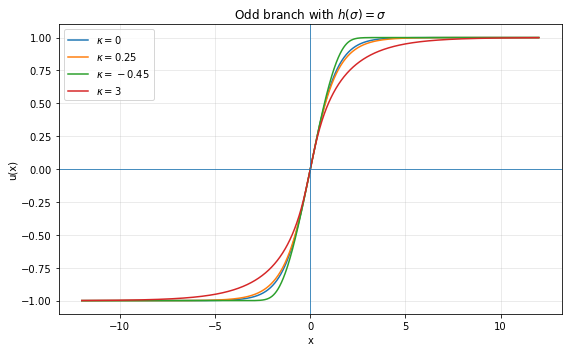}
				\put(95,0.5){$x$}
                \put(85,50){$u_{0,\ka}(x)$}
			\end{overpic}
		}&\resizebox{0.48\textwidth}{!}{
           \begin{overpic}
				[scale=0.5,trim=28 20 0 25,clip]{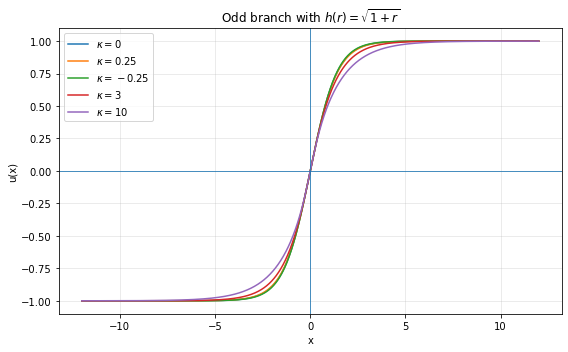}
                \put(95,0.5){$x$}
                \put(85,50){$u_{0,\ka}(x)$}
			\end{overpic}
            }
            
            \\
		\resizebox{0.48\textwidth}{!}{
			\begin{overpic}
				[scale=0.5,trim=28 20 0 25,clip]{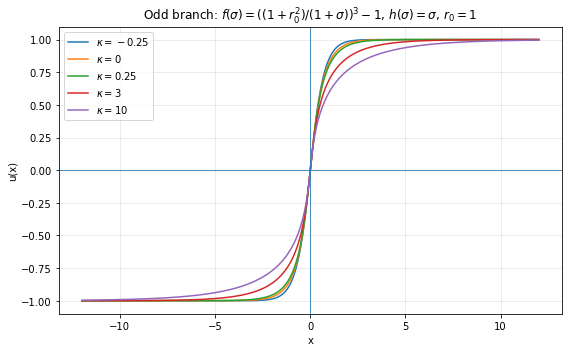}
                \put(95,0.5){$x$}
                \put(85,50){$u_{0,\ka}(x)$}
			\end{overpic}	
		}
		& \resizebox{0.48\textwidth}{!}{
			\begin{overpic}
				[scale=0.5,trim=28 20 0 25,clip]{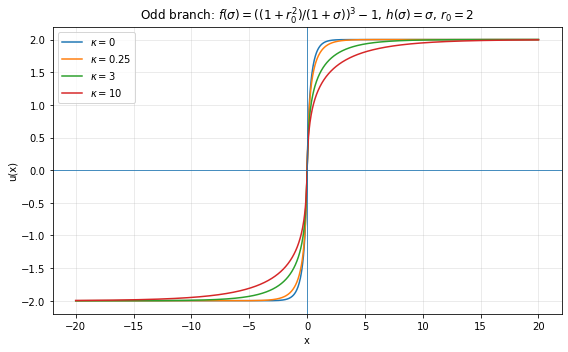}
				\put(97,0.5){$x$}
                \put(85,50){$u_{0,\ka}(x)$}
			\end{overpic}
		}
        
	\end{tabular}
\caption{Plot of the kink profile $u_{0,\ka}$  for several  nonlinearities and values of $\ka$. The top left panel displays profiles in Case~\ref{nonlin:1}. In the top right, profiles in Case~\ref{nonlin:2}.
The bottom left and right panels show black solitons in Case~\ref{nonlin:3} with $r_0=1$ and $r_0=2$, respectively.
}
\label{fig:kink}
\end{figure}
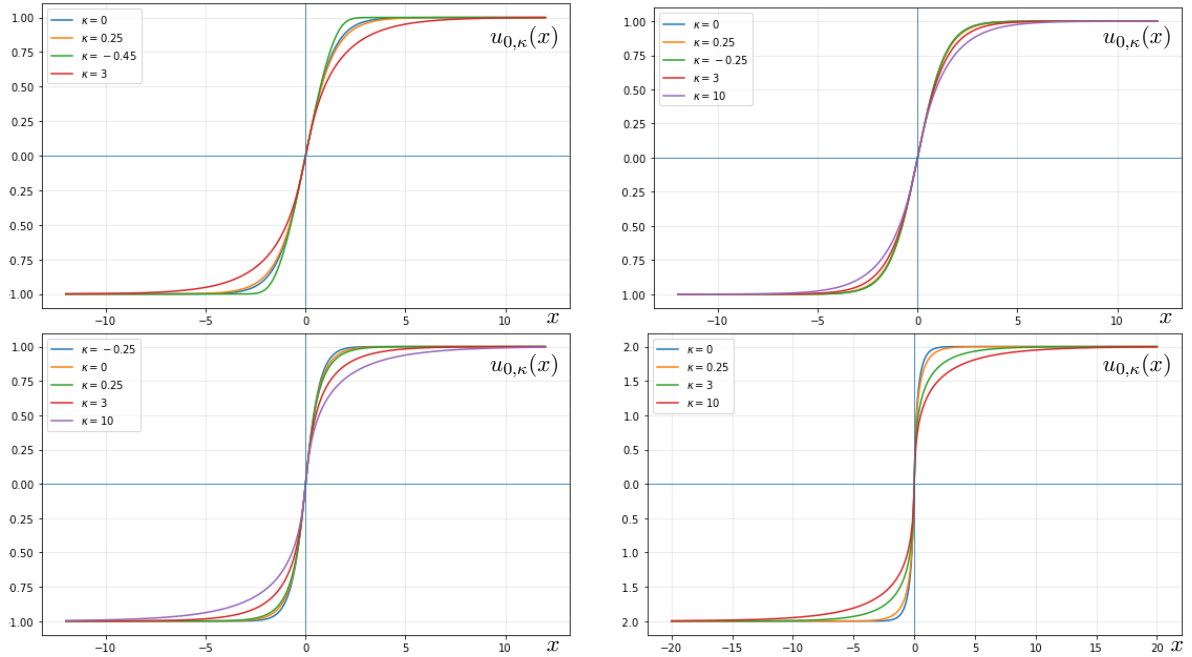
\begin{figure}[ht!]
		\begin{tabular}{cc}
		\resizebox{0.48\textwidth}{!}{
			\begin{overpic}
				[scale=0.5,trim=28 20 0 20,clip]{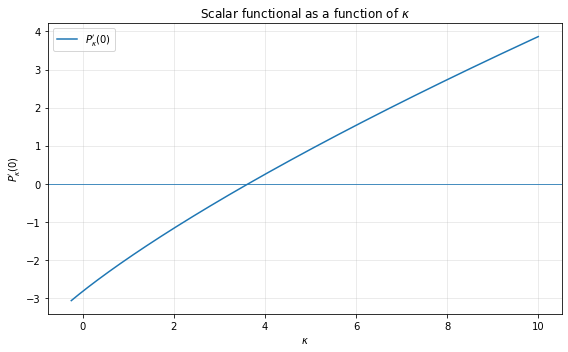}
				\put(97,0.5){$\ka$}
              
			\end{overpic}
		}&\resizebox{0.48\textwidth}{!}{
           \begin{overpic}
				[scale=0.5,trim=28 20 0 20,clip]{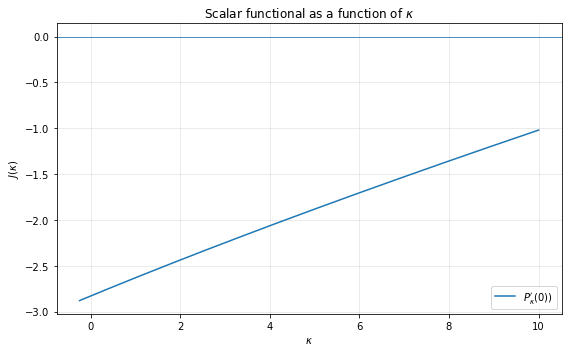}
                \put(97,0.5){$\ka$}
             
			\end{overpic}
            }\\
		\resizebox{0.48\textwidth}{!}{
			\begin{overpic}
				[scale=0.5,trim=28 20 0 25,clip]{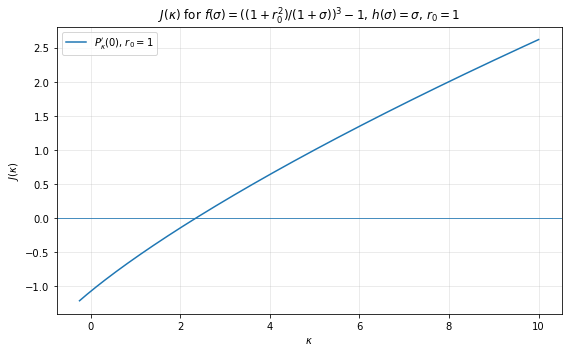}
                \put(97,0.5){$\ka$}
           
			\end{overpic}	
		}
		& \resizebox{0.48\textwidth}{!}{
			\begin{overpic}
				[scale=0.5,trim=28 20 0 25,clip]{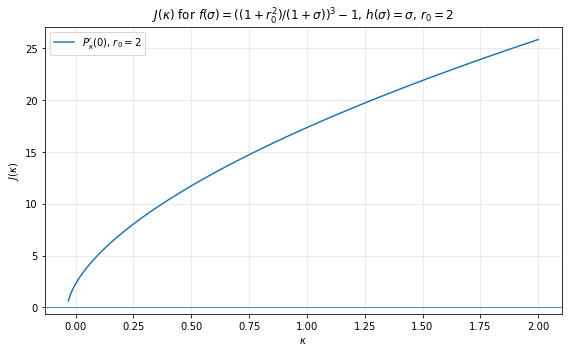}
				\put(97,0.5){$\ka$}
              
			\end{overpic}
		}
        
	\end{tabular}
\caption{Plot of $P_\ka'(0)$, the derivative of the momentum with respect to the speed at $0$, as a function of $\kappa$ near $\tilde\ka$, in Cases~\ref{nonlin:1}--\ref{nonlin:3}. The top left panel displays the slope condition in Case~\ref{nonlin:1}. In the top right, the condition in Case~\ref{nonlin:2}.
The bottom left and right pannels show $P_\ka'(0)$ with respect to $\ka$ in Case~\ref{nonlin:3} with $r_0=1$ and $r_0=2$, respectively.
}
\label{fig:Pk0}
\end{figure}

The following result supports this idea that the quasilinear term can contribute to the stability of the black soliton: Differentiating under the integral sign and applying Beppo Levi's theorem in the formula of $P_\ka'(0)$, we find
\begin{proposition}\label{prop:stabkappa}
Assume that \eqref{def:hyp1}--\eqref{def:hyp3} hold. Then    $P_\ka'(0)$ is an increasing function of $\ka$ in $(\tilde\ka,\infty)$, and $\lim_{\ka\to\infty}P_\ka'(0)=\infty.$
\end{proposition}
The Case~\ref{nonlin:1} with $h(\sigma)=\sigma$, $f(\sigma)=(1-\sigma)$, has been thoroughly studied by the author and  A.~de Laire in \cite{deLaire2023exotic,LeQuiniou2024stability}. We can take advantage of the explicit formulas of the momentum there to obtain
\begin{proposition}\label{prop:stabQGP}
    Let $h(\sigma)=\sigma$, $f(\sigma)=(1-\sigma)^2$ and take $\ka>-1/2$. If $u_{0,\ka}$ is the unique kink given by Theorem~\ref{thm:soli}, then there exists $\ka_0\approx3.636$ satisfying 
    \begin{equation*}
       P_{\ka_0}'(0)= \frac{4\ka_0-1}{4\sqrt{|\ka_0|}}\atan\Big(\sqrt{|\ka_0|}\sqrt{\frac{2}{1+2\ka_0}}\Big)-\frac{3}{2}\sqrt{\frac{1+2\ka_0}{2}}=0.
    \end{equation*}
    Moreover, if $\ka<\ka_0$, then the black soliton $u_{0,\ka}$ is orbitally stable.
\end{proposition}
The (orbital) stability of gray and black solitons has been analyzed in other settings: First, note that its weak formulation, equation \eqref{GTWc} formally corresponds to the action critical point equation
 $$E_\ka'(u_{c,\ka})=cP'(u_{c,\ka}).$$
This structure and the existence of a branch of solution of \eqref{GTWc} are key assumptions in the celebrated paper of M.~Grillakis, J.~Shatah, and W.~Strauss (GSS)~\cite{GrillakisShatahStrauss}. The general framework developed there leads to the following conjecture in the case of \eqref{QGP}: Upon spectral assumptions on $E_\ka''-cP''$, the so-called Vakhitov-Kolokolov (VK) criterion holds, that is \begin{align*}
     \text{ $u_{c,\ka}$ is orbitally stable if $\partial_cP(u_{c,\ka})<0$, and unstable if $\partial_cP(u_{c,\ka})>0$.}
 \end{align*}
Now, since \eqref{QGP} is a Hamiltonian equation featuring a non-surjective skew-symmetric operator, establishing instability requires an approximation argument, departing from the standard framework of \cite{GrillakisShatahStrauss}. In the semilinear case $\ka=0$, Z.~Lin~\cite{linbubbles} (2002) conducted the spectral analysis and established the Vakhitov-Kolokolov (VK) criterion for gray solitons. Similarly, in the quasilinear case, the stability criterion was shown for gray solitons by S.~Benzoni-Gavage, R.~Danchin, S.~Descombes, and D.~Jamet~\cite{BenzoniStab} in 2005, and the instability criterion was obtained by C.~Audiard~\cite{audiard2017} in 2017. The VK criterion for Hamiltonian systems with a non-surjective skew-symmetric operator was obtained in a unified manner in 2020~\cite{walshStab}. We also refer to \cite{LeQuiniou2024stability}, where we did a complete study, and computed the VK criterion for the unique branch of gray solitons in the case $h(r)=r$. In the GSS setting, C.A.~Stuart~\cite{stuartstab} (see also \cite{Chiron-stability,Barashenkovcrit}) showed that the stability case amounts to the existence of a Lyapunov functional of the form
\begin{equation}\label{eq:Lyapc}
   \Psi\mapsto E_0(\Psi)-cP(\Psi)- (E_0(u_{c,\ka})-cP(u_{c,\ka}))+K(P(\Psi)-P(u_{c,\ka}))^2,
\end{equation}
for $K>0$ large enough. This Lyapunov functional enables us to obtain explicit bounds for the size of the perturbation of a soliton at time $t>0$ with respect to the size at time $0$. 

It is generally expected that there exists a local branch of stable gray solitons (i.e.\ satisfying $P_\ka'(0)<0$) when $c$ is close to the speed of sound~\cite{Chironexistence1d,berthoumieu2023minimizing}. Moreover, in the general semilinear case $\ka=0$, the asymptotic stability of such gray solitons close to the speed of sound was recently obtained by J.~Berthoumieu~\cite{berthoumieu2025asymptoticstabilitytravellingwaves}.

There is a major obstruction to handling the orbital stability of the black soliton with the GSS framework: One cannot map suitable perturbations of the soliton to a vector space and satisfy the spectral assumptions on $E_\ka''(u_{c,\ka})$ simultaneously. Indeed, on one hand, the Madelung transform, which maps $\boN\boX(\R)$ to the vector space $H^1(\R)\times L^2(\R)$ is ill-defined in the neighborhood of the kink in $\boX(\R)$~\cite{deLGrSm2}. On the other hand, for $\delta=\delta_1+i\delta_2\in\boC_c^\infty(\R;\C)$ we have
\begin{equation*}
    E_0(u_{0,0}+\delta_1+i\delta_2)=E_0(u_{0,0})+\begin{pmatrix}
        L_+\delta_1\\\
        L_-\delta_2
    \end{pmatrix}\cdot\begin{pmatrix}
        \delta_1\\\delta_2
    \end{pmatrix}+o(||\delta||_{H^1}^2),
\end{equation*}
where $L_-\delta_2=\delta_2''+\delta_2f(u_{0,0}^2)$ has no spectral gap. However, the (orbital) stability of the kink has been shown rigorously in some cases, for instance in the Gross--Pitaevskii case $\ka=0$, and $f(\sigma)=1-\sigma$, P.~Gravejat and D.~Smets~\cite{asymptstabblack} exhibited a metric on $\boX(\R)$ tailored to $E_0''(u_{0,0})$ so that the function $E_0(\Psi)-E_0(u_{0,0})$ is a Lyapunov functional. Recently, J.~Holmer, P.~G.~Kevrekidis, and D.~E.~Pelinovsky~\cite{holmer2025orbitalstabilitykinksnls} generalized this approach to the cubic-quintic nonlinearity and for more general semilinear Schr\"odinger equations. Other results were obtained using modified or higher-order energies, for instance stability in $u_{0,0}+H^2(\R)$ was obtained by T.~Gallay and D.~E.~Pelinovsky~\cite{PelinovskyGallayblack} in the Gross--Pitaevskii case, and stability in $u_{0,0}+H^1(\R)$ for the quintic case $f(\sigma)=1-\sigma^2$, was established by M.~\'A.~Alejo and A.~J.~Corcho~\cite{AlejoCorchoQuinticStab}.
Stability for black solitons in a quasilinear setting has received less attention. We are only aware of a spectral stability result for Schr\"odinger equations with intensity-dependent dispersion by D.~E.~Pelinovsky and M.~Plum~\cite{PelinovskyGPBlackInIntensityDependantEq}.  
 
  Existence and stability results were also obtained in higher dimensions in the semilinear case \cite{bethuel,ChironPacherie,marisexistence,Baldelli2024travelingwavesnonlinearschrodinger,Lin3d}. In the quasilinear case, we are only aware of one study by C.~Audiard~\cite{audiard2017} showing the existence of traveling-wave solutions $u_{c,\ka}$ near the speed of sound in dimension two.

The idea of the proof is reminiscent of the arguments introduced by D.~Chiron in \cite{Chiron-stability}
for the stability of the black soliton in a general semilinear setting. Starting from \eqref{eq:Lyapc}, we attempt to define a Lyapunov functional in the neighborhood of the black soliton. Given  $M>|P_\ka'(0)|^{-1}$, we introduce
\begin{equation}\label{def:L}
   S(\Psi)=L(\Psi)-L(u_{0,0}),\quad\text{ where, }\quad L(\Psi)=E_\ka(\Psi)+2Mr_0^4\sin^2\Big(\frac{\boP(\Psi)-r_0^2\pi}{2r_0^2}\Big).
\end{equation}
Here $\boP(\cdot)$ denotes the untwisted momentum, a conservation law of \eqref{QGP} defined on $\boX(\R)$ and taking value in $\R/(2\pi r_0^2\Z)$, given by the expression
\begin{equation*}
    \boP(u)=\lim_{R\to+\infty}\left(\int_{-R}^R\Re(i\partial_xu\bar u) dx+r_0^2\Big[\arg(u(x))\Big]^{x=R}_{x=-R}\right).
\end{equation*}
This quantity coincides mod $2\pi r_0^2$ with the momentum $P(u)$ for nonvanishing functions $u\in\boN\boX(\R)$; we refer to Appendix~\ref{sec:topology} for a summary of the properties of $\boP(\cdot)$. 

Let us sketch the proof of Theorem~\ref{thm:stab}.
In a small neighborhood of the black soliton, we formally have $$\sup_{x\in\R}\Big ||\Psi(x)|-r_0\Big|=r_0-\inf_{x\in\R}|\Psi(x)|=\sup_{x\in\R}\Big||u_{0,\ka}(x)|-r_0\Big|-\inf_{x\in\R}|\Psi(x)|.$$ Thus, for the orbital stability, it is important to track the growth of $\inf_{x\in\R}|\Psi(x)|$.
To corroborate this fact, in Proposition~\ref{prop:lyap}, we are able to show the following inequality
\begin{equation}\label{ineq:lyapintro}
  \inf_{(z,\varphi)\in\R^2}  d_{\boX}(v,u_{0,\ka}(\cdot-z)e^{i\varphi})\leq K(S(\Psi)+K\inf_{x\in\R}|\Psi|)^{\frac14}.
\end{equation}
Then, a major part of our analysis consists of establishing the functional inequality
\begin{equation}\label{ineq:varintro}
    \inf_{x\in\R}|\Psi(x)|^2\leq K S(\Psi),
\end{equation}
for some $K>0$ and all $\Psi\in\boX(\R)$ in a small neighborhood of $u_{0,\ka}.$
To do so, letting $\tilde\xi$ so that $$1+2\ka\sigma h'(\sigma)^2>0,\quad\text{ and, }\quad F(\sigma)>0,\quad\text{ for all }\sigma\in(r_0^2;r_0^2+\tilde\xi],$$ we consider the quantity
        \begin{equation}\label{def:LyapMinIntro}
            \boL_{\min}(\mu)=\inf\{L(u)~|~u\in\boX(\R),\quad\mu^2\leq|u|^2\leq r_0^2+\tilde\xi,\quad\inf_{x\in\R}|u(x)|=\mu \}.
        \end{equation}
        The constraint on the amplitude $|u|^2\leq r_0^2+\tilde\xi$ is nonstandard. It is reminiscent of the fact that, when $\ka<0$ (focusing quasilinear terms), the energy has a negative component. In fact, Proposition~\ref{prop:illvar} shows that removing this constraint, there holds for all $\ka<0$, $$\inf\{L(u)~|~u\in\boX(\R),\inf_{x\in\R}|u(x)|=\mu \}=-\infty,\quad\text{ for all }\mu\geq0.$$
        
        Let $(v_n)$ be a minimizing sequence for $\boL_{\min}$ and denote by $v$ its weak limit (up to extraction). We aim to understand precisely the behavior of $v$.
        Notice that if $\mu>0$, there holds $$L'(v_n)=E'(v_n)-c_nP'(v_n), \quad \text{ with } c_n=Mr_0^2\sin(\pi-P(v_n)/r_0^2).$$
         Up to translation, we show that for $\mu>0$ small enough, there holds $\mu^2<|v|^2\leq r_0^2+\tilde\xi/2$ in $\R\backslash[-R,R]$, for some $R\geq0$. Thus, formally letting $n\to\infty$ in $L'(v_n)$, we get, up to extraction, the Euler--Lagrange equation
        $$E'(v)-cP'(v)=0,\quad\text{ with }c=Mr_0^2\sin(\pi-\limsup_{n\to\infty}P(v_n)).$$
      By quantifying the loss of compactness for $(v_n)$ in terms of $c$ and $R$, we can characterize $v$ whenever $\mu>0$ is small enough. Namely, there hold $R>0$ and $c=c(\mu)$, where $c(\mu)$ is the unique speed in the branch near the black soliton for which $\inf_\R|u_{c(\mu),\ka}|=\mu$, and $v$ satisfies
        \begin{equation}\label{eq:Lminimaintro}
              |v(x)|=\begin{cases}|u_{\mathfrak \mathbf{c}(\mu),\ka}(x+R)|     \text{ for all }x\leq -R,\\
                   \mu, \text{ for all }x\in(-R,R),\\
                   |u_{\mathfrak \mathbf{c}(\mu),\ka}(x -R)| \text{ for all }x\geq R.
               \end{cases}  
        \end{equation}
Using this, it is not difficult to show \eqref{ineq:varintro}.
An interesting byproduct of our variational analysis is that the infimum in $\boL_{\min}(\mu)$ is not attained. This justifies the careful analysis carried out thus far, rather than relying on a concentration-compactness method. 

Combining \eqref{ineq:lyapintro}--\eqref{ineq:varintro} yields, up to taking larger $K>0$,
$$\inf_{(z,\varphi)\in\R^2}  d_{\boX}(\Psi(t),u_{0,\ka}(\cdot-z)e^{i\varphi})\leq K(S(\Psi(t)+S(\Psi(t))^{\frac{1}{2}})^{\frac{1}{4}} \leq KS^{\frac{1}{8}}(\Psi(0)),$$
where we used that $L(\Psi)$ is conserved by the flow of \eqref{QGP}. To estimate the modulation parameters in the infimum above, and recover Theorem~\ref{thm:stab}, we implement the orthogonality argument of J.~L.~Bona and A.~Soyeur~\cite{BonaSoyeur} in Appendix~\ref{sec:infimum}.

Going back to the semilinear case, whenever $P_0'(0)>0$, the kink is linearly unstable~\cite{MenzaGalloLinInstab}, that is $iE_0''(u_{0,0})$ has an eigenvalue with positive real part. In~\cite{Chiron-stability}, it is shown that this implies that the kink is also orbitally unstable. 
Concerning the quasilinear case, A.~de Laire, G.~Dujardin, and S.~Tapia-Mandiola~\cite{dLDujardinTapiaQGP} found numerical evidence confirming that the change of sign of $P_\ka'(0)$ determines the linear and nonlinear stability of black solitons. Also, they exhibit the blowup phenomenon for $\partial_{xx}\Psi$ when the condition $1+2\ka|\Psi|^2h'(|\Psi|^2)>0$ is not met, in the case of focusing semilinear nonlinearity and Gaussian initial conditions.

   \paragraph{Bright solitons}
          In the setting of vanishing condition at infinity ($r_0=0$ in \eqref{eq:nonzero}), a nonlinear solution of interest is the so-called bright soliton, which is a standing wave solution to \eqref{QGP} (and higher dimension generalization) given by $\Psi(x,t)=e^{i\omega t}u_w(x)$. 
          It is difficult to stress how vivid the topic of standing wave solutions is in the quasilinear Schr\"odinger community.  We shall only mention recent studies on the qualitative properties of these solutions in higher spatial dimensions and point the interested reader to the references therein \cite{AlvesWangDualbadKappa,SHEN_JIANG_2026}. Regarding the time dynamics of \eqref{QGP} near bright solitons, the orbital stability of standing wave solutions often comes as a byproduct of the proof of existence by the concentration compactness method (see e.g. \cite{ColinDual2D}). Instability by blow-up was shown for some solitons using Virial type estimates~\cite{ColinJJSquassina,Guoblowup,ShuZhangVirial}. To prove in a unified framework the existence and (in)stability of a local branch of standing wave solutions, the nondegeneracy of the linearized equation near a soliton was shown in two settings: In one dimension by I.~D.~Iliev and K.~P.~Kirchev~\cite{Ilievquasilinstab}, and in arbitrary dimension, but for power laws nonlinearities by F.~Genoud and S.~Rota-Nodari~\cite{GenoudRotaStabQLS}. Some results exhibit stable orbits for compactly supported solutions in the setting of  Schr\"odinger equations with intensity-dependent dispersion~\cite{GermainCompactons,kevrekidis2024stabilitysmoothsolitarywaves}.
          To our knowledge, the asymptotic dynamics around solitons is an open problem in general: To this day, we only have access to a handful of global well-posedness results for~\eqref{QGP} in the case of cubic or quadratic nonlinearities \cite{ifrim2024globalQLS,ifrim2023tataruQLS} and small data, and two scattering results for small data: in 1d for cubic semilinear nonlinearity~\cite{scatteringdBSaut}, and in dimension three or higher for at least quadratic semilinear nonlinearity~\cite{AudiardHaspotGWP}.
            \subsection{Outline of the paper and notations}

  \paragraph{Notations.}
  $\boC_b(\R;\C)$ denotes the space of continuous bounded functions. We work in the complete metric space $(\boX(\R),d_{\boX})$ with $d_{\boX}$ given by \eqref{def:dX}. For $s\in\R$ and $p\in[1,\infty]$, and $I\subset\R$ an open interval, we write $W^{s,p}(I)$, the Sobolev space of complex-valued functions with $W^{s,2}(I)=H^s(I)$.
  For $j\in\N^*$, we denote the (non-standard) homogeneous Sobolev space by $$\dot{H}^j(I)=\{u\in L^1_{\loc}(I):\partial_x^l u\in L^2(I) \text{ for all integers }1\leq l\leq j\},$$
  and
  $\dot{H}^\infty(I)=\bigcap_{j\in\N^*}\dot H^j(I).$ 
  We write the integrand of $E_\ka(\Psi)$  as $e_\ka(\Psi)$ given by 
    $$e_\ka(\Psi)=|\partial_x\Psi(x)|^2+F(|\Psi(x)|^2)+\frac\ka2|\partial_x h(|\Psi(x)|^2)|^2.$$
    In Appendix~\ref{sec:infimum} we denote by $\langle\cdot,\cdot\rangle_{\C}$ the \emph{real} scalar product given by
    $$\langle z_1,z_2\rangle_{\C}=\Re(z_1\bar z_2)=\Re(z_1)\Re(z_2)+\Im(z_1)\Im(z_2), \quad\text{ for all }z_1,z_2\in\C.$$
\paragraph{Outline.}
In Section~\ref{sec:branch}, we show preparatory results on nondegenerate traveling-wave solutions. We show the main result in Section~\ref{sec:StabProof}.  Appendix~\ref{sec:topology} introduces different equivalent distances on $\boX(\R)$ that have been considered in the literature and recalls elementary properties on the energy and momentum. Finally, in Appendix~\ref{sec:infimum}, we obtain the estimates on the modulation parameters in Theorem~\ref{thm:stab}.
   %
 %
 %
 %
 %
 \section{Properties of the local branch of non-degenerate dark solitons}\label{sec:branch}
 In this section, we obtain general properties of non-degenerate traveling-wave solutions $u=u_{c,\ka}\in\boX(\R)$ to \eqref{GTWc}, i.e., satisfying $1+2\ka|u|^2h'(|u|^2)^2>0$ in $\R$. Unless specified, in this section we only assume smoothness of the nonlinearities (i.e. Assumption~\eqref{def:hyp1}).
  In the semilinear case $\ka=0$, the classification of finite energy travelling waves is a classical result~\cite{Chironexistence1d,BerestyckiLionsODE}.
 To verify the existence of a black soliton belonging to a continuous branch of dark solitons, we start by showing the following result for nondegenerate solutions to \eqref{GTWc} in a similar fashion.
    \begin{theorem}\label{thm:soli}
    Let $\ka\in\R$ and $c\in\R.$ Let $c_s=\sqrt{-2f'(r_0)}$ be the speed of sound, we have the following alternative
    \begin{enumerate}
        \item If $|c|>c_s$, then the only nondegenerate solutions to \eqref{GTWc} in $\boX(\R)$ are of the form $u(x)=r_0e^{i\varphi}$ for all $x\in\R$ where $\varphi\in\R.$ 
        \item\label{item:condsoli} If $|c|\leq c_s$, and letting $$\boV_c(\xi)=c^2\xi^2-4(r_0^2+\xi)F(r_0^2+\xi),$$ then the following condition is equivalent to the existence of a nontrivial solution $u_{c,\ka}\in\boX(\R)$ satisfying \eqref{cond:nondeg}:
        
        There exists $\xi_*<0$ (resp. $\xi_*>0$) such that $\boV_c(\xi_*)=0$ ,$\boV_c'(\xi_*)<0$, $\boV_c(\cdot)<0$ in $(\xi_*,0)$ (resp. such that  $\boV_c(\xi_*)=0$, $\boV_c'(\xi_*)>0$ and $\boV_c(\cdot)<0$ in $(0,\xi_*)$),
        and $1+2\ka (r_0^2+\xi)(h'(r_0^2+\xi))^2>0$ in $[\xi_*,0].$
    \end{enumerate}
     In particular, a kink or black soliton $v\in\boX(\R)$ vanishing at $x=0$ exists if and only if $c=0$ \eqref{def:hyp2}--\eqref{def:hyp3} holds.
    In that case, there exists $\varphi\in\R$ such that $v(x)=u_{0,\ka}(x)e^{i\varphi}$ where $u_{0,\ka}$ is the odd function defined using \eqref{eq:implikink}. 
           \end{theorem}
           
  The following result shows that equation \eqref{GTWc} can be recast as two equations for the (real-valued) intensity profile $\eta$, which is key for our classification result, in the spirit of  \cite{dLMar2022,bethuel2008existence,deLaire4}.
  \begin{proposition}
  \label{prop:Geqeta}
		Let $u\in \boX(\R)$ be a nondegenerate solution to \eqref{GTWc}, then $u\in\boC^\infty(\R)$ and $\eta=|u|^2-r_0^2,$ satisfies $\lim_{|x|\to\infty}\eta(x)=0$ and
        \begin{align}
            \label{Geta2}
            2(1+2\ka(r_0^2+\eta)h'(|u|^2)^2)\eta''+2\ka(\eta')^2(h'(|u|^2)^2+2h'(|u|^2)h''(|u|^2))+\boV_c'(\eta)&=0, \quad \text{in }\R,\\
            \label{Geta1}
            (1+2\ka(r_0^2+\eta)h'(|u|^2)^2)(\eta')^2+\boV_c(\eta)&=0, \quad \text{in }\R,
        \end{align}
            with $\boV_{c}(\xi)=c^2\xi^2-4(r_0^2+\xi)F(r_0^2+\xi)$ for all $\xi\in\R.$
		\end{proposition}
		\begin{proof}
      We show that the function $u$ is smooth.  By the Morrey inequality, since $u\in\boX(\R),$ we have $u\in \boC(\R)\cap L^\infty(\R).$ 
        Let $u=u_1+iu_2$, then \eqref{GTWc} can be recast as
        \begin{equation}\label{Geq:Reu}
				u_1''+u_1(f(|u|^2)+\ka h'(|u|^2)(h(|u|^2))'')=-cu_2',\quad \text{in }\R,
			\end{equation}
			\begin{equation}\label{Geq:Imu}
				u_2''+u_2(f(|u|^2)+\ka h'(|u|^2)(h(|u|^2))'')=cu_1',\quad \text{in }\R.
			\end{equation}
         Putting together the second-order terms in \eqref{Geq:Reu}-\eqref{Geq:Imu}, we obtain the ODE system
\begin{align}\label{eq:TWsyst}
    A(u)\begin{pmatrix}
        u_1\\u_2
    \end{pmatrix}''+c\begin{pmatrix}
        u_2\\-u_1
    \end{pmatrix}'+[f(|u|^2)+2\ka  (h'(|u|^2)^2|u'|^2+2 h'(|u|^2)h''(|u|^2)((|u|^2)')^2)]\begin{pmatrix}
        u_1\\u_2
    \end{pmatrix}=0,\end{align}
    where
    \begin{align*}A(u)=\begin{pmatrix}
        1+2\ka u_1^2 h'(|u|^2)^2&2\ka u_1u_2h'(|u|^2)^2\\
        2\ka u_1u_2h'(|u|^2)^2&1+2\ka u_1^2 h'(|u|^2)^2
    \end{pmatrix},\text{ and }\det A(u)=1+2\ka |u|^2h'(|u|^2)^2.
\end{align*}
Since  $\det A(u)=1+2\ka |u|^2h'(|u|^2)^2>0$ by assumption, and using $u\in\boC_b(\R)$, we deduce that $A^{-1}(u)\in L^\infty(\R)$, thus we get $u''\in L^1_{\loc}(\R)$ using \eqref{eq:TWsyst}. The fact that $u\in\boC^\infty(\R)$ follows by a bootstrap argument. 

Turning back to \eqref{Geq:Reu}--\eqref{Geq:Imu}, we show that \eqref{Geta2}--\eqref{Geta1} hold.
			Multiplying \eqref{Geq:Reu} by $-u_2$, and \eqref{Geq:Imu} by $u_1$, and adding these equations, we get
			\begin{align*}
				(u_2'u_1-u_2u_1')'=\frac{c}{2}\eta'.
			\end{align*}
			Since $u'\in L^2(\R)\cap\boC(\R)$, there exists a sequence $(R_n)_{n\in\N}$ such that $\lim_{n\rightarrow \infty}R_n=\infty$ and $u'(R_n)=u_1'(R_n)+iu_2'(R_n)\rightarrow0\text{ as }n\rightarrow \infty.$
			Integrating from $x\in\R$ to $R_n$, we obtain 
			\begin{equation}\label{eq:phase1}
				(u_2'u_1-u_2u_1')(x)=\frac{c}{2}\eta(x)+K_n,\quad\text{ for all }x\in\R,
			\end{equation}
			where $K_n= (u_2'u_1-u_2u_1')(R_n)-\frac{c}{2}\eta(R_n)\rightarrow0$, as $n\rightarrow \infty$.
			Thus equation \eqref{eq:phase1} reads
			\begin{equation}\label{eq:phase2}
				(u_1u_2'-u_1'u_2)=\frac{c}{2}\eta,\quad\text{in }\R.
			\end{equation}
			Besides, multiplying \eqref{Geq:Reu} by $u_1'$ and \eqref{Geq:Imu} by $u_2'$, and adding these equations, we have
			\begin{equation*}
				\frac{1}{2}\Big((u_1')^2+(u_2')^2\Big)'=-\frac{\eta'}{2}(f(r_0^2+\eta)+\ka h'(r_0^2+\eta)(h(r_0^2+\eta))''),
			\end{equation*}
			so integrating this relation from $x$ to $R_n$, and taking the limit as before,  we obtain
			\begin{align}\label{Geq:quadratic}
				|u'|^2&=F(r_0^2+\eta)-\frac{\ka}2(\eta' h'(r_0^2+\eta))^2.
			\end{align}
			In addition, multiplying \eqref{Geq:Reu} by $u_1$, \eqref{Geq:Imu} by $u_2$ and adding these equations, we have
			\begin{align*}
				c(u_1'u_2-u_1u_2')=u_1u_1''+u_2u_2''+(r_0^2+\eta)(f(r_0+\eta)+\ka h'(r_0+\eta)(h(r_0+\eta))'').
			\end{align*}
		We are now in a position to deduce \eqref{Geta2}. Indeed, since
			$2\eta''=4(|u'|^2+u_1u_1''+u_2u_2'')$, using \eqref{eq:phase2}--\eqref{Geq:quadratic}, we get
			\begin{align*}
				2\eta''&=4F(r_0^2+\eta)-2{\ka}(\eta' h'(r_0^2+\eta))^2-2{c^2}\eta-4(r_0^2+\eta)(f(r_0+\eta)+\ka h'(r_0+\eta)(h(r_0+\eta))'').
			\end{align*}
			 Putting terms with first and second-order derivatives in $\eta$ on the left-hand side of \eqref{Geq:quadratic}, we recover \eqref{Geta2}.
			To obtain \eqref{Geta1}, notice  that \begin{align}\label{Geq:eta21}\notag
				\Big((\eta')^2(1+2\ka(r_0^2+\eta)h'(|u|^2)^2\Big)'=&2\eta''\eta'(1+2\ka(r_0^2+\eta)h'(|u|^2)^2)\\&+2\ka(\eta')^3(h'(|u|^2)^2+2h'(|u|^2)h''(|u|^2),
			\end{align}
            so that  multiplying \eqref{Geta2} by $\eta'$, integrating from $x$ to $R_n$ and taking the limit as $n\to\infty$, we finally deduce \eqref{Geta1}.
            \end{proof}
             Working with nondegenerate solutions, we can deduce some additional properties, such as symmetry and monotonicity, using plane phase analysis arguments.
        \begin{proposition}\label{prop:Gpropu}
            Let $\ka\in\R$ and $c\in\R$. Let $u_{c,\ka}\in\boC^2(\R)\cap\boX(\R)$ be a solution to \eqref{GTWc} and write $\eta=|u_{c,\ka}|^2-r_0^2$. If $|c|>c_s$,  then there exists $\varphi\in\R$ such that $u_{c,\ka}(x)=r_0e^{i\varphi}$ for all $x\in\R$. If $|c|\leq c_s$, and $1+2\ka(|u_{c,\ka}|^2)h'(|u_{c,\ka}|)^2>0$, assuming that $\eta$ reaches a nonzero global extremum at $x=0$, then $\eta$ is even and the following statements hold.
            \begin{enumerate}
                \item\label{item:Gpropu1} If $\eta(0)<0$, then $\eta''(0)>0$ and $\eta'(x)>0$ for all $x\in(0,\infty)$.
                \item\label{item:Gpropu2} If $\eta(0)>0$, then $\eta''(0)<0$ $\eta'(x)<0$ for all $x\in(0,\infty)$.
            \end{enumerate}
            Also, if $c\ne0$, then $\eta(x)>-r_0^2$ for all $x\in\R$ so that $\inf_{x\in\R}|u_{c,\ka}(x)|>0$, and there exists $\varphi\in\R$ such that 
            \begin{equation}\label{eq:madelungsol}
                u_{c,\ka}(x)=\sqrt{1-\eta(x)}e^{i(\theta(x)+\varphi)},\quad\text{ with }\theta(x)=\int_{0}^x\frac{c\eta}{2(1-\eta)},
            \end{equation}
            whereas, if $c=0$, then there exists $\varphi\in\R$ such that $u(x)e^{i\varphi}\in\R$ for all $x\in\R.$
        \end{proposition}
        \begin{proof}
        The fact that the nondegenerate solutions are trivial when $\ka\in\R$ and $|c|>c_s$ follows using the same arguments as Theorem~1 in \cite{Chironexistence1d}: Expanding $\boV_c(\cdot)$ near $0$, we find \begin{equation}
            \boV_c(\xi)= \xi^2(c^2-2r_0^2F''(r_0^2))+\mathop{o}(\xi^2),
        \end{equation}
        and, by definition of $c_s$, there holds $\boV_c(\xi)>0$ in a punctured neighborhood of $0$, say $V=B(0,\delta)\backslash\{0\}$, for $\delta>0.$ Up to taking smaller $\delta$, we can also assume that $1+2\ka(r_0^2+\xi)h'(r_0^2+\xi)^2>0$ for all $\xi\in V$. By contradiction, if $u_{c,\ka}\in\boX(\R)$ is a nontrivial solution to \eqref{GTWc}, then there exists $x_0\in\R$ such that $|u_{c,\ka}(x_0)|^2-r_0^2\in V$, for otherwise $|u_{c,\ka}|^2-r_0^2$ must be identically zero, that is, $u_{c,\ka}$ is constant by \eqref{Geq:quadratic}. But then,  $|u_{c,\ka}|^2-r_0^2$ cannot satisfy \eqref{Geta1} since $\boV_c(|u_{c,\ka}(x_0)|^2-r_0^2 )/(1+2\ka|u_{c,\ka}(x_0)|^2h'(|u_{c,\ka}(x_0)|^2)^2)>0$.
        This contradicts Proposition~\ref{prop:Geqeta}, thus $u_{c,\ka}$ must be a trivial solution to \eqref{GTWc}.

            Concerning the symmetry property, since $\eta$ and $\tilde{\eta}(x)\coloneqq \eta(-x)$ satisfy \eqref{Geta2} with initial condition $(\eta(0),0)$, we deduce that $\eta(x)=\eta(-x)$ for all $x\in\R$ by the Cauchy--Lipschitz theorem.
            It is also clear that $\eta''(0)\ne0$ (or equivalently $\boV_c'(\eta(0))\ne0$), for otherwise,
            $(\eta(0),0)$ is an equilibrium point for \eqref{Geta2}. Thus depending, on whether $\eta(0)<0$ or $\eta(0)>0$, we have $\eta''(0)>0$ or $\eta''(0)<0$ respectively.
            
            For the monotonicity, assume that $\eta<0,$ then since $\lim_{x\to\infty}\eta(x)=0$, there exists $x_0>0$ such that $\eta'(x_0)>0$. Statement~\ref{item:Gpropu1} follows because $\eta'$ cannot vanish in $(0,\infty)$. Indeed, assume by contradiction that there exists $x_1>0$ satisfying $\eta'(x_1)=0.$ Recall that \eqref{Geta1}  yields $\boV_c(\xi)\leq0$ for all $\xi\in(\eta(0),0)$, and taking $x=x_1$ in \eqref{Geta1}, we deduce that $\boV_c(\eta(x_1))=0.$ Hence $\boV_c(\xi)$ reaches a local maximum at $\xi=\eta(x_1)$ so that $\boV_c'(\eta(x_1))=0$. This implies that $(\eta(x_1),0)$ is an equilibrium for equation \eqref{Geta2}. We obtain $\eta(x)=\eta(x_1)$ for all $x\in\R$ by the Cauchy--Lipschitz theorem. In particular, taking $x=0$, we get $\eta(x_1)=\eta(0)<0$, thus $\eta$ does not tend to zero at infinity, contradicting our assumptions on $u_{c,\ka}$.  

            Let $c\ne0$, assume by contradiction that $u_{c,\ka}(0)=0$, i.e. $\eta(0)=-r_0^2$. Clearly $-r_0^2$ is a global minimum for $\eta$ and thus $\eta'(0)=0$. But then $\boV_c(\eta(0))=c^2r_0^4>0$ which put in \eqref{Geta1} yields a contradiction.
            Therefore, $\eta(x)>-r_0^2$ for all $x\in\R$. Then, using the Madelung transform $u_{c,\ka}(x)=\sqrt{1-\eta(x)}e^{i\theta(x)}$ and equation \eqref{eq:phase2} we deduce that $\theta'(x)=c\eta(x)/(2-2\eta(x))$
            and \eqref{eq:madelungsol} follows.
            If $c=0$, and, on one hand $u_{0,\ka}(0)\ne0$, then proceeding as above we get $u_{0,\ka}=\sqrt{r_0^2+\eta(x)}e^{i\theta(x)}$ with $\theta'(x)=0$ for all $x\in\R$. On the other hand, if $u_{0,\ka}(0)=0$, then equations \eqref{Geq:Reu}--\eqref{Geq:Imu} have a weaker correlation. Indeed, using \eqref{Geq:quadratic}, we have 
            $|u_{0,\ka}'(0)|=F(0),$
            and thus there exists $\varphi\in\R$ such that $(u_{0,\ka}e^{i\varphi})'(0)=F(0).$ We deduce that $u_{0,\ka}e^{i\varphi}$ and $(\Re (u_{0,\ka}e^{i\varphi}),0)$ satisfy the same ODE problem given by \eqref{Geq:Reu}--\eqref{Geq:Imu}, therefore 
            $u_{0,\ka}e^{i\varphi}\in\R$ by the Cauchy--Lipschitz theorem.
        \end{proof}

        We can now complete the proof of Theorem~\ref{thm:soli}.
        \begin{proof}[Proof of Theorem~\ref{thm:soli}]
        The case $|c|>c_s$ was treated in Proposition~\ref{prop:Gpropu}. Taking $|c|<c_s$, notice that statements~\ref{item:Gpropu1}--\ref{item:Gpropu2} in Proposition~\ref{prop:Gpropu} and equation \eqref{Geta1} yield the necessary conditions on $\boV_c$ in Theorem~\ref{thm:soli}-\ref{item:condsoli}.
           These conditions are identical to the semilinear case $\ka=0$ and are, in fact, sufficient for the existence of nontrivial solutions. For instance, in the case $\xi_*=0$, the condition in Theorem~\ref{thm:soli}-\ref{item:condsoli} becomes  \eqref{def:hyp2} and $\boV_c(0)<0$ in $[-r_0^2;0]$, but we saw that this last point is possible only if $c=0$ and \eqref{def:hyp3} holds. Thus \eqref{def:hyp2}--\eqref{def:hyp3} are necessary conditions for the existence, and, conversely, we can check, using  \eqref{Geta1}, that the odd function defined using \eqref{eq:implikink}
        is a black soliton.
        
           If $\tilde{u}_{0,\ka}$ is another black soliton vanishing at $x=0$, then using Proposition~\ref{prop:Gpropu}, we can find  $\varphi\in\R$ so that 
           $\tilde{u}_{c,\ka}(x)e^{-i\varphi}\in\R$ for all $x\in\R.$ Using this alongside \eqref{Geq:quadratic}, it can be shown that $\tilde u_{0,\ka}(\cdot)e^{i\varphi}$ and $u_{0,\ka}(\cdot)$ satisfy the same ODE problem, given by \eqref{Geq:Reu}. 
         \end{proof}

Now, we show that if $u_{c,\ka}\in\boX(\R)$ is a nondegenerate dark soliton, there exists a branch of nondegenerate finite energy traveling-wave solutions with speeds close to $c$. We also establish exponential decay estimates uniformly in $(c,\ka)$

           \begin{proposition}\label{prop:implisol}
             Assume $c_s>0$ and let $\ka\in\R$ and $c_*\in [0,c_s).$ If $u_{c_*}\coloneqq u_{c_*,\kappa}\in\boX(\R)$ is a nontrivial and nondegenerate solution to $($TW$(c_*,\ka))$ with $|u_{c_*}|$ even, then there exists $\delta>0$ such that, $u_{c_*}$ belongs to a local branch ofnondegenerate traveling-wave solutions $$[c_*,c_*+\delta)\ni c\mapsto u_{c}.$$ This local branch is locally unique in the sense that there exists $R>0$ such that, up to constant phase change, there is no other solution $u\in\boX(\R)$ to \eqref{GTWc} with $|u|$ even, such that $||u_{c_*}-u||_{L^\infty(\R)}\leq R$. Furthermore, if $\mu_c$ is the only nontrivial extremum of $|u_{c}(\cdot)|$ defined by the mapping
             $$[c_*,c_*+\delta)\ni c\mapsto \mu_c=|u_c(0)|,$$
             then $\mu_c$ is smooth and monotonic, so that it admits an inverse denoted by $\mathbf{c}(\mu)$.
            Moreover, the function $(c,x)\mapsto u_{c}(x)$ is smooth in $[c_*,c_*+\delta)$, and the traveling waves satisfy the following exponential decay uniformly in $c\in[c_*,c_*+\delta/2]$: There is $C>0$ such that for all $x\in\R$, we have
             \begin{equation}\label{ineq:decaysolsharp}
                \frac1C\exp\Big(-\sqrt{\frac{c_s^2-c^2}{1+2\ka r_0^2h'(r_0^2)^2}}|x|\Big)\leq||u_c(x)|^2-r_0^2|+|\partial_xu_{c}(x)|\leq C \exp\Big({-\sqrt{\frac{c_s^2-c^2}{1+2\ka r_0^2h'(r_0^2)^2}}|x|}\Big).
            \end{equation}
            More generally, for all $\alpha\in\N^2$, there exists $C_\alpha>0$ depending continuously on $c$ such that
            \begin{equation}\label{ineq:decaysol}
                |\partial_{c,x}^\alpha(|u_{c}(x)|^2-r_0^2)|+|\partial_{c,x}^\alpha\partial_xu_{c}(x)|\leq C_\alpha e^{-C_\alpha|x|},\quad\text{ for all }x\in\R.
            \end{equation}
            Additionally, whenever $c_*\ne0$, the same result holds in $(c_*-\delta,c_*]$, whereas if $c_*=0$, the continuous branch extends to $c\in(-\delta,0]$ by taking 
            $$u_{c}(x)=\overline{u_{|c|}(x)},\quad\text{ for all } x\in\R.$$ In particular, if \eqref{def:hyp1}--\eqref{def:hyp3} holds, then there exists $\delta>0$, such that for all $\ka>\tilde\ka$ given by \eqref{def:katil}, the kink soliton $u_{0,\ka}$ belongs to a unique continuous branch $[0,\delta)\ni c\mapsto u_{c,\ka}$ and the function $\mu_c$ is increasing in $[0,\delta)$. 
           \end{proposition}
         \begin{remark}\begin{enumerate}
             \item Using that every solution in the branch is nondegenerate, by Theorem~\ref{thm:soli}, we find that there exists a closed neighborhood $V$ of $\ka$ such that for any $\ka_2\in V$, we can build a branch $[c,c+\delta)\ni c\mapsto u_{c,\ka_2} $ with the same $\delta>0$. Then,
                we can show that the decay estimates \eqref{ineq:decaysolsharp}--\eqref{ineq:decaysol} also hold for the mapping $(c,\ka_2,x)\mapsto u_{c,\ka_2}(x)$ and all its derivatives.
            \item Using that $\partial_cu_{c,\ka}$ and all its derivatives are exponentially decaying in space, one can show the Hamilton group relation
           \begin{equation}\label{eq:hamiltonGP}
               \frac d{dc}E_\ka(u_{c,\ka})-c\frac{d}{dc}P(u_{c,\ka})=\langle (E_\ka'-cP')(u_{c,\ka}),\partial_cu_{c,\ka})\rangle_{H^{-1}\times H^1}=0,
           \end{equation}
           where we used that the weak formulation of \eqref{GTWc} writes $(E_\ka'-cP')(u_{c,\ka})=0$ in $H^{-1}(\R).$
         \end{enumerate}
           \end{remark}
           
        \begin{proof}
           We introduce the function $\eta_{c_*}(x)\coloneqq|u_{c_*}(x)|^2-r_0^2$ for all $x\in\R$. For the sake of clarity, we divide the proof into two steps.
           \begin{step}\label{step:branch}
                {Existence of a locally unique branch of traveling waves around $u_{c_*}$}
           \end{step}
           Using Theorem~\ref{thm:soli}--\ref{item:condsoli}, we know that $(c,\xi)\mapsto\boV_c(\xi)$ vanishes at $(c_*,\eta_{c_*}(0))$ and $$\partial_\xi\boV_{c_*}(\xi)_{\rvert\xi=\eta_{c_*}(0)}\ne0.$$ Hence, the implicit function theorem implies that there exist $\delta>0, K>0$ and $\xi(\cdot)\in\boC^1([c_*,c_*+\delta))$ such that \begin{equation}\label{impliV}\boV_c(\xi)=0,\quad\text{ for some } (c,\xi)\in(c_*,c_*+\delta)\times(\eta_{c_*}(0)-R,\eta_{c_*}(0)+R)\quad\iff \xi=\xi(c).\end{equation} Moreover, computing $$\partial_c\boV_c(\xi(c))=2c\xi(c)^2+\partial_\xi\boV_c(\xi(c))\partial_c\xi(c)=0,\quad\text{ for all }c\in[c_*,c_*+\delta),$$ we deduce that $\partial_\xi\boV_c(\xi(c))\partial_c\xi(c)<0$ in $(c_*,c_*+\delta)$ so that $\xi(\cdot)$ is monotonous. Without loss of generality, we assume that $\eta_{c_*}(0)<0$ so that, by Theorem~\ref{thm:soli}--\ref{item:condsoli}, there holds $\partial_\xi\boV_c(\xi)_{\rvert\xi=\eta_{c_*}(0)}<0$. Hence, $\xi(\cdot)$ is increasing in $(c_*,c_*+\delta)$ and $\partial_\xi\boV_c(\xi(c))<0$ for all $c\in[c_*,c_*+\delta)$, so that $\boV_c(\cdot)<0$ in $(\xi(c),\xi(c)+\ve)$ for some small $\ve>0$. Then, taking $\delta>0$ small enough so that $c_*+\delta<c_s$, and using the Taylor expansion near $\xi=0$, we get
           \begin{equation}\label{eq:taylorV}
               \boV_c(\xi)=(c^2-c_s^2)\xi^2+\mathop{o}\limits_{\xi\to0}(\xi^3).
           \end{equation}
           Thus, we can assume that $\boV_c(\cdot)<0$ in $(-\ve,0)$. Finally, since $\inf_{\xi\in[\xi(c_*)+\ve,-\ve]}\boV_{c_*}(\xi)<0$, we conclude that for $\delta>0$ small enough there holds \begin{equation}\label{cond:soliimpli}
               \boV_c(\xi(c))=0,\quad\boV'_c(\xi(c))<0,\quad\text{ and }\boV_c(\cdot)<0\text{ in }(\xi(c),0), \quad\text { for all }c\in[c_*,c_*+\delta).
           \end{equation}
           By Theorem~\ref{thm:soli}-\ref{item:condsoli}, the condition~\eqref{cond:soliimpli} imply the existence of a nontrivial nondegenerate travelling wave $u_c\coloneqq u_{c,\ka}\in\boX(\R)$ such that $|u_{c}|$ is even. The case $\eta_{c_*}(0)>0$ follows using similar ideas. For the local uniqueness property of the branch, recall that for any other solution $u\in\boX(\R)$ to \eqref{GTWc} with  $|u|$ even, we have $\boV_c(|u(0)|^2-r_0^2)=0$. If furthermore $||u(0)|^2-|u_{c_*}(0)|^2|<R$, then, using \eqref{impliV}, we recover $|u(0)|^2-r_0^2=\xi(c)$, and we can conclude that $\eta_{c}(x)=|u(x)|^2-r_0^2$ for all $x\in\R$, since they satisfy the same ODE problem given by equation \eqref{Geta2}.
           The properties on $\mu_c=\inf|u_c|=|u_c(0)|$ follow by noticing $\xi(c)=\eta(0).$
    \begin{step}
       Regularity and uniform decay estimates of the mapping  $(c,x)\mapsto u_{c}(x)$
    \end{step} 
             We obtain the smoothness of $(c,x)\mapsto u_{c}(x)$ using standard ODE arguments on equation \eqref{eq:TWsyst}. Using \eqref{eq:madelungsol}, we compute $|\partial_xu_{c}(x)|^2=(\partial_x\eta_{c}^2+c\eta_{c}^2)/({4\eta_{c}+4r_0^2}),$ for all $|x|>1$. Since $\inf_{|x|>1}(\eta_{c}(x)+r_0^2)>0$ uniformly in $(c,\ka),$ we only need to prove the exponential decay for $\eta_{c}$ and all its derivative to recover \eqref{ineq:decaysol}.
             Let $\alpha=(\alpha_1,\alpha_2)\in\N^2$, and assume that $\eta_{c}(0)<0$ so that $\eta_{c}(x)<0$ and $\partial_x\eta_{c}(x)>0$, for all $x\in(0,\infty)$ using Proposition~\ref{prop:Gpropu}. We proceed by induction on $\alpha_1.$
             For $\alpha_1=0,$ we get using \eqref{Geta1} that $\eta_{c}$ satisfies
             \begin{equation}\label{eq:decayeta1}
                 \partial_x\eta_{c}=\boF(\eta_{c}),\text{ in }(0,\infty),\quad \text{ with }\boF(\eta_{c})=\sqrt{\frac{-\boV_c(\eta_{c})}{1+2\ka(\eta_{c}+r_0^2)h'(\eta_{c}+r_0^2)^2}}.
             \end{equation}
             We can see that 
             \begin{equation}\label{eq:factorGronw}
                 \boF(\eta_{c}(x))=|\eta_{c}(x)|C_{c,\ka}+\boR(\eta_{c})\eta_{c}^2,\quad\text{ with }\quad C_{c,\ka}=\sqrt{\frac{c_s^2-c^2}{1+2\ka r_0^2h'(r_0^2)^2}},
             \end{equation}
             where $\boR(\eta_{c})/\eta_{c}(x)$ is bounded in $(0,\infty)$.
             Since $\eta_{c}$ tends to 0 at infinity, we infer that there exists $x_0>0$ such that $C_0\coloneqq\inf_{x\in(x_0,\infty)}-\boF(\eta_{c}(x))/\eta_{c}(x)>3C_{c,\ka}/4.$ Integrating the differential inequality $\partial_x\eta_{c}\geq-C_0\eta_{c}$, we obtain the exponential decay $|\eta_{c}(x)|\leq e^{-C_0x}$ for all $x\in(x_0,\infty)$.
             To refine this inequality, multiply \eqref{eq:decayeta1} by $e^{C_{c,\ka}x}$ and integrate from $x_0$ to $x>0$ to get
             \begin{equation*}
                 e^{C_{c,\ka}x}\eta_{c}(x)-\eta_{c}(x_0)=\int_{x_0}^x e^{C_{c,\ka}x}\boR(\eta_{c})\eta_{c}^2dw.
             \end{equation*}
             Moreover, we can bound the right-hand side 
             \begin{equation*}
                 \left|\int_{x_0}^x e^{C_{c,\ka}w}\boR(\eta_{c})\eta_{c}^2dw\right|\lesssim\sup_{(x_0,\infty)}R(\eta_{c}(\cdot))e^{(C_{c,\ka}-2C_0)x}\lesssim e^{-C_{c,\ka}x/2},
             \end{equation*}
             and the estimate \eqref{ineq:decaysolsharp} readily follows.
             We also obtain the decay for $\partial_x^{\alpha_2}\eta_{c,\ka}$ for every $\alpha_2\in\N$ iteratively, by differentiating equation \eqref{eq:decayeta1}.
             Since $\eta_{c}$ is even, we deduce the exponential decay of $\eta_{c}$ and all its derivative in $(-\infty,0).$
         
             Now assume that there exists $\alpha_1\geq0$ such that,  for all $\alpha_2\in\N$ and for some $C_{\alpha_1,\alpha_2}>0$ depending continuously on $c$, there holds $$|\partial_c^{j}\partial_x^{\alpha_3}\eta_{c}(x)|\leq C_{\alpha_1,\alpha_2}e^{-C_{\alpha_1,\alpha_2}|x|},\text{ for all }x\in\R\quad\text{and for all } 0\leq j\leq\alpha_1. $$
We show that this inequality is still satisfied for $j=\alpha_1+1$. Differentiating \eqref{eq:decayeta1} $\alpha_1+1$ times with respect to $c$, we deduce that $\partial_c^{\alpha_1+1}\eta_{c}(x)$ satisfies for all  $x\in(0,\infty)$ 
\begin{equation}\label{eq:flowderc}
    \partial_x\partial_{c}^{\alpha_1+1}\eta_{c}=\frac{d^{\alpha_1+1}}{(dc)^{\alpha_1+1}}\boF(\eta_{c})=\boF'(\eta_{c}(x))\partial_c^{\alpha_1+1}\eta_{c}+\boG(\eta_{c},\partial_c\eta_{c},\dots,\partial_c^{\alpha_1}\eta_{c}),
\end{equation}
where $\boG(\cdot)$ contains only lower order terms so that, by induction hypothesis, there exists $C_{\alpha_1}>0$ such that $|\boG(\eta_{c},\dots,\partial_c^{\alpha_1}\eta_{c})(x)|\leq C_{\alpha_1}e^{-C_{\alpha_1}|x|},$ for all $x\in\R.$
We compute \begin{equation*}
    \boF'(\xi)=\frac{-\boV_c'(\xi)}{2\sqrt{-\boV_c(\xi)(1+2\ka(r_0^2+\eta)h'(r_0^2+\eta)^2)}}-\frac{\sqrt{-\boV_{c}(\xi)}(1+2\ka(r_0+\xi)h'(r_0^2+\xi)^2)'}{2\sqrt{(1+2\ka(r_0^2+\eta)h'(r_0^2+\eta)^2)^3}}.
\end{equation*}
Then, using \eqref{eq:taylorV}, we obtain $\boV_c'(\xi)=2(c^2-c_s^2)\xi+\mathop{o}\limits_{\xi\to0}(\xi^2)$, so that $$\lim_{\xi\to0^-}\boF'(\xi) = -\sqrt{(c_s^2-c^2)/(1+2\ka r_0^2 h'(r_0^2)^2)} <0.$$
We deduce that there exists $R>0$ depending continuously on $(c,\ka)$ such that 
$\inf_{|x|>R}-\boF'(\eta_{c}(x))>0$. Let $\boF_1(x)=\int_R^x-\boF'(\eta_{c}(\sigma)) dr$ for all $x\geq R$, and multiply \eqref{eq:flowderc} by $e^{\boF_1(x)}$ to obtain
\begin{equation*}
    \partial_x(e^{\boF_1(x)}\partial_c^{\alpha_1+1}\eta_{c}(x))=e^{\boF_1(x)}\boG(\eta_{c}(x),\dots,\partial_c^{\alpha_1}\eta_{c}(x)).
\end{equation*}
Integrating in space between $R$ and $x\geq R$, then multiplying the resulting equation by $e^{-\boF_1(x)}$ yields
\begin{equation}\label{eq:flowderc2}
    |\partial_c^{\alpha_1+1}\eta_{c}(x)|=\left|e^{-\boF_1(x)}\eta_{c}(R)+\int_R^xe^{\boF_1(w)-\boF_1(x)}\boG(\eta_{c}(w),\dots,\partial_{c}^{\alpha_1}\eta_{c}(w))dw\right|. 
\end{equation} Taking $C<\min(\inf_{|x|>R}-\boF'(\eta_{c}(x)),C_{\alpha_1})$,
we get $\exp({\boF_1(w)-\boF_1(x)})=\exp({-\int_w^x-\boF(\sigma) dr})\geq \exp({-C(x-w)})$
and we can deduce the decay of $\partial_c^{\alpha_1+1}\eta_{c}$ from \eqref{eq:flowderc2}, using the decay of $\boG$ and the triangular inequality.
We obtain similar estimates for the space derivatives $\partial_x^{\alpha_2}\partial_c^{\alpha_1+1}\eta_{c,\ka}$ in $(0,\infty)$ for all $\alpha_2\in\N$ by a boostrap procedure.
        \end{proof}
        From now on, we assume \eqref{def:hyp1}--\eqref{def:hyp3} holds and we focus our attention on the branch near the kink solution. The following result gives a closed formula for $P_\ka'(0)=\partial_cP(u_{c,\ka})_{\rvert c=0}$. In the course of the proof, we also obtain a useful expansion of 
$\mu_c=\inf_{x\in\R}|u_{c,\ka}|$ as $c\to0$.
          \begin{proposition}\label{prop:moment0}
             Assume that \eqref{def:hyp1}--\eqref{def:hyp3} hold.  Let $\ka>\tilde\ka$ given by \eqref{def:katil} and $c\mapsto u_{c,\ka}$ be the continuous branch of solitons defined near $c=0$ given by Proposition~\ref{prop:implisol}.
               Then there holds \begin{equation}\label{eq:TaylorMu}
                   \inf_{x\in\R}|u_{c,\ka}|=\mu_c=cr_0^2/\sqrt{4F(0)}+\mathop{O}_{c\to0}(c^2).
               \end{equation}
               Moreover, we have $P(u_{c,\ka})\to r_0^2\pi$ as $c\to0$ and the continuation of $c\mapsto P(u_{c,\ka})$ is $\boC^1$ at $c=0$ with 
            \begin{equation}\label{eq:dercp}
                \frac{dP(u_{c,\ka})}{dc}_{\rvert c=0}=-\frac{8r_0^3}{3\sqrt{F(0)}}+\int_{0}^{r_0^2}\frac{(r^2-r_0^2)^2}{r^{2}}\left(\sqrt{\frac{1+2\ka r^2 h'(r^2)^2}{F(r^2)}}-\frac{1}{\sqrt{F(0)}}\right)dr.
            \end{equation}
           \end{proposition}
        \begin{proof}
             Let us consider the continuous branch around the kink $[0,\delta)\ni c\mapsto u_{c,\ka}$ given by Step~\ref{step:branch}.
             By Proposition~\ref{prop:Gpropu}, we know that $\inf_{x\in\R}|u_{c,\ka}(x)|>0$ whenever $c>0$, thus \eqref{eq:madelungsol} holds and
             \begin{equation}\label{eq:formulaPsol}
                 P(u_{c,\ka})=c\int_\R \frac{\eta_{c,\ka}^2}{2(\eta_{c,\ka}-r_0^2)}dx=c\int_0^\infty \frac{\eta_{c,\ka}^2}{\eta_{c,\ka}-r_0^2}\sqrt\frac{1+2\ka(r_0^2+\eta_{c,\ka})h'(r_0^2+\eta_{c,\ka})^2}{-\boV_c(\eta_{c,\ka})}\eta_{c,\ka}'dx,
             \end{equation} where we used that $\eta_{c,\ka}$ is even and satisfies \eqref{Geta1}.
             We make two observations: firstly the term in $(-\boV_c(\eta_{c,\ka}(x)))^{-1/2}$ has an integrable singularity at $x=0$, also, as $c\to0,$ there holds $\xi(c)\to-r_0^2$, hence the term $(\eta_{c,\ka}-r_0^2)^{-1}$ becomes singular. This motivates the following decomposition, performing first the change of variable $w=\eta_{c,\ka}$, and $dw=\eta_{c,\ka}'(x)dx$ in \eqref{eq:formulaPsol}.
             \begin{align}\label{eq:taylorP1}
                 \frac{P(u_{c,\ka})}{c}=&\int_{\xi(c)}^0\frac{w^2}{w+r_0^2}\sqrt\frac{1+2\ka(\xi(c)+r_0^2)h'(\xi(c)+r_0^2)^2}{-\boV_c'(\xi_c)(w-\xi(c))}dw\\&+\int_{\xi(c)}^0\frac{w^2}{w+r_0^2}\left(\sqrt\frac{{{1+2\ka(w+r_0^2)h'(w+r_0^2)^2}}}{{-\boV_c(w)}}-\sqrt\frac{1+2\ka(\xi(c)+r_0^2)h'(\xi(c)+r_0^2)^2}{{\boV_c'(\xi(c))(\xi(c)-w)}}\right)dw.\label{eq:taylorP2}
             \end{align} We can deal with \eqref{eq:taylorP2} using Lebesgue's dominated convergence theorem: By changing the variable $w=\xi(c)t$, we show that the integral in \eqref{eq:taylorP2} is equal to
            \begin{align}\notag
               &\int_1^0\frac{\xi(c)^3t^2}{r_0^2+t\xi(c)}\left(\sqrt\frac{{1+2\ka(t\xi(c)+r_0^2)h'(t\xi(c)+r_0^2)^2}}{{-\boV_c(t\xi(c))}}-\sqrt\frac{1+2\ka(\xi(c)+r_0^2)h'(\xi(c)+r_0^2)^2}{{\xi(c)\boV_c'(t\xi(c))(1-t)}}\right)dt\\\notag
               &=(-r_0)^3\int_1^0\frac{t^2}{r_0^2(1-t)}\left(\sqrt\frac{{1+2\ka r_0^2(1-t)h'(r_0^2-tr_0^2)^2}}{{-\boV_0(-tr_0^2)}}-\frac{1}{\sqrt{4r_0^2F(0)(1-t)}}\right)dt+\mathop{o}\limits_{c\to0}(1),\\
               &=\int_{-r_0^2}^0\frac{w^2}{w-r_0^2}\left(\frac{\sqrt{{1+2\ka(r_0^2+w)h'(r_0^2+w)^2}}}{\sqrt{-\boV_0(w)}}-\frac{1}{\sqrt{4F(0)(r_0^2+w)}}\right)dw+\mathop{o}\limits_{c\to0}(1).\label{eq:limtaylorP2}
           \end{align}
          Indeed, letting
          $$R\coloneqq\xi(c)\boV_c'(t\xi(c))(t-1)\Big((1+2\ka(\xi(c)+r_0^2)h'(\xi(c)+r_0^2)^2)^\frac{1}{2}-(1+2\ka(\xi(c)+r_0^2)h'(t\xi(c)+r_0^2)^2)^{\frac12}\Big),$$ we get $R\lesssim(1-t)^2$, and, $ \xi(c)\geq-r_0^2$, so that $r_0^2-t\xi(c)\geq r_0^2(1-t)$. Using this and invoking the Taylor expansion of $\boV_c(t\xi(c))$ near $t=1$, we obtain the estimate 
        \begin{align*}
           &\left| \frac{\xi(c)^3t^2}{r_0^2+t\xi(c)}\left(\sqrt\frac{1+2\ka(t\xi(c)+r_0^2)h'(t\xi(c)+r_0^2)^2}{-\boV_c(t\xi(c))}-\sqrt\frac{1+2\ka(\xi(c)+r_0^2)h'(\xi(c)+r_0^2)^2}{{\xi(c)\boV_c'(t\xi(c))(1-t)}}\right)\right|\\
            &\lesssim \left|\frac{(1+2\ka(\xi(c)+r_0^2)h'(\xi(c)+r_0^2)^2)^{\frac12}(\boV_c(t\xi(c))-\xi(c)\boV_c'(t\xi(c))(t-1))+R}{(1-t)\sqrt{-\boV_c(t\xi(c))}\sqrt{\xi(c)\boV_c'(t\xi(c))(1-t)}(\sqrt{\xi(c)\boV_c'(t\xi(c))(1-t)}+\sqrt{-\boV_c(t\xi(c))})}\right|,\\
            &\lesssim \frac{(t-1)^2}{(1-t)\sqrt{1-t}\sqrt{1-t}\sqrt{1-t}}=\frac{1}{\sqrt{1-t}}\in L^1((0,1),
        \end{align*} which yields \eqref{eq:limtaylorP2} by Lebesgue's dominated convergence theorem.
           Besides, the integral in \eqref{eq:taylorP1} can be computed. Letting $w=\xi(c)+(\xi(c)+r_0^2)t^2$, $dw= 2(\xi(c)+r_0^2)tdt$ with $t>0$, and using $$\frac{(\xi(c)+(\xi(c)+r_0^2)t^2)^2}{1+t^2}=\frac{r_0^4}{1+t^2}+(\xi(c)-r_0^2)(\xi(c)+r_0^2)+(\xi(c)+r_0^2)^2t^2,$$
            we deduce that the integral in \eqref{eq:taylorP1} is equal to 
           \begin{align}\notag
              &\frac{2}{\sqrt {\xi(c)+r_0^2}} \int_{0}^{\sqrt{-\xi(c)/(\xi(c)+r_0^2)}}\frac{(\xi(c)+(\xi(c)+r_0^2)t^2)^2}{1+t^2}\sqrt\frac{{1+2\ka(\xi(c)+r_0^2)h'(\xi(c)+r_0^2)^2}}{{-\boV_c'(\xi(c))}}dt\\\notag
              =&\frac{2}{\sqrt {\xi(c)+r_0^2}}\sqrt\frac{{1+2\ka(\xi(c)+r_0^2)h'(\xi(c)+r_0^2)^2}}{{-\boV_c'(\xi(c))}}\Big(r_0^4\atan\sqrt{\frac{-\xi(c)}{\xi(c)+r_0^2}}\\\notag
              &+(\xi(c)-r_0^2)\sqrt{-\xi(c)(\xi(c)+r_0^2)}+\frac{\sqrt{-\xi(c)^3(\xi(c)+r_0^2)}}{3}\Big).
           \end{align}
           Using the identity $\atan(y)=\pi/2-\atan(1/y)$ this expression yield
            \begin{align}\label{eq:atanmoment}
              \notag
              &\frac{2}{\sqrt {\xi(c)+r_0^2}}\sqrt\frac{{1+2\ka(\xi(c)+r_0^2)h'(\xi(c)+r_0^2)^2}}{{-\boV_c'(\xi(c))}}\Big(r_0^4\frac\pi2-\sqrt{{\xi(c)+r_0^2}}(r_0^4{(-\xi(c))}^{-1/2}\\
              &-(\xi(c)-r_0^2)(-\xi(c))^{1/2}-(-\xi(c))^{3/2}/3)+\mathop{o_{c\to0} }(\xi(c)+r_0^2)\Big).
           \end{align}
           Since $\xi(c)\to -r_0^2$, using condition \eqref{def:hyp3}, we deduce that $F(\xi(c)+r_0^2)\ne0$ for all $c\in[0,\delta)$, so that equation $\boV_c(\xi_c)=0$ writes $\xi(c)=G(\xi(c))\coloneqq-r_0^2+c^2(\xi(c))^2/(4F(\xi(c)+r_0^2))$. We have $$G(\xi)=-r_0^2+c^2\frac{r_0^4}{4F(0)}-c^2\frac{2r_0^2F(0)+r_0^4F'(0)}{(2F(0))^2}(\xi+r_0^2)+\mathop{o(c^2(\xi+r_0^2))},\quad\text{ as }\xi\to-r_0^2,$$  and $\xi(c)=-r_0^2+\mathop{o}_{c\to0}(1)$. Combining these two expansions, we obtain the expansion
           \begin{align}\notag
               \xi(c)&=G(G(\xi(c)),
               \\\notag&=-r_0^2+c^2\frac{r_0^4}{4F(0)}-c^2\frac{2r_0^2F(0)+r_0^4F'(0)}{4(F(0))^2}(G(\xi(c))+r_0^2)+c^2(G(\xi_c)+r_0^2)\mathop{o}\limits_{c\to0}(1),\\
               &=-r_0^2+c^2\frac{r_0^4}{4F(0)}-c^4r_0^6\frac{2F(0)+r_0^2F'(0)}{4(F(0))^3}+\mathop{o}\limits_{c\to0}(c^4).\notag
           \end{align}
           In particular, we get the expansion of $\inf_{x\in\R}|u_{c,\ka}|$ in \eqref{eq:TaylorMu} and \begin{align*}
               \frac{2}{\sqrt {\xi(c)+r_0^2}}\sqrt\frac{{1+2\ka(\xi(c)+r_0^2)h'(\xi(c)+r_0^2)^2}}{{-\boV_c'(\xi(c))}}&=\Big(2\frac{\sqrt{4F'(0)}}{cr_0^2}+\mathop{o}\limits_{c\to0}(1)\Big)\Big(\frac{1}{\sqrt{4F'(0)}}+\mathop{O}\limits_{c\to0}(c^2)\Big).\\
               &=\frac{2}{cr_0^2}+\mathop{o}\limits_{c\to0}(1).
           \end{align*}
           Using this and \eqref{eq:atanmoment} we deduce that the integral in \eqref{eq:taylorP1} is equal to
           \begin{equation}\label{eq:intmomexpli}
              \Big(\frac{2}{cr_0^2}+\mathop{o}\limits_{c\to0}(1)\Big)\Big(r_0^4\frac\pi2-\frac{cr_0^2}{3\sqrt{F(0)}}8r_0^3+\mathop{o}\limits_{c\to0}(c)\Big)=\frac{r_0^2\pi}{c}-\frac{8r_0^3}{3\sqrt{F(0)}}+\mathop{o}\limits_{c\to0}(1).
           \end{equation}
           Changing variables $r^2=w+r_0^2$ in \eqref{eq:limtaylorP2} and adding \eqref{eq:intmomexpli}, we obtain the desired result
           \begin{equation*}
               P(u_{c,\ka})=r_0^2\pi+c\Big(-\frac{8r_0^3}{3\sqrt{F(0)}}+\int_{0}^{r_0^2}\frac{(r^2-r_0^2)^2}{r^{2}}\left(\sqrt{\frac{1+2\ka r^2 h'(r^2)^2}{F(r^2)}}-\frac{1}{\sqrt{F(0)}}\right)dr\Big)+\mathop{o}\limits_{c\to0}(c).
           \end{equation*}

          To show that the mapping $c\mapsto P(u_{c,\ka})$ is $\boC^1$ at $c=0$, recall that for $c>0$, there holds 
          $$\partial_cP(u_{c,\ka})=P'(u_{c,\ka})\cdot\partial_cu_{c,\ka}=\int_\R \Re(i\partial_xu_{c,\ka}\partial_c\bar u_{c,\ka})dx.$$ By Proposition~\ref{prop:implisol}, $\partial_{c}^2P(u_{c,\ka})$ is bounded in $(0,\delta)$, and then, the mean value theorem provides the continuity of $\partial_cP(u_{c,\ka})$ at $c=0$.
          
           \end{proof}
           \section{Variational analysis and stability}\label{sec:StabProof}
        Suppose that \eqref{def:hyp1}--\eqref{def:hyp3} holds so that we can find $\tilde\xi>0$ satisfying $$F(\sigma)>0,\quad\text{ and, }\quad1+2\ka\sigma h'(\sigma)^2>0,\quad\text{ for all }\sigma\in(r_0^2;r_0^2+\tilde\xi].$$
        Letting $\ka>\tilde\ka$, we consider the minimization problem, for $\mu\geq0$
        \begin{equation}\label{def:LyapMin}
            \boL_{\min}(\mu)=\inf\{L(u)~|~u\in\boX(\R),\quad\mu^2\leq|u|^2\leq r_0^2+\tilde\xi,\quad\inf_{x\in\R}|u(x)|=\mu \},
        \end{equation}
        where $L(\cdot)$ is given as in \eqref{def:L} by
        $$L(u)=E_\ka(u)+2Mr_0^4\sin^2\Big(\frac{\boP(u)-r_0^2\pi}{2r_0^2}\Big).$$
        Before analyzing this problem in detail, we recall some useful energy estimates, using the assumptions on $f$ and $h$ in \eqref{def:hyp1}--\eqref{def:hyp3}. 
        \begin{lemma}\label{lem:pointwiseEnergy}
        Under Assumptions~\eqref{def:hyp1}--\eqref{def:hyp3},  if $u\in\boX(\R)$ satisfies $|u|^2\leq r_0^2+\tilde{\xi}$, then we have the pointwise estimate
            \begin{equation}\label{ineq:pointwiseEnergy}
                |\partial_xu|^2+F(|u|^2)+\frac{\ka}{2}(\partial_xh(|u|^2))^2\gtrsim |\partial_xu|^2+(|u|^2-r_0^2)^2,\text{a.e. in }x\in\R.
                 \end{equation}
            In particular, if $\inf_{x\in\R}|u(x)|=\mu>0,$ then 
            \begin{equation}\label{ineq:momentE}
                |P(u)|\lesssim\mu^{-1} {E_\ka(u)}.
            \end{equation} Moreover, letting $\eta=|u|^2-r_0^2$, we have $||\eta||^2_{H^1(\R)}\lesssim E_\ka(u)$.
        \end{lemma}
        \begin{proof}
            Using $F(r_0^2)=F'(r_0^2)=0$ and $F''(r_0^2)=c_s^2/4>0$, we deduce that $C=\inf_{r\in(0,r_0^2+\tilde{\xi}]}(F(r)/(r-r_0^2)^2)$ is positive  thus \begin{equation}\label{ineq:pointwiseEnergy0}
            F(|u|^2)\geq C(|u|^2-r_0^2)^2.
            \end{equation}
            To bound the other terms in \eqref{ineq:pointwiseEnergy}, we split the proof between the case $\ka\geq0$ and $\tilde\ka<\ka<0$.
            Let $\tilde C=\inf_{r\in[0,r_0^2+\tilde{\xi}]}(1+2\ka r^2h'(r)^2) $, so that $\tilde{C}>0$ using condition \eqref{def:hyp2}.
            If $\ka\geq0$ then $\tilde{C}=1$ and the inequality $|\partial_xu|^2+\ka(\partial_xh(|u|^2))^2)/2\geq \tilde{C}|\partial_xu|^2$ follows immediately.
            Besides, if $\ka<0,$ then
            \begin{equation*}
                |\partial_xu|^2+\frac{\ka}2(\partial_xh(|u|^2))^2=|\partial_xu|^2-2|\ka|(|u|(\partial_x|u|)h'(|u|^2))^2.
            \end{equation*}
            Since $-2|\ka|(\partial_x|u|)^2\geq-2|\ka||\partial_xu|^2 $ a.e. in $x\in\R$, we deduce the desired inequality 
            \begin{equation}\label{ineq:pointwiseEnergy1}
    |\partial_xu|^2+\frac{\ka}2(\partial_xh(|u|^2)^2\geq\tilde{C}|\partial_xu|^2      ,\quad\text{ a.e. in }x\in\R,   \end{equation}
                and \eqref{ineq:pointwiseEnergy} follows by summing \eqref{ineq:pointwiseEnergy0} and \eqref{ineq:pointwiseEnergy1}. 
               We obtain \eqref{ineq:momentE} by integrating the inequality \begin{equation*}
               \frac{  |\partial_xu\bar u|}{|u|^2} |\eta|  \leq\frac{1}{2\mu}(|\partial_xu|^2+\eta^2).
               \end{equation*}
                The fact that $||\eta||^2_{H^1(\R)}\lesssim E_\ka(u)$ is a consequence of the inequalities $(\eta')^2\leq|u|^2|\partial_xu|^2\leq (r_0^2+\tilde{\xi})|\partial_xu|^2$ and \eqref{ineq:pointwiseEnergy}.
        \end{proof}
        We deduce a compactness result for sequences of functions with uniformly bounded energy.
        \begin{corollary}\label{coro:wlim}
            Let $(v_n)\subset\boX(\R)$ satisfying $|v_n(x)|\leq r_0+\tilde\xi$ for all $x\in\R$ so that $E_\ka(v_n)\geq0$ for all $n\in\N$. If $E_\ka(v_n)\leq E$ for some $E\geq0$ and for all $n\in\N$, then there exists $v\in\boX(\R)$ such that, up to a subsequence, we have the following convergence for $(v_n):$
            \begin{align*}
                 |v_n|^2-r_0^2&\wto|v|^2-r_0^2,\text{ in } H^1(\R), \quad
                 \partial_xv_n\wto\partial_xv,\text{ in } L^2(\R),\\
                \text{and}\quad v_n&\to v, \text{ uniformly in $(-R,R)$, for every $R>0$.}
             \end{align*}
        \end{corollary}
        
        In the case $\mu=0$, we recover this result on $\boL_{\min}(0)$, that is well known in the Gross--Pitaevskii case $\ka=0$, $F(r)=(r-1)^2/2$.
        \begin{proposition}\label{prop:minkink}
        There holds
            \begin{equation}\label{Emin0}
                \inf\{E_\ka(u)~|~u\in\boX(\R),0\leq|u|^2\leq r_0^2+\tilde{\xi},\inf_{x\in\R}|u(x)|=0\}=E_\ka(u_{0,\ka}),
            \end{equation}
            with $E_\ka(u_{0,\ka})=4\int_0^{r_0}\sqrt{F(r^2)(1+2\ka(r h'(r^2))^2}dr.$
            In particular, $\boL_{\min}(0)=E_\ka(u_{0,\ka})$ and if $v\in\boX(\R)$ is another minimizer for $\boL_{\min}(0)$, then there exists $(x_0,\varphi)\in\R^2$ such that $v(\cdot)=u_{0,\ka}(\cdot-x_0)e^{i\varphi}.$ 
                \end{proposition}
        \begin{proof}
              On one hand, we have $\boL_{\min}(0)\leq L(u_{0,\ka})=E_\ka(u_{0,\ka})$. Then, using that the kink is odd, real-valued, and satisfies 
          $     (\partial_xu_{0,\ka})^2=F(u_{0,\ka}^2)-2\ka (u_{0,\ka}h'(u_{0,\ka}^2)\partial_xu_{0,\ka})^2,\text{ in }(0,\infty)$, we compute
            \begin{equation*}
                E_\ka(u_{0,\ka})=4\int_0^\infty F(u_{0,\ka}^2)dx=4\int_0^{r_0}\sqrt{F( r^2)(1+2\ka r^2 h'(r^2)^2)}dr.
            \end{equation*}
            On the other hand, let $v\in\boX(\R)$ be such that $|v|^2\leq r_0^2+\tilde{\xi}$ and $v(x_0)=0$. Using $(\partial_xu)^2\geq(\partial_x|u|)^2$ a.e. and letting $G(r)=\int_{r_0}^r\sqrt{F(r^2)(1+2\ka r^2 h'(r^2)^2)}dr$, we deduce
            \begin{align}\label{ineq:emin0}\notag
\int_{0}^\infty e_\ka(v)dx&\geq\int_0^\infty((1+2\ka|v|^2h'(|v|^2)^2(\partial_x|v|)^2+F(|v|^2))dx\\&\geq2\int_0^\infty\left|\sqrt{F(|v|^2)(1+2\ka|v|^2h'(|v|^2)^2)}\partial_x|v|\right|dx\notag
                \geq \int_0^\infty|\partial_xG(|v|)|dx\\
                &\geq\left|\int_0^\infty\partial_xG(|v|)dx\right|=2\int_{0}^{r_0}\sqrt{F(r^2)(1+2\ka r^2 h'(r^2)^2)}dr.
            \end{align}
            Performing similar computations in $(-\infty,0)$ shows that the kink has the least energy among functions in $\boX(\R)$ that vanish at some point, and that $\boL_{\min}(0)=E_\ka(u_{0,\ka}),$ using $L(v)\geq E_\ka(v)$. Now, let $v\in\boX(\R)$ be a minimizer for $\boL_{\min}(0)$ so that $v(x_0)=0$ for some $x_0\in\R$, and \eqref{ineq:emin0} is an equality. We deduce that $|\partial_x(|v|)|=|\partial_xv|$. Moreover, equality in $$2|\sqrt{F(|v|^2)(1+2\ka|v|^2h'(|v|^2)^2)}\partial_x|v||\leq(1+2\ka|v|^2h'(|v|^2)^2)(\partial_x|v|)^2+F(|v|^2)),$$ implies $ |1+2\ka|v|^2h'(|v|^2)^2|\times(\partial_x|v|)^2=F(|v|^2)$, hence $|v|$ satisfies the first order autonomous ODE \eqref{Geq:quadratic}. Thus, by integration, we get $|v(x)|=|u_{0,\ka}(x-x_0)|$ in $\R$. Using this and $|\partial_x(|v|)|=|\partial_xv|$, we deduce that the phase of $v$ is constant in $(-\infty,x_0)$ and in $(x_0,\infty)$. Thus, $v(x)=e^{i\varphi_\pm}u_{0,\ka}(x-x_0)$ for $x\geq\pm x_0$. Since $L(v)=E_\ka(v)$, we have $\boP(v)/r_0^2\equiv\pi\mod 2\pi $ with $\boP(v)=r_0^2(\varphi_+-\varphi_-)\mod2r_0^2\pi$, we conclude that $v=e^{-i\varphi_+}u_{0,\ka}(x-x_0).$
        \end{proof}
        In what follows, we describe the behavior of $\boL_{\min}(\mu)$ for small $\mu$.        In particular, we establish in Proposition~\ref{prop:limformula3} that if $(v_n)\subset\boX$ is a minimizing sequence for $\boL_{\min}(\mu)$, then there exist $(x_0,\varphi,C)\in\R^3$ and $R>0$ such that $v_n$ converges to $v$ in the sense of Corollary~\ref{coro:wlim} where $v\in\boX(\R)$ satisfies
           \begin{equation}\label{eq:Lminima}
               v(x-x_0)=\begin{cases}e^{i\varphi}u_{\mathfrak \mathbf{c}(\mu),\ka}(x+R) \text{ for all }x\leq -R,\\
                   \mu e^{i\varphi+iC(x+R)}, \text{ for all }x\in(-R,R),\\
                   e^{i\varphi+2iCR}u_{\mathfrak \mathbf{c}(\mu),\ka}(x -R) \text{ for all }x\geq R.
               \end{cases}               
           \end{equation}
            and $\mathfrak{\mathbf{c}(\mu)}$ is given by Proposition~\ref{prop:implisol}.
        \begin{proposition}\label{prop:Lmin}Let $P_\ka'(0)\coloneqq\partial_cP(u_{c,\ka})_{\rvert c=0}$ be given by Proposition~\ref{prop:implisol}. If $P_\ka'(0)<0$ and $M>|P_\ka'(0)|^{-1},$ then
           there exist $K>0$ and a small $\mu_*>0$ such that, for any $0<\mu<\mu_*$,
           \begin{equation}\label{ineq:LL0}
               \boL_{\min}(\mu)\geq E_\ka(u_{0,\ka})+\frac{\mu^2}{K}.
           \end{equation}
         \end{proposition}
           We divide the proof of Proposition~\ref{prop:Lmin} into several subresults.
        Let $u_{0,\ka}\in\boX(\R)$ be the kink solution given by Theorem~\ref{thm:soli}. Consider, as in Proposition~\ref{prop:implisol}, the locally unique branch of traveling waves $[0,\delta)\ni c\mapsto u_{c,\ka}$ defined around the kink, and the increasing function $[0,\delta)\ni c\mapsto{\mu_c}\coloneqq\inf_{x\in\R}|u_{c,\ka}(x)|$.  Let $0<\mu<\mu_*$ where $\mu_*<\mu_{\delta/2}$ is fixed but determined later.
        The travelling wave of speed $\mathbf{c}(\mu)$ is a comparison map for $\boL_{\min}(\mu)$. Integrating the Hamilton group relation~\eqref{eq:hamiltonGP} between $c=\mathbf{c}(\mu)$ and $0$, and using Proposition~\ref{prop:moment0}, we deduce that 
           \begin{equation}\label{eq:taylorEK}
           E_\ka(u_{c,\ka})=E_\ka(u_{0,\ka})+\frac{c^2}{2}\frac{d}{dc}P(u_{c,\ka})_{\rvert c=0}+\mathop{o}\limits_{c\to0}(c^2).
           \end{equation}
        Thus
        \begin{align}\label{eq:TaylorLc}
            L(u_{c,\ka})&=E_\ka(u_{0,\ka})+\frac{c^2}{2}P_\ka'(0)+\mathop{o}\limits_{c\to0}(c^2)+2Mr_0^4\sin^2\Big(\frac{cP_\ka'(0)+\mathop{o}_{c\to0}(c)}{2r_0^2}\Big),\notag\\
            &=E_\ka(u_{0,\ka})+\frac{c^2}{2}(P_\ka'(0)+MP_\ka'(0)^2)+\mathop{o}\limits_{c\to0}(c^2).
        \end{align}
       From the expansion \eqref{eq:TaylorMu} in Proposition~\ref{prop:moment0}, we obtain for $\mu\geq0$ small enough $$\mathbf{c}(\mu)=\mu\sqrt{4F(0)}/r_0^2+\mathop{O}_{\mu\to0}(\mu^2),$$ therefore, using \eqref{eq:TaylorLc}, we deduce that there exists $K>0$ such that, for $0<\mu<\mu_*$ small enough, we have
        \begin{equation}\label{ineq:suplmin}
            \boL_{\min}(\mu)\leq E_\ka(u_{0,\ka})+K\mu^2.
        \end{equation}
        Note that if one considers $M$ so small so that $P_\ka'(0)+1/M>0$, then \eqref{eq:TaylorLc} yield the stronger estimate $\boL_{\min}(\mu)\leq E_\ka(u_{0,\ka})-\tilde{K}\mu^2$, hence $M>|P_\ka'(0)|^{-1}$ is compulsory in Proposition~\ref{prop:Lmin}. 
        First, proceeding as in Proposition~\ref{prop:minkink}, one can show the following estimate for the local patches of energy.
        \begin{lemma}\label{lem:ineq:energyamplitude}
              Let $v\in\boX(\R)$ satisfying $|v(0)|=y_1$ and 
      $|v(x_0)|=y_2,$ for some $x_0\in\R$ or $x_0=\pm\infty$ and $0\leq y_1< y_2\leq \sqrt{r_0^2+\tilde\xi}$. Then there holds
      \begin{equation}\label{ineq:energyamplitude}
          \int_{0}^{x_0}e_\ka(v)dx\geq 2\int_{y_1}^{y_2}\sqrt{F(r^2)(1+2\ka r^2 h'(r^2)^2)} dr>0.
      \end{equation}
        \end{lemma}
        Then, notice that if $u\in\boX(\R)$ with $\inf_{x\in\R}|u(x)|>0,$  we can write $u=Ae^{i\varphi}$ where $\partial_xA=\Re(\bar u\partial_xu)/|u|$ and $\partial_x\varphi=\Im(\partial_xu\bar u)/|u|^2$ so that $A\in r_0+H^1(\R)$ and $\varphi\in\boC(\R)\cap\dot{H}^1(\R)$.
         We can also rewrite the energy and momentum of $u$ in terms of $A$ and $\partial_x\varphi$
            \begin{align}\label{eq:Ehydro}
    E_\ka(u)&=\int_\R((\partial_xA)^2(1+2\ka A^2 h'(A^2)^2)+(A\partial_x\varphi)^2+F(A^2))dx,\\\label{eq:Phydro}
 \boP(u)&=P(u) \mod 2\pi r_0^2= \int_\R(A^2-r_0^2)\partial_x\varphi dx\mod2\pi r_0^2.
            \end{align}
Moreover, we formally have  $L'(u)=E_\ka'(u)-c{P}'(u)$ with
            \begin{align}\notag
                E_\ka'(u)=&\Big(\frac d{dA}E_\ka(u),\frac d{d(\partial_x\varphi)}E_\ka(u)\Big),\quad\text{ with } \frac{d}{d(\partial_x\varphi)}E_\ka(u)=-2\partial_x(A^2\partial_x\varphi),\\
               \frac{d}{dA}E_\ka(u)=& 2\Big(-\partial_{x}((1+2\ka A^2h'(A^2)^2)\partial_xA)+(\partial_xA)^2(2\ka A h'(A^2))(h'(A^2)+2A^2h''(A^2))\notag\\&+A(\partial_x\varphi)^2+AF'(A^2)\Big),\label{eq:Eprimhydro}\\\label{eq:Pprimhydro}
                {P}'(u)=&\Big(\frac d{dA}P(u),\frac d{d(\partial_x\varphi)}P(u)\Big)=(2A\partial_x\varphi,A^2-r_0^2),
            \end{align}
            and            $c=c(u)={2Mr_0^2}{}\cos(\pi/2-{P}(u)/2r_0^2)\sin(\pi/2-{P}(u)/2r_0^2)/(2r_0^2)$.
            Observe that, due to the term in $A(\partial_x\varphi)^2$ in \eqref{eq:Eprimhydro}, it is not clear how to obtain an equation for the weak limit of a minimizing sequence of $\boL_{\min}(\mu).$
            We overcome this difficulty by constructing suitable minimizing sequences $(v_n)=(A_ne^{i \theta_n})$ 
            with more compactness on $\partial_x\theta_n$.
             \begin{proposition}\label{prop:minseq}
                   There exists a minimizing sequence $(v_n)\subset\boX(\R)$ for $\boL_{\min}(\mu)$ that satisfies additionnally  $P(v_n)\in[0,r_0^2\pi]$, and $v_n=\rho_ne^{i\theta_n}$ with \begin{equation}\label{eq:hydroseq}
                \partial_x\theta_n=c_n\frac{\rho_n^2-r_0^2}{2\rho_n^2}, \quad\text{ and, }c_n\equiv 2Mr_0^2\cos\Big(\frac{\pi}{2}-\frac{P(v_n)}{2r_0^2}\Big)\sin\Big(\frac{\pi}{2}-\frac{P(v_n)}{2r_0^2}\Big).
            \end{equation}
            In this setting, letting $(x_n)\subset\R$ be such that $\rho_n(x_n)=\mu$ for all $n\in\N$ and relabeling $v_n=v_n(\cdot-x_n)$, the following compactness result holds:
             There exist $\rho\in r_0+H^1(\R)$ with $\inf_{\R}\rho=\mu$, and $\theta\in\boC(\R)\cap\dot H^2(\R)$, such that, up to subsequence extraction, we have as $n\to\infty$ \begin{align}\label{eq:convP}
           & {P}(v_n)\to P_\infty,\text{ so that, }c_n\to c\coloneqq2{Mr_0^2}{}\cos({\pi/2-P_\infty}/{(2r_0^2)})\sin({\pi/2-P_\infty}/{(2r_0^2)}).\\\label{eq:convhydro}
               & \rho_n-r_0\wto\rho-r_0,\text{ in }H^1(\R), \quad\partial_x\theta_n\wto\partial_x\theta\coloneqq c(\rho^2-r_0^2)/(2\rho^2) ,\text{ in } H^1(\R),\\
               &\text{ and, }\rho_n(\cdot-x_n)\to\rho, \quad\partial_x\theta_n(\cdot-x_n)\to\partial_x\theta ,\text{ uniformly in every compact set in }\R. \label{eq:convhydro2}
            \end{align}
            Finally, taking $\mu>0$ small enough, we have $\rho^2(x)\leq r_0^2+\tilde\xi/2$ for all $x\in\R$, and, up to translation, we can assume that there exists $R\geq0$ such that \begin{equation*}
            \rho(\pm R)=\mu,\quad\text{ and, }\quad\rho(x)>\mu,\text{ for all }|x|> R. 
        \end{equation*}
             \end{proposition}   
            \begin{proof}
                Let $\mu>0$ and denote by $\boA(\mu)$ be the admissible set for $\boL_{\min}(\mu)$ given by
                $$\boA(\mu)=\{A\in r_0^2+H^1(\R;\R),\inf_{x\in\R}A(x)=\mu,A^2(x)\leq r_0^2+\tilde\xi\}.$$
                Using \eqref{eq:Ehydro}--\eqref{eq:Phydro} and writing $\partial_x\varphi=w,$ there holds
            \begin{align}
                \boL_{\min}(\mu)&=\inf_{A\in\boA}\left\{\int_\R(\partial_xA)^2(1+2\ka A^2 h'(A)^2)+F(A^2))dx+I(A)\right\}, \quad\text{ where }\label{eq:Lminpol}\\
                I(A)&=\inf_{w\in L^2(\R;\R)}\left\{\int_\R(Aw)^2dx+2Mr_0^4\sin^2\Big(\int_\R \frac{A^2-r_0^2}{2r_0^2}wdx-\frac\pi2\Big)\right\},\quad\text{ for all }A\in\boA.\notag
            \end{align}
            The function $I(A)$ is also given by
            \begin{equation*}
                \inf_{p\in\R}\inf\Big\{\int_\R (Aw)^2dx+2Mr_0^4\sin((p-r_0^2\pi)/(2r_0^2))^2~|~w\in L^2(\R;\R), \int_\R (A^2-r_0^2)w=p \Big\}.
            \end{equation*}
            For each $p\in\R$, this minimization problem in $w$ amounts to minimizing a coercive quadratic form on an affine hyperplane. Thus, by standard convexity arguments, there exists a unique minimizer $w_p\in L^2(\R;\R)$ that satisfies the Euler--Lagrange equation $A^2w_p=\lambda (A^2-r_0^2)$ for some $\lambda\in\R$. Multiplying the equation by $(A^2-r_0^2)/A^2$ and integrating, we can compute $\lambda$ and deduce 
            \begin{equation}\label{eq:wp}
                w_p=p\Big(\int_\R (A^2-r_0^2)/A^2dx\Big)^{-1}(A^2-r_0^2)/A^2.
            \end{equation}
            Therefore, $I(A)$ is given by the expression
            \begin{equation*}
                \inf_{p\in\R}\Big\{\int_\R A^2w_p^2dx+{2Mr_0^4}{}\sin^2\Big(\frac{p-r_0^2\pi}{2r_0^2}\Big)\Big\}=\inf_{p\in\R}\Big\{p^2\Big(\int_\R \frac{(A^2-r_0^2)^2}{A^2}dx\Big)^{-1}+{2Mr_0^4}{}\sin^2\Big(\frac{p-r_0^2\pi}{2r_0^2}\Big)\Big\}.
            \end{equation*}
           The infimum in $p$ above is achieved for $p\in[-\pi r_0^2,\pi r_0^2],$ indeed, if $p>r_0^2\pi,$ then $p-2\pi r_0^2$ is a better competitor. Since $p\mapsto\sin((p-\pi r_0^2)/(2r_0^2))^2$ is even, we can restrain the aformentionned problem to $p\in[0,\pi r_0^2]$. We deduce that $p$ satisfies
            \begin{align}\notag
                 \frac{d}{dp}\Big(p^2\Big(\int_\R \frac{(A^2-r_0^2)^2}{A^2}dx\Big)^{-1}+{2Mr_0^4}{}\sin^2\Big(\frac{p-r_0^2\pi}{2r_0^2}\Big)\Big)&=0,\\\label{eq:popti}
                 2p\Big(\int_\R \frac{(A^2-r_0^2)^2}{A^2}dx\Big)^{-1}+{2Mr_0^4}{}\frac{1}{r_0^2}\cos\Big(\frac{r_0^2\pi-p}{2r_0^2}\Big)\sin\Big(\frac{r_0^2\pi-p}{2r_0^2}\Big)&=0.
            \end{align} 
            so that, multiplying \eqref{eq:popti} by $A^2-r_0^2$ and using \eqref{eq:wp}, we deduce that every minimizing sequence $(\rho_n,\theta_n)$ for \eqref{eq:Lminpol} satisfies \eqref{eq:hydroseq}.
            
Now, taking $(\rho_n,\theta_n)$ to be such a minimizing sequence and letting $v_n=\rho_ne^{i\theta_n}$, we have $E_\ka(v_n)\leq L({v_n})$. Hence, the energy of $(v_n)$ is bounded, then, taking $(x_n)\subset\R$ satisfying $\rho_n(x_n)=\mu,$ for all $n\in\N$, and using Corollary~\ref{coro:wlim}, we obtain the convergence of $(\rho_n)$ as in \eqref{eq:convP}--\eqref{eq:convhydro2}.  
Observe that the sequence $(\partial_x\theta_n)=(c(\rho_n^2-r_0^2)/(2\rho_n^2))$ is bounded in $H^1(\R)$, which is not clear using solely $\partial_x\theta_n=\Im(\partial_xv_n\bar v_n)/|v_n|^2$. We deduce that \eqref{eq:convP}--\eqref{eq:convhydro2} holds. 

Since $\inf_{x\in\R}\rho(x)=\mu$ and $\lim_{x\to\pm\infty}\rho(x)=r_0$. It is clear that we can find $z\in\R$ and $R\geq0$ such that $\rho(\pm R+z)=\mu$ and $\rho(x)>\mu$ for all $x\in\R$ satisfying $|x-z|> R.$ 
 
Finally, we show that $\rho^2(\cdot)\leq r_0^2+\tilde\xi/2$ in $\R$. Assume by contradiction that there exists $x_0>R$ such that $\rho(x_0)\geq y$ for some $y>0$ satisfying $r_0^2+\tilde\xi/2<y^2<r_0^2+\tilde\xi$.  By uniform convergence of $\rho_n$, there exists $N\in\N$ such that for all $n\geq N$, we have $\rho_n(R)<2\mu$ and $\rho_n(x_0)>\sqrt{r_0^2+\tilde\xi/2}$. Using Lemma~\ref{lem:ineq:energyamplitude}, we can find $l>0$ such that
\begin{align*}
    E_\ka(v_n)\geq& \int_{-\infty}^Re_\ka(v_n)dx+ \int_{R}^{x_0}e_\ka(v_n)dx\\
    \geq& 2\int_{\mu}^{r_0}\sqrt{F(r^2)(1+2\ka r^2 h'( r^2)^2)} dr+2\int_{2\mu}^{\sqrt{r_0^2+\tilde\xi/2}}\sqrt{F( r^2)(1+2\ka r^2 h'( r^2)^2)} dr\\
    \geq& 4\int_0^{r_0}\sqrt{F( r^2)(1+2\ka r^2 h'( r^2)^2)} dr
    -4\int_{0}^{2\mu}\sqrt{F( r^2)(1+2\ka r^2 h'( r^2)^2)} dr\\
    &+2\int_{r_0}^{\sqrt{r_0^2+\tilde\xi/2}}\sqrt{F( r^2)(1+2\ka r^2h'( r^2)^2)} dr=(1+l)E_\ka(u_{0,\ka})-\mathop{O}_{\mu\to0}(\mu).
\end{align*}
In particular, $\boL_{\min}(\mu)\geq(1+l)E_\ka(u_{0,\ka})-\mathop{O}_{\mu\to0}(\mu)$, which contradicts $\boL_{\min}(\mu)\leq E_\ka(u_{0,\ka})+K\mu^2$ for $\mu$ small enough. Therefore, $\rho^2(\cdot)\leq r_0^2+\tilde\xi/2$ in $(R,\infty)$. The case $x_0<R$ can be handled by arguing similarly.   \end{proof}
We perform an in-depth study of the limit $v=\rho e^{i\theta}$. The next lemma estimates the loss of energy due to the weak convergence of  $(v_n)=(\rho_ne^{i\theta_n})$ to $v$ in the sense of \eqref{eq:convP}--\eqref{eq:convhydro2}.
\begin{lemma}\label{lem:nocompac}
    If $(v_n)\subset\boX(\R)$ is a minimizing sequence for $\boL_{\min}(\mu)$ converging to $v$ is the sense of Corollary~\ref{coro:wlim}, then, there exists $K>0$ such that 
     \begin{equation*}
         E_{\Delta}\geq \frac{|P_\Delta|}{K},\text{ where } E_\Delta=\liminf_{n\to\infty}E_\ka(v_n)-E_\ka(v)\geq0,\text{ and }P_\Delta=\lim_{n\to\infty}
P(v_n)-P(v)=P_\infty-P(v).
     \end{equation*}
\end{lemma}
\begin{proof}
 Take $\ve>0$ small, and let $R>0$ be such that 
\begin{equation}\label{ineq:EPloc}
    \Big|E_\ka(v)-\int_{|x|<R}|\partial_xv|^2+F(|v|^2)+\frac{\ka}{2}|\partial_xh(|v|^2)|^2dx\Big|<\ve,~~ \Big|P(v)-\int_{|x|<R}(\rho^2-r_0^2)\partial_x\theta dx\Big|<\ve.
\end{equation}
Letting $R_0>R$ be such that $|v(x)|>r_0/2$ for all $|x|>R_0$, then, for all $n\in\N$ large enough, there holds $|v_n(x)|\geq r_0/8$ for all $|x|>R_0$. Indeed,
let $N\in\N$ be so large so that $$\inf_{|x|<R}|v_n(x)|\leq2\mu,\quad\text{ and }\quad|v_n(\pm R_0)|>r_0^2/4, \quad\text{ for all }n\geq N.$$ By contradiction, assume that there exist infinitely many indices $n\geq N$ for which there exists $x_n\in\R$ with $|x_n|>R_0$ and $|v_n(x_n)|=r_0^2/8$. Then, we would have, using \eqref{ineq:energyamplitude},
\begin{align*}
    E_\ka(v_n)&=\Big(\int_{-\infty}^{R_0}e_\ka(v_n)dx+\int_{x_n}^{\infty}e_\ka(v_n)dx\Big)+\int_{R_0}^{x_n}e_\ka(v_n)dx\\
    &\geq 4\int_{2\mu}^{r_0}\sqrt{F( r^2)(1+2\ka(r h'( r^2))^2)} dr+2\int_{r_0/8}^{r_0/4}\sqrt{F( r^2)(1+2\ka(r h'( r^2))^2)}dx.
\end{align*}
Therefore, $L(v_n)\geq (1+l)E_\ka(u_{0,\ka})-K\mu$ for some $l>0$ independant of $n\geq N$ and $\mu>0$ which contradicts \eqref{ineq:suplmin} for $\mu>0$ small enough.
 A simple generalization of inequality \eqref{ineq:momentE} provides,
 \begin{equation*}
     \Big|\int_{|x|>R_0}(\rho_n^2-r_0^2)\partial_x\theta_n\Big|\lesssim \frac{8}{r_0^2}\int_{|x|>R_0}\Big(|\partial_xv_n|^2+F(|v_n|^2)+\frac{\ka}{2}|\partial_xh(|v_n|^2)|^2\big)dx.
 \end{equation*}
 Putting this inequality in \eqref{ineq:EPloc} yields
 \begin{equation}\label{ineq:liminfEP}
     E_\ka(v_n)-E_\ka(v)\geq \int_{|x|<R_0}e_\ka(v_n)dx-\int_{|x|<R_0}e_\ka(v)dx+\frac{1}{K}\Big|\int_{|x|>R_0}(\rho_n^2-r_0^2)\partial_x\theta_ndx\Big|-\ve,
 \end{equation}
 Using the expression of $E_\ka(v_n)$ in terms of $\rho_n$ and $\partial_x\theta_n$ in \eqref{eq:Ehydro} and the convergences \eqref{eq:convP}--\eqref{eq:convhydro2} we get
 \begin{equation*}
     \liminf_{n\to\infty}\int_{|x|<R_0} e_\ka(v_n)\geq\int_{|x|<R_0}e_\ka(v),\text{ and }\lim_{n\to\infty}\int_{|x|<R_0}(\rho_n^2-r_0^2)\partial_x\theta_ndx=\int_{|x|<R_0}(\rho^2-r_0^2)\partial_x\theta dx.
 \end{equation*}
Taking the liminf of \eqref{ineq:liminfEP} and using these identities, we obtain
\begin{equation*}
   E_\Delta\geq \frac{1
    }{K}\lim_{n\to\infty}\Big|P(v_n)-\int_{|x|<R}(\rho_n^2-r_0^2)\partial_x\theta_ndx\Big|-\ve=\frac{1}{K}\Big|P_\infty-\int_{|x|<R}(\rho^2-r_0^2)\partial_x\theta dx\Big|-\ve.
\end{equation*}
Using \eqref{ineq:EPloc}, we deduce that $E_\Delta\geq |P_\Delta|/K -(1+{1}/{K})\ve$ and can conclude by letting $\ve\to0.$ 
\end{proof}
  Observe that the equation $ \langle L'(u),(\chi_1,\chi_2)\rangle_{H^{-1}\times H^{1}}=0$ for every $\chi_1,\chi_2\in\boC_c^1(\R)$ corresponds to the weak form of \eqref{GTWc}. Using the expression of traveling-wave solutions in  Theorem~\ref{thm:soli}, we obtain the following characterization of the limit function $v$ on $\R\backslash[-R,R]$.
            \begin{proposition}\label{prop:limformula} Let $(\rho,\theta)$ be the limit in the sense \eqref{eq:convP}--\eqref{eq:convhydro2} of the minimizing sequence $(v_n)\subset\boX(\R)$ in Proposition~\ref{prop:minseq}. Let $c\in\R$ and $R\geq0$ be defined according to Proposition~\ref {prop:minseq}. Then, there holds $0\leq c\leq \mathbf{
    c}(\mu)$ and there exist $\varphi_\pm\in\R$ and $z\geq0$ with $(\mu^2-\mu_c^2)^{1/2}/K\leq z\leq K(\mu^2-\mu_c^2)^{1/2},$ for some $K>0$, such that, $$\rho(x)e^{i\theta(x)}=e^{i\varphi_-}u_{c,\ka}(x-R+z)\text{ for all } x\geq R,\text{and } \rho(x)e^{i\theta(x)}=e^{i\varphi_+}u_{c,\ka}(x+R-z)\text{ for all }x\leq -R,$$ where $\mathbf{c}(\mu)$, $\mu_c$ and and $u_{c,\ka}$ are given by Proposition~\ref{prop:implisol}.
            \end{proposition}
\begin{proof}
By Proposition~\ref{prop:minseq}, the function $\rho$ satisfies $\mu^2<\rho^2\leq r_0^2+\tilde\xi/2$ in $(R,\infty)$. Let $\chi_1\in\boC_c^1((R,\infty)).$ By uniform convergence of $\rho_n$ to $\rho$ in the support of $\chi_1$, we can find $N>0$ and $t_0>0$ such that $\mu^2\leq(\rho_n+t\chi_1)^2\leq r_0^2+\tilde\xi$ for all $n\geq N$ and $t\in(-t_0,t_0)$. Letting $\chi_2\in\boC_c^1((R,\infty)),$ we compute          \begin{align*}
        L((\rho_n+t\chi_1)e^{i(\theta_n+t\chi_2)})-L(v_n)=t\langle E_\ka'(v_n)-c(v_n)P'(v_n),(\chi_1,\chi_2)\rangle+\mathop{o}_{t\to0}(t),\text{ for all }t\in(-t_0,t_0),
            \end{align*} where $E_\ka'(\cdot)$ and $P'(\cdot)$ are given according to \eqref{eq:Eprimhydro}--\eqref{eq:Pprimhydro} and $c(v_n)=c_n$ is  given by \eqref{eq:hydroseq}.
            Letting $n\to\infty$ in the equation above, since $(v_n)$ is a minimizing sequence for $\boL_{\min}(\mu)$, we can find $t_1>0$ such that $t\langle L'(v),(\chi_1,\chi_2)\rangle\geq0$ for all $t\in(-t_1,t_1)$. 
            Thus, taking $t>0$ and $t<0$ we deduce that $v$ satisfies \eqref{GTWc} in $(R,\infty).$ 
            
            It remains to show the estimate on $c$ to recover the precise formula of $v$ in $(R,\infty)$.
            Proceeding as in  Poposition~\ref{prop:Geqeta}--\ref{prop:Gpropu}, we can show that $\eta=\rho^2-r_0^2$ satisfies \eqref{Geta1} in $(R,\infty)$ so that $\boV_c(\xi)<0$ for all $\xi\in(\mu^2-r_0^2,0).$ Taking $\mu>0$ small enough, the sign condition of $\boV_c(\cdot)$ implies that $|c|\leq \mathbf{c}(\mu)\leq \delta,$ where $\mathbf{c}(\mu)>0$ and $\delta>0$ are given by Proposition~\ref{prop:implisol}. Indeed, since $\boV_c(\mu^2-r_0)^2\to c^2r_0^2>0$ as $\mu\to0$, we must have $c<c_*$ for $\mu$ sufficiently small. Assume by contradiction that $c\in(\mathbf{c}(\mu),c_*)$, then proceeding as in Proposition~\ref{prop:implisol}, we can find $\xi(c)>\xi(\mathbf{c}(\mu))=\mu^2-r_0^2$ such that $\boV_{c}(\xi(c))=0$ which contradicts the condition $\boV_c(\xi)<0$ for all $\xi\in(\mu^2-r_0^2,0).$

Taking $z\geq0$ to be the unique nonnegative number such that $|u_{c,\ka}(z)|=\mu$, we deduce, using equation \eqref{Geta1}, that  $\eta(\cdot)$ and $\tilde{\eta}(\cdot)\coloneqq|u_{c,\ka}(\cdot+z)|^2-r_0^2$ both satisfy the equation $\eta(x)=\boF_c^{-1}(x-R)$ where $\boF_c:[\mu^2-r_0^2,0)\to[0,\infty)$ is a continuous increasing function given by
\begin{equation*}
  \boF_c(y)= \int_{\mu^2-r_0^2}^{y} \sqrt{\frac{1+2\ka(r_0^2+\xi)(h'(r_0^2+\xi))^2}{-\boV_c(\xi)}}d\xi,
\end{equation*}
the formula of $v=\rho e^{i\theta}$ in $(R,\infty)$ follows. To estimate $z\geq0$, observe that \begin{equation*}
    z=\int_{\xi(c)}^{\mu^2-r_0^2} \sqrt{\frac{1+2\ka(r_0^2+\xi)(h'(r_0^2+\xi))^2}{-\boV_c(\xi)}}d\xi.
\end{equation*}
Since $\boV_c(\cdot)$ has a simple zero at $\xi=\xi(c)$, we deduce
$$\frac{1}{2K\sqrt{\xi-\xi(c)}}\leq\sqrt{\frac{1+2\ka(r_0^2+\xi)(h'(r_0^2+\xi))^2}{-\boV_c(\xi)}}\leq \frac{K}{2\sqrt{\xi-\xi(c)}},\quad\text{ for all }\xi\in(\xi_c,\mu^2-r_0^2),$$ and the bounds on $z$ follow by integration.
This concludes the proof in the case $x\geq R$, whereas the case $x\leq-R$ follows similarly, using the fact that $|u_{c,\ka}|$ is even.
\end{proof}
To characterize the values of $v$ in $[-R,R]$, we will use the following formulae.  
 \begin{lemma}\label{lem:Lminexpli}
 In the setting of Proposition~\ref{prop:limformula}, there holds
  \begin{align}\label{eq:Lminexpli}
    \boL_{\min}(\mu)=\liminf_{n\to\infty}L(v_n)=E_\ka(v)+E_\Delta+Mr_0^4\Big(1-\sqrt{1-\frac{c^2}{M^2r_0^4}}\Big),
\end{align}
with $E_\Delta=\liminf_{n\to\infty} E_\ka(v_n)-E_\ka(v)\geq0$ and \begin{align}\label{eq:Evexpli}
     E_\ka(v)=E_\ka(u_{0,\ka})+\frac{c^2}{2}P_\ka'(0) +\mathop{o_{c\to0}(c^2)}-4z||F||_{L^{\infty}(0,\mu^2)}+\int_{-R}^Re_\ka(v)dx.
 \end{align} 
 \end{lemma}
 \begin{proof} Using expansion~\ref{eq:taylorEK}, we compute
 \begin{align}\notag
     E_\ka(v)&=E_\ka(u_{c,\ka})-\int_{-z}^ze_\ka(u_{c,\ka})dx+\int_{-R}^Re_\ka(v)dx,\\\notag
     &=E_\ka(u_{0,\ka})+\frac{c^2}{2}P_\ka'(0) +\mathop{o_{c\to0}(c^2)}-\int_{-z}^ze_\ka(u_{c,\ka})dx+\int_{-R}^Re_\ka(v)dx.
 \end{align}
 Then using $|u_{c,\ka}'|^2=F(|u_{c,\ka}|^2)-\ka h(|u_{c,\ka}|^2)'/2$, we have
 \begin{equation*}
    \int_{-z}^ze_\ka(u_{c,\ka})=\int_{-z}^z2F(|u_{c,\ka}|^2)dx\leq 4z||F||_{L^\infty((0,\mu^2))},
 \end{equation*} since $0\leq|u_{c,\ka}(\cdot)|\leq\mu$ in $(-z,z)$,
and \eqref{eq:Evexpli} follows.

 To show \eqref{eq:Lminexpli}, recall that $P(v_n)\in[0,r_0^2\pi]$ with  $c_n=Mr_0^2\sin(\pi-P(v_n)/r_0^2)$ (by the doubling angle property), thus
 \begin{align*}\notag
     2Mr_0^4\sin^2\Big(\frac{P(v_n)-r_0^2\pi}{2r_0^2}\Big)&=Mr_0^4\Big(1-\cos\Big(\frac{P(v_n)-r_0^2\pi}{r_0^2}\Big)\Big)=Mr_0^4\Big(1-\mathfrak{s}\sqrt{1-\sin\Big(\frac{P(v_n)-r_0^2\pi}{r_0^2}\Big)^2}\Big),\\
     &=Mr_0^4\Big(1-\mathfrak{s}\sqrt{1-\frac{c_n^2}{M^2r_0^4}}\Big),
 \end{align*}
 where $\mathfrak{s}=-1$ if $P(v_n)\leq r_0^2\pi/2$ and $\mathfrak{s}=1$ otherwise.
Since $c_n\to c=Mr_0^2\sin(\pi-P_\infty/r_0^2)$ we have
 \begin{align*}
    \boL_{\min}(\mu)=\liminf_{n\to\infty}L(v_n)=E_\ka(v)+E_\Delta+Mr_0^4\Big(1-\mathfrak{s}\sqrt{1-\frac{c^2}{M^2r_0^4}}\Big).
\end{align*}
    Thus, to obtain \eqref{eq:Lminexpli}, it remains to show that $P(v_n)\in[r_0^2\pi/2,r_0^2\pi]$ and $\mathfrak{s}=1$ for $n$ large enough. By contradiction, if for all $N\in\N$ there exists $n\geq N$ such that $0\leq P(v_n)<r_0^2\pi/2$, then using \eqref{eq:Evexpli} above, we obtain
 \begin{align*}
    \boL_{\min}(\mu)&\geq E_\ka(u_{0,\ka})+\frac{c^2}{2}P_\ka'(0) +\mathop{o_{c\to0}(c^2)}-4z||F||_{L^{\infty}(0,\mu^2)}+\int_{-R}^Re_\ka(v)dx+Mr_0^4\Big(1+\sqrt{1-\frac{c^2}{M^2}}\Big),\\
    &\geq E_\ka(u_{0,\ka})-4z||F||_{L^\infty((0,\mu^2))}+2Mr_0^4+\mathop{O}_{c\to0}(c^2),
\end{align*}
where we used $E_\Delta\geq0$ and $e_\ka(v)\geq0$ in $(-R,R)$.
Using $0\leq z\lesssim\mu$, we deduce that $\boL_{\min}(\mu)\geq E_\ka(u_{0,\ka})+Mr_0^4$ for $\mu$ small enough, contradicting inequality \eqref{ineq:suplmin}.
Therefore, $P_\infty\in[r_0^2\pi/2,r_0^2\pi]$ and $\mathfrak{s}=1$ which concludes. 
 \end{proof}
 As a last preliminary, the following results establish the expression of $P_\Delta=P_\infty-P(v)$ in terms of $c$, and a useful expansion of $P\Delta$ at order one in $\mu$.
\begin{lemma}\label{lem:PDelta}
In the setting of Proposition~\ref{prop:limformula}, and taking $\eta=|u_{c,\ka}|^2-r_0^2$, there holds
   \begin{align}\label{eq:Pdelta}
    P_\Delta=-c\Big({P   }_0+\frac{1}{M}\Big)+c\int_0^z\frac{\eta}{r_0^2+\eta}dx-c\int_{-R}^R\frac{(\rho^2-r_0^2)^2}{2\rho^2}dx+\mathop{o}_{c\to0}(c).
\end{align}
Moreover, letting $\xi(c)=\min_{x\in\R} \eta(x)$, we have the expansion 
 \begin{equation}\label{eq:PDeltaexp}
       P_\Delta=-c\Big({P   }_0+\frac{1}{M}\Big)+2r_0^2\Big(\frac{\pi}{2}-\atan\sqrt{\frac{\xi(c)+r_0^2}{\mu^2-r_0^2-\xi(c)}}\Big)-c\int_{-R}^R\frac{(\rho^2-r_0^2)^2}{2\rho^2}dx+\mathop{o}_{\mu\to0}(\mu)
   \end{equation} 
\end{lemma}
\begin{proof}
   By definition of $c={2Mr_0^2}{}\cos({\pi/2-P_\infty}/{(2r_0^2)})\sin({\pi/2-P_\infty}/{(2r_0^2)})$,
   and the doubling angle property,
   there holds $P_\infty=r_0^2\pi-r_0^2\arcsin(c/(Mr_0^2))$. Thus, splitting $P(v)$ into two integrals yields
   \begin{align*}
    P_\Delta=P_\infty-P(v)=r_0^2\pi-r_0^2\arcsin\Big(\frac{c}{Mr_0^2}\Big)-\Big(P(u_{c,\ka})-c\int_0^z\frac{\eta}{r_0^2+\eta}dx\Big)-c\int_{-R}^R\frac{(\rho^2-r_0^2)^2}{2\rho^2}dx.
    \end{align*}
   Equation \eqref{eq:Pdelta} follows using the Taylor expansion of $P(u_{c,\ka})+r_0^2\arcsin(c/(Mr_0^2))$ at $c=0$ above.

At fixed $\mu>0$, using $\eta_{c,\ka}(z)=\mu$ and proceeding as in Proposition~\ref{prop:moment0}, we get 
 \begin{align}\label{eq:taylorPP1}
                 \int_0^z&\frac{\eta}{\eta_{}+r_0^2}dx=\int_{\xi(c)}^{\mu^2-r_0^2}\frac{w^2}{w+r_0^2}\sqrt\frac{1+2\ka(\xi(c)+r_0^2)(h'(\xi(c)+r_0^2))^2}{-\boV_c'(\xi_c)(w-\xi(c))}dw\\
                &+\int_{\xi(c)}^{\mu^2-r_0^2}\frac{w^2}{w+r_0^2}\left(\frac{\sqrt{{1+2\ka(w+r_0^2)(h'(w+r_0^2))^2}}}{\sqrt{-\boV_c(w)}}-\sqrt\frac{1+2\ka(\xi(c)+r_0^2)(h'(\xi(c)+r_0^2))^2}{{\boV'_c(\xi(c))(\xi(c)-w)}}\right)dw.\label{eq:taylorPP2}
             \end{align}
             By Lebesgue's dominated convergence theorem, we get that the integral in \eqref{eq:taylorPP2} is equal to  
\begin{equation*}
    \int_{-r_0^2}^{\mu-r_0^2}\frac{w^2}{w+r_0^2}\left(\sqrt\frac{{{1+2\ka(w+r_0^2)(h'(w+r_0^2))^2}}}{{-\boV_0(w)}}-\sqrt\frac{1}{{-\boV_0(-r_0^2)(r_0^2+w)}}\right)dw+\mathop{o}_{c\to0}(1),
\end{equation*}
In particular, \eqref{eq:taylorPP2} tends to zero as $\mu\to0$. Proceeding again as in Proposition~\ref{prop:moment0},
the integral on the right-hand side in \eqref{eq:taylorPP1} is equal to
 \begin{align*}\notag
              &\frac{2}{\sqrt {\xi(c)+r_0^2}} \int_{0}^{\sqrt{(\mu^2-r_0^2-\xi(c))/(\xi(c)+r_0^2)}}\frac{(\xi(c)+(\xi(c)+r_0^2)t^2)^2}{1+t^2}\sqrt\frac{{1+2\ka(\xi(c)+r_0^2)(h'(\xi(c)+r_0^2))^2}}{{-\boV_c'(\xi(c))}}dt\\\notag
              &=\frac{2}{\sqrt {\xi(c)+r_0^2}}\sqrt\frac{{1+2\ka(\xi(c)+r_0^2)(h'(\xi(c)+r_0^2))^2}}{{-\boV_c'(\xi(c))}}\Big(r_0^4\atan\sqrt{\frac{\mu^2-r_0^2-\xi(c)}{\xi(c)+r_0^2}}\\\notag
              &\quad+(\xi(c)-r_0^2)\sqrt{(\mu^2-r_0^2-\xi(c))(\xi(c)+r_0^2)}+\frac{\sqrt{(\mu^2-r_0^2-\xi(c))^3(\xi(c)+r_0^2)}}{3}\Big)\\
              &=\frac{2}{\sqrt {\xi(c)+r_0^2}}\sqrt\frac{{1+2\ka(\xi(c)+r_0^2)(h'(\xi(c)+r_0^2))^2}}{{-\boV_c'(\xi(c))}}\Big(r_0^4\frac{\pi}{2} -r_0^4\atan\sqrt{\frac{\xi(c)+r_0^2}{\mu^2-r_0^2-\xi(c)}}\Big)+\mathop{o}_{c\to0}(1),            
           \end{align*}
           where we used $\xi(c)+r_0^2=\mathop{o}_{c\to0}(1)$. Precisely, recall the expansion $\xi(c)=-r_0^2+c^2r_0^4/(4F(0))+\mathop{O}_{c\to0}(c^4)$, which implies that 
          \begin{equation*}
              \frac{2}{\sqrt {\xi(c)+r_0^2}}\sqrt\frac{{1+2\ka(\xi(c)+r_0^2)(h'(\xi(c)+r_0^2))^2}}{{-\boV_c'(\xi(c))}}=\frac{2}{r_0^2c}+\mathop{o}_{c\to0}(1).
          \end{equation*} 
Combining the above computations in \eqref{eq:taylorPP1}--\eqref{eq:taylorPP2}, we get
\begin{equation*}
     \int_0^z\frac{\eta}{\eta_{}+r_0^2}dx=2r_0^2\Big(\frac{\pi}{2}-\atan\sqrt{\frac{\xi(c)+r_0^2}{\mu^2-r_0^2-\xi(c)}}\Big)+\mathop{o}_{c\to0}(1)+\mathop{o}_{\mu\to0}(1),
\end{equation*}
which, put in \eqref{eq:Pdelta} yields \eqref{eq:PDeltaexp}.
\end{proof}
In the next result, we crucially use the slope condition $P_\ka'(0)<0$ to show that $R\geq0$ is in fact nonzero.
In this setting, we prove that $\rho(\cdot)\equiv \mu$ in $(-R,R)$. For this second point, we proceed by induction. We can then show that $v$ satisfies \eqref{GTWc} in an interval of $(-R,R)$; however, due to the different boundary conditions, equations \eqref{Geta1}--\eqref{Geta2} do not hold. Ruling out the case of small amplitude solutions in this setting requires precise estimates on $\max_{x\in[-R,R]}\rho(x)$, on $R$ and on $c$ in terms of $\mu>0$.
\begin{proposition}\label{prop:limformula2}
In the setting of Proposition~\ref{prop:limformula},
    There holds $0<R\lesssim\mu$ and $\rho(x)=\mu$ for all $x\in(-R,R)$.
\end{proposition}
\begin{proof}
 Recall the coarse estimates $0\leq c\leq \mathbf{c}(\mu)\sim \mu$ and $z\lesssim \mu$ in Proposition~\ref{prop:limformula}. We need three preparatory steps for this proof, where we obtain refined estimates on $\max_{x\in[-R,R]}\rho(x)$, on $R$ and on $c$ respectively in terms of $\mu>0$.
 \begin{step}
     There holds  $\rho(\cdot)\lesssim\mu$ in $[-R,R].$
 \end{step}
For $R=0$ it is clear. In the case $R>0$, let $x_0\in(-R,R)$ be such that $\rho(x_0)=\min(r_0/2,\max_{x\in[-R,R]}\rho(x))$, then, using the formula of $E_\ka(v)$ in Lemma~\ref{lem:Lminexpli} and Lemma~\ref{lem:ineq:energyamplitude}, we obtain
 \begin{align*}
           E_\ka(v)&\geq E_\ka(u_{c,\ka})- \int_{-z}^ze_\ka(u_{c,\ka})dx+\int_{R}^{x_0}e_\ka(u_{c,\ka})dx\\
           &\geq E_\ka(u_{0,\ka})+\mathop{O}_{c\to0}(c^2)-4\int_0^z F(r_0^2+\eta)dx+\int_{\mu}^{\rho(x_0)}\sqrt{F( r^2)(1+2\ka(r h'( r^2))^2} dr\\
           &\geq E_\ka(u_{0,\ka})-4K\mu||F||_{L^\infty((0,\mu))}+(\rho(x_0)-\mu)\inf_{\sigma\in(\mu,r_0/2)}\sqrt{F( r^2)(1+2\ka(\sigma h'( r^2))^2} +\mathop{O}_{\mu\to0}(\mu^2).
       \end{align*}     
Using $E_\ka(v)\leq\boL_{\min}(\mu)\leq E_\ka(u_{0,\ka})+K\mu^2$, we deduce that $\rho(x_0)\lesssim\mu$.
For $\mu$ small enough, we get $\rho(x_0)<r_0/2$ so that $\rho(x_0)=\max_{x\in[-R,R]}\rho(x)\lesssim\mu$ as stated. 

\begin{step}
    The interval radius $R$ satisfies $R\lesssim \mu$.
\end{step}
This bound follows readily from Step~1. Indeed, using $\int_{-R}^Re_\ka(v)dx\geq\int_{-R}^RF(|v|^2)dx\geq 2R\inf_{r\in(\mu^2,r_0^2/2)}F(r)$ and Lemma~\ref{lem:Lminexpli}, we get 
\begin{equation*}
    \boL_{\min}(\mu)\geq E_\ka(u_{0,\ka})+\frac{c^2}{2}(P_\ka'(0)+\frac{1}{M})+\mathop{o}_{c\to0}(c^2)+2R\inf_{r\in(\mu^2,r_0^2/2)}F(r)-4z||F||_{L^\infty((0,\mu^2))}.
\end{equation*}
We conclude that $R\lesssim ({\inf_{r\in(\mu^2,r_0^2/2)}F(r)})^{-1}(z+\mathop{O}(\mu^2))\lesssim\mu$ using $\boL_{\min}(\mu)\leq E_\ka(u_{0,\ka})+K\mu^2$.

\begin{step}
    There holds $c\gtrsim\mu$.
\end{step}
We check that $P_\Delta$ (and thus $E_\Delta$) has low-order terms in $\mu$ that need the scaling $c\gtrsim\mu$ to cancel. 
Using \eqref{eq:Lminexpli}--\eqref{eq:Evexpli} and invoking Step~5, we deduce that 
\begin{equation*}
    \boL_{\min}(\mu)\geq E_\ka(u_{0,\ka})+\frac{|P_\Delta|}{K}+\mathop{O}_{c\to0}(c^2)+2R\inf_{r\in(\mu^2,r_0^2/2)}F(r)-4z||F||_{L^\infty((0,\mu^2))}.
\end{equation*}
Since $c+R+z\lesssim\mu$ and $\boL_{\min}(\mu)-E_{\ka}(u_{0,\ka})\lesssim\mu^2$, we get $|P_\Delta|\lesssim\mu$. Using this and \eqref{eq:PDeltaexp}, we obtain
\begin{equation}\label{ineq:Pdeltasmal}
    2r_0^2\Big(\frac{\pi}{2}-\atan\sqrt{\frac{\xi(c)+r_0^2}{\mu^2-r_0^2-\xi(c)}}\Big)-c\int_{-R}^R\frac{\eta^2}{2(\eta+r_0^2)}dx\lesssim\mu.
\end{equation}
Besides, using the expansion of $\xi(c)$ in Proposition~\ref{prop:implisol}, we compute 
\begin{align}\label{ineq:c/mu}
      \atan\sqrt{\frac{\xi(c)+r_0^2}{\mu^2-r_0^2-\xi(c)}}\lesssim\sqrt{\frac{\xi(c)+r_0^2}{\mu^2-r_0^2-\xi(c)}}\lesssim\frac{c}{\mu},
    \text{ and, }\quad c\int_{-R}^R\frac{\eta^2}{2(r_0^2+\eta)}dx \lesssim \frac{cR}{\mu^2}\lesssim\frac{c}{\mu}.
\end{align}
Putting \eqref{ineq:c/mu} in \eqref{ineq:Pdeltasmal} implies $r_0^2\pi-Kc/\mu\lesssim\mu$ thus
$Kc\gtrsim \mu-\mu^2$, so that $c\gtrsim\mu$  for $\mu>0$ small enough.
\begin{step}
    The interval radius $R$ is nonzero.
\end{step}
    By contradiction, suppose that $R=0$. The contradiction will arise from $\boL_{\min}(\mu)\leq E_\ka(u_{0,\ka})+K\mu^2$ for $\mu>0$ small enough, by showing that $|P_\Delta|\gtrsim \mu$ and $z\lesssim\mu^2$ so that $\boL_{\min}\geq E_{\ka}(u_{0,\ka})+\mu/K-K\mu^2$.
Setting $R=0$ in the expression of $P_\Delta$ in Lemma~\ref{lem:PDelta}, we get
\begin{equation}\label{eq:PdeltaR0}
    P_\Delta=-c\Big(P_\ka'(0)+\frac{1}{M}\Big)+2r_0^2\Big(\frac{\pi}{2}-\atan\sqrt{\frac{\xi(c)+r_0^2}{\mu^2-r_0^2-\xi(c)}}\Big)+\mathop{o}_{\mu\to0}(\mu).
\end{equation}
Using again the formula of $\boL_{\min}$ in Lemma~\ref{lem:Lminexpli}, and the estimate $E_\Delta\gtrsim P_\Delta$ in Lemma~\ref{lem:nocompac} with $R=0$, we obtain 
\begin{equation}\label{ineq:Rposi}
E_\ka(u_{0,\ka})+K\mu^2\geq\boL_{\min}(\mu)\geq E_\ka(u_{0,\ka})+|P_\Delta|/K-K(z+\mu^2),\end{equation}
we deduce that $|P_\Delta|\lesssim\mu$. From equation \eqref{eq:PdeltaR0}, we see that there holds 
$$\frac{\pi}{2}-\atan\sqrt{\frac{\xi(c)+r_0^2}{\mu^2-r_0^2-\xi(c)}}\lesssim\mu.$$
Hence taking $\mu>0$ small enough we must have $\mu^2-r_0^2-\xi(c)\leq \xi_c+r_0^2$. In that regime, we compute
\begin{equation*}
   \frac{\pi}{2}-\atan\sqrt{\frac{\xi(c)+r_0^2}{\mu^2-r_0^2-\xi(c)}}= \atan\sqrt{\frac{\mu^2-r_0^2-\xi(c)}{\xi(c)+r_0^2}}\gtrsim\sqrt{\frac{\mu^2-r_0^2-\xi(c)}{\xi(c)+r_0^2}},
\end{equation*}
so that $\mu^2-r_0^2-\xi(c)\lesssim\mu^2(r_0^2+\xi(c))=-\mu^2(\mu^2-r_0^2-\xi(c))+\mu^4$.
We deduce that $$(\mu^2-r_0^2-\xi(c))\lesssim \mu^4,$$ and in particular, $z\lesssim\mu^2$ by Proposition~\ref{prop:limformula}. 
Besides, using $c\gtrsim\mu$ (see Step~3) and $\pi/2-\atan(y)\geq0$ for all $y\in\R$, we deduce from \eqref{eq:PdeltaR0} that $P_\Delta\gtrsim -\mu(1/M+P_\ka'(0))/2>0,$ for $\mu>0$ small enough. Here, it is important that $1/M+P_\ka'(0)<0$ holds true to conclude, hence  $P_\ka'(0)<0$ is compulsory in our framework. Using $z\lesssim\mu^2$ and $|P_\Delta|\gtrsim\mu$ in \eqref{ineq:Rposi} contradicts $\boL_{\min}(\mu)\leq E_\ka(u_{0,\ka})+K\mu^2$, for $\mu>0$ small enough, therefore $R>0$.
\begin{step}
  There holds $\rho(x)=\mu$ for all $x\in(-R,R)$.
\end{step}
Assume by contradiction that there exists $(x_1,x_2)\subset(-R,R)$ such that 
             $\mu^2<|v(x)|^2\leq r_0^2$ for all $x\in(x_1,x_2)$ and $|v(x_1)|=|v(x_2)|=\mu$. Proceeding as in Proposition~\ref{prop:limformula}, we can show that $v$ satisfies \eqref{GTWc} in $(x_1,x_2)$. Let $x_0\in(x_1,x_2)$ be such that $\eta\coloneqq\rho^2-r_0^2$ reaches a local maximum at $x_0$, then, following the same lines as in Proposition~\ref{prop:Geqeta}, we deduce that $\eta$ satisfies 
             \begin{align}
\notag2(1+2\ka(|v|h'(|v|^2))^2)\eta''+2\ka(\eta')^2(h'(|v|^2)+2h'(|v|^2)h''(|v|^2))&\\\label{eq:Geta2loc}
            +\boV_c'(\eta)+4(F(r_0^2+\eta(x_0))-|v'(x_0)|^2)&=0, \quad \text{in }(x_1,x_2),
        \end{align}
        Since $x_0$ is a local maximum of $\eta$, there holds $\eta''(x_0)\leq0$, which put in \eqref{eq:Geta2loc} gives
        \begin{equation*}
           -2c^2\eta(x_0)-4(r_0^2+\eta(x_0))f(r_0^2+\eta(x_0))+\frac{c^2(\eta(x_0))^2}{4(\eta(x_0)+r_0^2)}\leq0.
        \end{equation*} 
         Re-aranging the terms above, we deduce
          \begin{equation*}
-2c^2(\eta(x_0)+r_0^2)\eta(x_0)+c^2(\eta(x_0))^2\leq 4 (\eta(x_0)+r_0^2)^2||f||_{L^\infty(0,r_0^2)}.
          \end{equation*}
         Thus, using $\rho(\cdot)\lesssim \mu$ (see Step~1), we get
          $\eta(x_0)\lesssim \mu^2-r_0^2$, and
        \begin{equation*}
            c^2(1+\mu^2)\lesssim\mu^4,
        \end{equation*}
        Putting $c\gtrsim\mu$ (see Step~3) above, we obtain a contradiction for $\mu>0$ small enough, therefore $\rho(x)=\mu$ for all $x\in(-R,R)$.
\end{proof}
We can now draw the full picture and show that the limit of the minimizing sequences of Proposition~\ref{prop:minseq} enjoys the formula \eqref{eq:Lminima}.
  \begin{proposition}\label{prop:limformula3}
  In the setting of Proposition~\ref{prop:limformula} and Proposition~\ref{prop:limformula2}, there hold $c=\mathbf{c}(\mu)$ and $z=0$, so that $v$ satisfies \eqref{eq:Lminima} with $R>0$, $x_0=0$, $C=\mathbf c(\mu)(\mu^2-r_0^2)^2/\mu^2$, and $\varphi=\varphi_+.$
 \end{proposition} 
 \begin{proof}
 Letting $\chi\in\boC_c^1(\R)$ satisfying $\chi(x)\geq0$ for all $x\in\R$, we can proceed as in Proposition~\ref{prop:limformula} to deduce that $t\langle E'_\ka(v)-c\boP(v),(\chi,0)\rangle_{H^{-1}\times H^1}\geq0$ for all $t\geq0$ small enough.
 Rewriting this identity in terms of $\rho$ and $\partial_x\theta$ using \eqref{eq:Eprimhydro}--\eqref{eq:Pprimhydro}, we get
 \begin{align}\label{ineq:z01}
    0\leq &\int_{\R}(1+2\ka\rho^2 h'(\rho^2)^2)\partial_x\rho\partial_x\chi+(\partial_x\rho)^2(2\ka \rho h'(\rho^2))(h'(\rho^2)+2\rho^2h''(\rho^2))\chi dx\\
     &+\int_\R(\rho(\partial_x\theta)^2+\rho F'(\rho^2)- c\rho\partial_x\theta )\chi dx.\notag
 \end{align}
 Take $\chi(\cdot)$ supported in $(R-\ve,\infty)$ for some $\ve>0$ with $\chi(R)>0$ and integrate by part in the term in $\partial_x\chi$ in \eqref{ineq:z01} to get, for $C=1+2\ka\mu^2h'(\mu^2)^2>0$
 \begin{align}\notag
     0\leq&-C\partial_x\rho_{c,\ka}(z)\chi(R)+\int_{R}^\infty(-(1+2\ka\rho^2 h'(\rho^2)^2)\partial_{xx}\rho+\rho(\partial_x\theta)^2+\rho F'(\rho^2)- c\rho\partial_x\theta )\chi dx\\
     &+\int_{R-\ve}^R(\mu(\partial_x\theta)^2+\mu F'(\mu^2)- c\mu\partial_x\theta )\chi dx,\label{ineq:z02}
 \end{align}
 where we used that $\rho(x)=\mu$ for all $x\in(-R,R)$, so that $\partial_x\rho(R^-)=0$, and the identity
$$\partial_x((1+2\ka(\rho h'(\rho^2))^2)\partial_x\rho)=(\partial_x\rho)^2(2\ka \rho h'(\rho^2))(h'(\rho^2)+2\rho^2h''(\rho^2))+(1+2\ka(\rho h'(\rho^2))^2)\partial_{xx}\rho.$$
Since $\rho$ satisfies \eqref{GTWc} in $(R,\infty)$ the first integral in \eqref{ineq:z02} is zero. Therefore, letting $\ve\to0$ in \eqref{ineq:z02}, we deduce that $\partial_x\rho_{c,\ka}(z)\leq0$. Since $z\geq0$, we must have $z=0$ and $\partial_x\rho_{c,\ka}(z)=0$, using the elementary properties of $|\rho_{c,\ka}|^2-r_0^2$ in Proposition~\ref{prop:Gpropu}. Using $v=\rho e^{i\theta}$ where $\theta$ is given by Proposition~\ref{prop:minseq}, we can show that \eqref{eq:Lminima} holds true.
 \end{proof}
An interesting byproduct of the analysis done so far is that the variational problem $\boL_{\min}(\mu)$ does not have a minimum for $\mu>0$ small enough. 
 \begin{proposition}
    Let $\mu_*>0$ be so small so that Proposition~\ref{prop:limformula3} holds. If $0<\mu<\mu_*$, then the infimum in $\boL_{\min}(\mu)$ is not attained.
 \end{proposition}
 \begin{remark}
     This result implies that, for any minimizing sequence $(v_n)$ for $\boL_{\min}(\mu)$, we always have at least two bubbles going apart in the profile decomposition of $v_n$.
 \end{remark}
 \begin{proof}
     By contradiction, suppose that there exists a minimizer $v$ for $L_{\min}(\mu)$ Then, going along the same lines as in the proof of Proposition~\ref{prop:minseq}, we can assume that $$v=|v|e^{i\theta},\quad\text{ with }\quad\theta(x)=\frac c2\int_{0}^x\frac{(|v|^2-r_0^2)^2}{|v|^2}dx,\quad\text{ and }c=Mr_0^2\sin(\pi-P(v)/r_0^2).$$ Given that $(v)$ is a stationary minimizing sequence for $\boL_{\min}(\mu)$, we can apply Proposition~\ref{prop:limformula3} to obtain $c=c(\mu)$ and establish the existence of some $x_0\in\R$ and $R_{\mu}>0$ such that \eqref{eq:Lminima} holds with $R=R_\mu$. On one hand, using this and the formula of $c$, we get an implicit formula for $R_\mu$.
     \begin{equation}\label{eq:impliformulaR}
         c=Mr_0^2\sin\Big(\pi-\frac{P(u_{c,\ka})}{r_0^2}-cR_\mu\frac{\mu^2-r_0^2}{\mu^2}\Big).
     \end{equation}
     On the other hand, we compute
     \begin{equation*}
         L(v)=E(u_{c,\ka})+2R_\mu F(\mu^2)+2Mr_0^4\sin^2\Big(\frac\pi2-\frac{P(u_{c,\ka})}{2r_0^2}-cR_\mu\frac{\mu^2-r_0^2}{2\mu^2}\Big).
     \end{equation*}
     Then, using \eqref{eq:impliformulaR}, we get $$\partial_{R_\mu}L(v)=2F(\mu^2)+c>0,$$ therefore, one can find $R$ in the neighborhood of $R_\mu$, and a function $w$ satisfying \eqref{eq:Lminima} with that $R$ so that $L(v)>L(w)$, contradicting the minimality of $v$.  \end{proof}
However we do not have $\boL_{\min}(\mu)=L(v),$ the key inequality  $\boL_{\min}(\mu)\geq E_\ka(u_{0,\ka})+K^{-1}\mu^2$ can be readily established using the expression of $v$ in \eqref{eq:Lminima}, as shown below.
 \begin{proof}[Proof of Proposition~\ref{prop:Lmin}]
 We write $c=\mathbf{c}(\mu)$ for the sake of conciseness. Recall that $c=\sqrt{4F(0)}\mu/r_0^2+\mathop{O}(\mu^2)$, hence using the formula of $v$  in \eqref{eq:Lminima} and Lemma~\ref{lem:Lminexpli}, we get
 \begin{align*}
     \boL_{\min}(\mu)&= E_\ka(u_{0,\ka})+E_\Delta+\frac{c^2}{2}\Big(P_\ka'(0)+\frac{1}{M}\Big) +\mathop{o}_{c\to0}(c^2)+2R\Big(\frac{c^2(\mu^2-r_0^2)^2}{4\mu^2}+F(\mu^2)\Big),\notag\\
     &\geq E_\ka(u_{0,\ka})+E_\Delta+\frac{c^2}{2}\Big(P_\ka'(0)+\frac{1}{M}\Big) +\mathop{o}_{c\to0}(c^2)+4RF(0)+\mathop{o}_{\mu\to0}(\mu^2),
 \end{align*}
 where the $\mathop{o}_{c\to0}(c^2)$ and $\mathop{o}_{\mu\to0}(\mu^2)$ are independant of $R\lesssim \mu.$
 Combining inequality $E_\Delta\gtrsim P_\Delta$ in Lemma~\ref{lem:nocompac}, the formula of $P_\Delta$ in Lemma~\ref{lem:PDelta} with $z=0$, and the formula of $\boL_{\min}$ in Lemma~\ref{lem:Lminexpli} with $z=0$, we infer that $\boL_{\min}(\mu)$ is greater or equal to
 \begin{equation}\label{ineq:Rminor}
   E_\ka(u_{0,\ka})+\frac{c}{K}\Big|\Big(P_\ka'(0)+\frac{1}{M}\Big)+\mathop{o}_{\mu\to0}(1)+R\frac{(\mu^2-r_0^2)^2}{\mu^2}\Big|+\frac{c^2}{2}\Big(P_\ka'(0)+\frac{1}{M}\Big)+4RF(0)+  \mathop{o}(\mu^2).
 \end{equation}
 Since $c(\mu^2-r_0^2)^2/(K\mu^2)\gtrsim\mu^{-1}$ is greater than $4F(0)$ for $\mu>0$ small enough, we deduce that the expression in \eqref{ineq:Rminor} is a piecewise affine function with respect to $R$, 
 which is decreasing in $(0,R(\mu))$, and increasing in $(R(\mu),\infty)$ where
 \begin{equation*}
     R(\mu)=-\mu^2\Big(P_\ka'(0)+\frac{1}{M}+\mathop{o}_{\mu\to0}(1)\Big)/(\mu^2-r_0^2)^2=-\mu^2(P_\ka'(0)+\frac{1}{M})/r_0^4+\mathop{o}_{\mu\to0}(\mu^2).
 \end{equation*}
 Hence setting $R=R(\mu)$ in \eqref{ineq:Rminor}, we obtain
 \begin{equation*}
\boL_{\min}    (\mu)\geq   E_\ka(u_{0,\ka})+\Big(\frac{4F(0)}{r_0^4}\mu^2-\frac{c^2}{2}\Big)\Big|P_\ka'(0)+\frac{1}{M}\Big|+\mathop{o}_{\mu\to0}(\mu^2),
 \end{equation*}
which concludes, since $4F(0)\mu^2/r_0^4-c^2/2\gtrsim\mu^2.$
        \end{proof}
The following coercivity inequality will be crucial in the proof of Theorem~\ref{thm:stab}. 
\begin{proposition}\label{prop:lyap} There exist $\ve_0>0$ and $K>0$ such that for all $\Psi\in\boX(\R)$ satisfying $d_{\boX}(\Psi,u_{0,\ka})\leq\ve_0,$ if $z\in\R$ is such that $\Psi(z)=\inf_\R|\Psi|$ and $\varphi\in\R$ satisfies $\Psi(z)e^{i\varphi}\in\R,$ then
\begin{equation}\label{ineq:lyap}
           d_{\boX}(\Psi(\cdot+z)e^{i\varphi},u_{0,\ka}(\cdot))^2\leq K\sqrt{L(\Psi)-E_\ka(u_{0,\ka})+\inf_\R|\Psi(\cdot)|}.
\end{equation}
\end{proposition}
\begin{proof}
   We define $\mu=\inf_{x\in\R}\Psi(x)\geq0$. Up to translation, we can find $l\geq0$, such that
    $$|\Psi(0)|=|\Psi(l)|=\mu\quad\text{ and }|\Psi(\cdot)|>\mu,\quad\text{ in }\R\backslash(0,l).$$ Then, we have the local liftings $\Psi(x)=|\Psi(x)|e^{i\phi(x)}$ for all $x\in\R\backslash[0,l]$
    where $\phi\in \boC\cap \dot{H}^1(\R\backslash[0,l])$.
    Using the pointwise inequality $|\partial_x\Psi(x)|\geq|\partial_x|\Psi(x)||$ for a.e. $x\in(0,l)$ and proceeding as in Proposition~\ref{prop:minkink}, we obtain
\begin{align*}
    \int_0^\infty e_\ka(\Psi)dx\geq&\int_l^\infty(|\Psi|\partial_x\phi)^2dx+\int_0^\infty |\partial_x|\Psi||^2(1+2\ka(|\Psi|h'(|\Psi|^2))^2)+F(|\Psi|^2)dx,\\
    \geq&\int_l^\infty|\Psi|^2\partial_x\phi^2dx+\int_0^\infty\Big(\sqrt{1+2\ka(|\Psi|h'(|\Psi|^2))^2}\partial_x|\Psi|-\sqrt{F(|\Psi|^2)}\Big)^2dx\\
    &+2\int_\mu^{r_0}\sqrt{F( r^2)(1+2\ka(\sigma h'( r^2))^2} dr.
\end{align*}
Hence, arguing similarly in $(-\infty,0)$ and using $E_\ka(u_{0,\ka})=4\int_0^{r_0}\sqrt{F( r^2)(1+2\ka(\sigma h'( r^2))^2} dr$, we get 
\begin{align}\label{ineq:hboundE}
    E_\ka(\Psi)\geq &E_\ka(u_{0,\ka})+\int_{\R\backslash(0,l)}|\Psi|^2(\partial_x\phi)^2dx+\int_\R\Big(\sqrt{1+2\ka(|\Psi|h'(|\Psi|^2))^2}\partial_x|\Psi|-\sqrt{F(|\Psi|^2)}\Big)^2dx\notag\\
    &-4\int_0^{\mu}\sqrt{F( r^2)(1+2\ka(\sigma h'( r^2))^2)} dr.
\end{align}
Let $K_0>1$ be such that  $4\int_0^{\mu}\sqrt{F( r^2)(1+2\ka(\sigma h'( r^2))^2)} dr\leq (K_0-1)\mu.$
We define $\delta=(E_\ka(\Psi)-E_\ka(u_{0,\ka})+K_0\mu)^{1/4}$ so that $$\mu\leq\delta^4\leq K_0(L(\Psi)-E_\ka(u_{0,\ka})+\mu)\lesssim d_{\boX}(\Psi,u_{0,\ka})\lesssim\ve_0.$$

The strategy for proving \eqref{ineq:lyap} is the following. First, we carefully establish the estimate 
    $|||\Psi(x)|-|u_{0,\ka}|||_{H^1(\R)}\lesssim\delta^2$. This result will enable us to show that $l\lesssim \delta^2$ so that 
    $d_{\boX}(\indicator_{(0,l)}\Psi,\indicator_{(0,l)}u_{c,\ka})^2\lesssim\delta^2$.
    Then we will conclude by showing the following bound in $\R\backslash(0,l)$ 
    $$ d_{\boX}(\indicator_{\R\backslash(0,l)}\Psi,\indicator_{\R\backslash(0,l)}u_{c,\ka})^2\lesssim L(\Psi)-E_\ka(u_{0,\ka})+K_0\mu+\delta^2.$$
\begin{step}
    $H^1(\R)$-estimates on $\mathfrak{N}=|\Psi|-|u_{0,\ka}|.$
\end{step}
Using the formula of $(\partial_x|u_{0,\ka}|)^2$ in \eqref{Geq:quadratic}, we have
\begin{equation*}
    \partial_x\mathfrak{N}=\sqrt{\frac{F(|\Psi|^2)}{{1+2\ka|\Psi|^2h'(|\Psi|^2)^2}}}-\sqrt{\frac{F(|u_{0,\ka}|^2)}{{1+2\ka|u_{0,\ka}^2|h'(|u_{0,\ka}|^2)^2}}}+g,\quad\text{ a.e. in $\R$},
\end{equation*}
where $g=\partial_x|\Psi|-\sqrt{{F(|\Psi|^2)}/{({1+2\ka|\Psi|^2h'(|\Psi|^2)^2})}}$ satisfies $||g||_{L^2(\R)}^2\lesssim\delta^4.$ Indeed using \eqref{ineq:hboundE}, we have
\begin{align*}
    \inf_{x\in\R}|{1+2\ka|\Psi(x)|^2h'(|\Psi(x)|^2)^2}|||g||^2_{L^2(\R)}\leq E_\ka(\Psi)-E_\ka(u_{0,\ka})+4\int_{0}^{\mu}\sqrt{F( r^2)(1+2\ka(\sigma h'( r^2))^2)} dr.
\end{align*}
Let $R>0$, we show that $||\mathfrak{N}||_{H^1((0,R))}\lesssim\delta^2$ by a trapping argument. Indeed, using Taylor expansion, we have a.e. in $(0,\infty)$ 
\begin{equation}\label{ineq:diffmathfrak{N}}
    \partial_x\mathfrak{N}(x)=\frac{d}{dy}\Big(\sqrt{\frac{F(y^2)}{1+2\ka y^2h'(y^2)^2}}\Big){\rvert}_{y=|u_{0,\ka}(x)|}\times\mathfrak{N}(x)+G(x,\mathfrak{N}(x))+g(x),
\end{equation}
where $G(x,\theta)\leq C \theta^2$ whenever $|\theta|\lesssim\theta_0$ and for all $x\in\R$.
Since $u_{0,\ka}(\cdot)>0$, $\partial_xu_{0,\ka}(\cdot)>0$ and \eqref{Geq:quadratic} holds in $(0,\infty)$, we get $|u_{0,\ka}(\cdot)|=u_{0,\ka}(\cdot)$ and
\begin{equation*}
    \partial_{xx}u_{0,\ka}=\frac{d}{dy}\Big(\sqrt{\frac{F(y^2)}{1+2\ka y^2h'(y^2)^2}}\Big){\rvert}_{y=u_{0,\ka}}\times\partial_xu_{0,\ka},\quad\text{ in }(0,\infty).
\end{equation*}
Hence multiplying \eqref{ineq:diffmathfrak{N}} by $(\partial_xu_{0,\ka})^{-1}$ and integrating from $0$ to $x>0$ we obtain
\begin{equation}\label{ineq:diffkinkmathfrak{N}}
    \frac{\mathfrak{N}(x)}{\partial_xu_{0,\ka}(x)}-\frac{\mu}{\partial_xu_{0,\ka}(0)}+\int_0^x\frac{\partial_{xx}u_{0,\ka}}{(\partial_xu_{0,\ka})^2}\mathfrak{N} dw=\int_0^x\frac{\partial_{xx}u_{0,\ka}}{(\partial_xu_{0,\ka})^2}\mathfrak{N}+\frac{G(w,\mathfrak{N})+h(w)}{\partial_xu_{0,\ka}}dw,
\end{equation}
where we performed integration by parts on the left-hand side. Recall the sharp exponential decay of $\eta_{c,\ka}$ and $\partial_xu_{c,\ka}$ in Proposition~\ref{prop:implisol}, which writes in the case of the kink 
\begin{equation}\label{ineq:decaykink}
    \frac{1}Ke^{-C_\ka |x|}\leq(|u_{0,\ka}|^2-r_0^2)^2+(\partial_xu_{0,\ka})^2\leq Ke^{-C_\ka|x|}, \text{ with } C_\ka= \frac{c_s}{\sqrt{1+2\ka r_0^2h'(r_0^2)^2}},
\end{equation}
for some $K>0$. Multiplying \eqref{ineq:diffkinkmathfrak{N}} by $\partial_xu_{0,\ka}(x)$ and using this fact, we get
\begin{equation}\label{ineq:Gronwmathfrak{N}}
    |\mathfrak{N}(x)|\leq K^2\mu+K^2\int_0^xe^{-C_\ka(w-x)}|G(w,\mathfrak{N})+h(w)|dw,\quad\text{ for all }x\in\R.
\end{equation}
Take $\theta_1<\min\{\theta_0, C_\ka (K^2C)^{-1}\}$ and
consider the set $$I=\{R>0: |\mathfrak{N}(x)|\leq\theta_1, \text{ for all }x\in[0,R]\}.$$ Assume that $\ve_0>0$ (and thus $\mu>0$) is so small so that $0\in I$ and $I\subset\R$ is a nonempty closed set.
Using H\"older's inequality, we obtain, for all $R\in I$
\begin{equation*}
    |\mathfrak{N}(x)|\leq K^2\mu+\frac{K^2C}{C_\ka}||\mathfrak{N}||^2_{L^\infty([0,R])}+\frac{K^2}{\sqrt{2C_\ka}}||g||_{L^2((0,\infty))},\text{ for all }x\in[0,R].
\end{equation*}
Hence there holds
\begin{equation}\label{ineq:mathfrak{N}infty}
    ||\mathfrak{N}||_{L^\infty([0,R])}\lesssim\frac{\mu+\delta^2}{1-{\theta}_1C_\ka^{-1}K^2C},
\end{equation}
which is strictly smaller than $\theta_1$ whenever $\ve_0>0$ is small enough. We deduce that $I=\R$ by connexity, so that $G(x,\mathfrak{N}(x))\leq C\mathfrak{N}(x)^2$ for all $x\in\R$. Let $R>0$, taking the $L^2$-norm in \eqref{ineq:Gronwmathfrak{N}} and using the Young--Hausdorff's inequality, we obtain
\begin{align*}
    ||\mathfrak{N}||_{L^2((0,R))}\leq& K^2\mu\sqrt{R}+K^2\int_0^Re^{-C_\ka w}dw(C||\mathfrak{N}^2||_{L^2((0,R))}+||g||_{L^2((0,\infty)})\\
\leq&K^2\mu\sqrt{R}+K_1||\mathfrak{N}||_{L^\infty(0,R)}||\mathfrak{N}||_{L^2(0,R)}+K_2\delta^2.
\end{align*}
Using \eqref{ineq:mathfrak{N}infty}, we can take $\ve_0>0$ small enough so that $K_1||\mathfrak{N}||_{L^\infty(0,R)}\leq1/2$. Using this and \eqref{ineq:Gronwmathfrak{N}}, we get
$ ||\mathfrak{N}||_{L^2((0,R))}\lesssim\mu\sqrt{R}+\delta$, which combined with \eqref{ineq:diffmathfrak{N}} readily yield
\begin{equation}\label{ineq:mathfrak{N}bound}
    ||\mathfrak{N}||^2_{H^1((0,R))}\lesssim\mu^2R+\delta^4.
\end{equation}

Similarly we get $||\mathfrak{N}||^2_{H^1((-R,0))}\lesssim\mu^2R+\delta^4,$ hence it remains to bound $||\mathfrak{N}||^2_{H^1(\R\backslash(-R,R))}$. Notice that we must take $R\lesssim\mu^{-1}$ in \eqref{ineq:mathfrak{N}bound} to ensure $||\mathfrak{N}||_{H^1(\R)}^2\lesssim\delta^4$. Using the exponential decay in \eqref{ineq:decaykink},
 we have for $|x|>R$
\begin{align*}
    (\partial_x\mathfrak{N})^2\leq &2(\partial_x|\Psi|)^2+2(\partial_xu_{0,\ka})^2\lesssim |\partial_x\Psi|^2+e^{-2C_\ka |x|}\\
    \mathfrak{N}^2\leq&2(|\Psi|-r_0)^2+2(|u|-r_0)^2\leq\frac{2}{(|\Psi|+r_0)^2}(|\Psi|^2-r_0^2)^2+\frac{2K}{(|u_{0,\ka}|+r_0)^2}e^{-2C_\ka |x|}.\notag
\end{align*}
Summing these two inequalities and using the pointwise inequality \eqref{lem:pointwiseEnergy}, we deduce $||\mathfrak{N}||^2_{H^1(\R\backslash(-R,R)}\lesssim\int_{|x|>R}e_\ka(\Psi)dx+e^{-2C_\ka |x|}.$
Since $F(u_{0,\ka}(x)^2)\lesssim(u_{0,\ka}(x)^2-r_0^2)^2$ for $R$ large enough and $|x|>R$,  using \eqref{ineq:decaykink}, we readily compute
   $ \int_{|x|>R}e_\ka(u_{0,\ka})\lesssim e^{-2C_\ka |x|}.$
Therefore, using $|e_\ka(\Psi)|\geq e_\ka(|\Psi|)$ a.e. in $\R$, we get
\begin{equation}\label{ineq:mathfrak{N}unbound}
     ||\mathfrak{N}||^2_{H^1(\R\backslash(-R,R)}\lesssim E_\ka(\Psi)-E_\ka(u_{0,\ka})+\int_{-R}^Re_\ka(u_{0,\ka})-e_\ka(|\Psi|)dx+e^{-2C_\ka |x|} .
\end{equation}
Proceeding as in \eqref{eq:Eprimhydro} and performing integration by parts for the terms in $\partial_x\mathfrak{N}$, we get
\begin{align}\notag
    \int_{-R}^Re_\ka(u_{0,\ka})-e_\ka(|\Psi|)dx=&
    2\int_{-R}^R\Big((1+2\ka(|u_{0,\ka}|h'(|u_{0,\ka}|^2))^2\partial_{xx}|u_{0,\ka}| +|u_{0,\ka}|F'(|u_{0,\ka}|^2)\Big)\mathfrak{N} dx\\\notag
   &+2\Big[(1+2\ka(|u_{0,\ka}|h'(|u_{0,\ka}|^2)\partial_xu_{0,\ka}\mathfrak{N}\Big]_{x=-R}^{x=R}+\mathop{\text {O}}(||\mathfrak{N}||_{H^1((-R,R))}^2),\\\notag
    \lesssim&(\partial_xu_{0,\ka}(R))^2+||\mathfrak{N}||^2_{L^\infty((-R,R))}+||\mathfrak{N}||_{H^1((-R,R))}^2,\\
    \lesssim&e^{-2C_\ka |x|}+\mu^2R+\delta^4 ,\notag
    \end{align}
    where we used that $u_{0,\ka}$ solves $(\text{TW}(0,\ka))$ to cancel the integral in the right-hand side. Using this inequality in \eqref{ineq:mathfrak{N}unbound} and adding \eqref{ineq:mathfrak{N}bound} yield 
    $$||\mathfrak{N}||_{H^1(\R)}^2\lesssim \mu^2R+\delta^4+e^{-C_\ka R}.$$
    Hence, taking $R=\mu^{-1}$, we infer that, for $\ve_0>0$ small enough, we have $e^{-2C_\ka/\mu}\leq\mu$. We deduce the desired inequality \begin{equation}\label{ineq:H1mathfrak{N}}
        ||\mathfrak{N}||_{H^1(\R)}\lesssim\delta^2.
    \end{equation}
\begin{step}
   Proof of the bound $d_{\boX}(\indicator_{(0,l)}\Psi,\indicator_{(0,l)}u_{0,\ka})^2\lesssim\delta^2$.
\end{step} To obtain this bound, we estimate $l>0.$ Appealing to the Sobolev embedding and \eqref{ineq:H1mathfrak{N}}, we have
$$|u_{0,\ka}(l)|\leq|\Psi(l)|+||\mathfrak{N}||_{L^\infty(\R)}\leq K(\mu+\delta^2)$$  
Since $\partial_x(|u_{0,\ka}|^2)>K_1>0$ in  $[0,l]$, applying $(|u_{c,\ka}|^2)^{-1}$ and the mean value inequality, we deduce
$$l\leq |u_{c,\ka}|^{-1}(K(\mu+\delta^2)) -|u_{c,\ka}|^{-1}(0)\leq\frac{K}{K_1}(\mu+\delta^2)\lesssim\delta^2.$$
Therefore we have the bound in $(0,l)$
\begin{equation}\label{ineq:stab0l}
    d_{\boX}(\indicator_{(0,l)}\Psi,\indicator_{(0,l)}u_{0,\ka})^2=\int_0^l|\partial_x\Psi-\partial_xu_{0,\ka}|^2+(|\Psi|-|u_{0,\ka}|)^2dx\lesssim l+\mu^2\lesssim\delta^2.
\end{equation}
\begin{step}
    Bounds in $\R\backslash(0,l)$ and end of the proof
\end{step}
Observe that inequality \eqref{ineq:hboundE} yields $|||\Psi|\partial_x\phi||^2_{L^2(\R\backslash(0,l))}\leq\delta^4$. Hence
\begin{align}\label{ineq:graddiff}\notag
   ||\partial_x\Psi-\partial_xu_{0,\ka}||_{L^2({\R\backslash(0,l)})}^2=&\int_{\R\backslash(0,l)}|ie^{i\phi}|\Psi|\partial_x\phi+e^{i\phi}\partial_x|u_{0,\ka}|-\partial_xu_{0,\ka}+e^{i\phi}\partial_x\mathfrak{N} |^2dx\\\notag
    \lesssim&||e^{i\phi}\partial_x|u_{0,\ka}|-\partial_xu_{0,\ka}||^2_{L^2({\R\backslash(0,l)})}+||\mathfrak{N}||_{H^1(\R)}^2+\delta^4,\\
    \lesssim&||e^{i\phi}\partial_x|u_{0,\ka}|-\partial_xu_{0,\ka}||^2_{L^2({\R\backslash(0,l)})}+E_\ka(\Psi)-E_\ka(u_{0,\ka})+K_0\mu.
\end{align}
We expand the integral in the right-hand side of \eqref{ineq:graddiff} as follows
\begin{align*}
    ||e^{i\phi}\partial_x|u_{0,\ka}|-\partial_xu_{0,\ka}||^2_{L^2({\R\backslash(0,l)})}&=\int_{-\infty}^0|e^{i\phi}+1|^2(\partial_xu_{0,\ka})^2dx+\int_l^\infty|e^{i\phi}-1|^2(\partial_xu_{0,\ka})^2dx,\\
    &=2\int_{-\infty}^0(1+\cos(\phi))(\partial_xu_{0,\ka})^2dx+2\int_l^\infty(1-\cos(\phi))(\partial_xu_{0,\ka})^2dx.
\end{align*}
We want to bound the above quantity in terms of $\delta$. Denoting $a=\delta$, we have, for all $x\in(l,\infty)$
\begin{align*}
    |1-\cos(\phi(x))|&\leq|1-\cos(\phi(l+a))|+\Big|\int_{l+a}^x\partial_x\phi\sin(\phi)dw\Big|,\\
    &\leq|1-\cos(\phi(l+a))|+\frac{{|x-l-a|^{1/2}}}{\inf_{(l+a,\infty)}|\Psi(\cdot)|^2}\int_{l+a}^\infty(|\Psi|\partial_x\phi)^2 dw\\
    &\leq|1-\cos(\phi(l+a))|+\frac{\delta^4}{\inf_{(l+a,\infty)}|\Psi(\cdot)|^2}|x-l-a|^{1/2}
\end{align*}
Hence, proceeding similarly for $1+\cos(\phi(\cdot))$ in $(-\infty,-a)$ and using \eqref{ineq:graddiff}, we obtain
\begin{align}\label{ineq:gradcomplex}\notag
    ||\partial_x\Psi-\partial_xu_{0,\ka}||_{L^2({\R\backslash(0,l)})}^2\leq&K\delta^4+\Big(|1+\cos(\phi(-a))|+|1-\cos(\phi(l+a))|+\frac{\delta^4}{\inf_{\R\backslash(-a,l+a)}|\Psi(\cdot)|^2}\Big)\\
    &\times\int_\R (1+\sqrt{|x-l-a|}+\sqrt{|x-a|})(\partial_xu_{0,\ka})^2dx,\notag\\
    \lesssim&|1+\cos(\phi(-a))|+|1-\cos(\phi(l+a))|+\frac{\delta^4}{\inf_{\R\backslash(-a,l+a)}|\Psi(\cdot)|^2}.
\end{align}
To bound this term, we need to overcome two difficulties. Up to a constant phase change, we can assume $\phi(l+a)=0$ that is $|1-\cos(\phi(l+a))|=0$, but then, we need to show that  
$$|1+\cos(\phi(-a))|\lesssim \sqrt{L(\Psi)-E_\ka(u_{0,\ka})+K_0\mu},$$ for all $\ve_0>0$ small enough. Furthermore, in the case $\mu=0$ we have  $\inf_{w\in\R\backslash(-a,l+a)}|\Psi(w)|\to0$ as $\ve_0\to0$, hence we need information on the convergence rate to control the quotient term in \eqref{ineq:gradcomplex}. We can estimate the infimum above using simple computations: For all $x\in\R\backslash(-a,l+a)$, we have
\begin{equation*}
    |\Psi(x)|=|u_{c,\ka}(x)|-\mathfrak{N}\geq|u_{c,\ka}(-a)|-||\mathfrak{N}||_{L^\infty(\R)}\geq\int_{-a}^0\partial_x|u_{0,\ka}|(w)dw-||\mathfrak{N}||_{L^\infty(\R)}\geq \frac{1}K(a-||\mathfrak{N}||_{H^1(\R)}),
\end{equation*}
thus, whenever $\ve_0>0$ is small enough so that $\delta<1$, we deduce 
\begin{equation}\label{ineq:infloc}
    \inf_{\R\backslash(-a,l+a)}|\Psi(\cdot)|\gtrsim a-\delta^2\gtrsim a-\delta^2\gtrsim a.
\end{equation}
In general, we see that $a\ll\delta^2$ is compulsory to close the estimate and that, without further assumptions on $\Psi\in\boX(\R)$, the best rate for the quotient term in \eqref{ineq:gradcomplex} is $\delta^2$.
Proceeding similarly in $(-a,l+a)$ and using $l\lesssim\delta^2$, we obtain
\begin{equation}\label{ineq:suploc}
     \sup_{\R\backslash(-a,l+a)}|\Psi(\cdot)|\lesssim l+a+\delta^2\lesssim a.
\end{equation}
It remains to bound $|1+\cos(\phi(-a))|$. Since $E_\ka(\Psi)-E_\ka(u_{0,\ka})+K_0\mu>0$, we obtain
\begin{equation}\label{ineq:Lboundphase}
    L(\Psi)-E_\ka(u_{0,\ka})+K_0\mu\geq Mr_0^2\left(\sin\Big(\frac{\boP(\Psi)-r_0^2\pi}{2r_0^2}\Big)\right)^2\gtrsim({\boP(\Psi)}+\pi+2jr_0^2\pi)^2.
\end{equation}
Moreover, for all $x>l+a$ we have $\Re(i\Psi(x)\overline{\partial_x\Psi(x))}=|\Psi(x)|^2\partial_x\phi(x)$, hence
 \begin{equation*}
     \boP(\Psi)=\int_{-a}^{l+a}\Re(i\Psi\overline{\partial_x\Psi})dx+\int_{\R\backslash(-a,l+a)}(|\Psi|^2-r_0^2)\partial_x\phi dx+r_0^2\phi(-a)-r_0^2\phi(l+a).
 \end{equation*}
Putting that in \eqref{ineq:Lboundphase} and using $\phi(l+a)=0$, the triangular inequality and the Cauchy--Schwarz inequality, we get
\begin{align*}\notag
    |\phi(-a)+\pi+2j\pi|\lesssim&\sqrt{L(\Psi)-E_\ka(u_{0,\ka})+K_0\mu}+\int_{-a}^{l+a}|\Psi\overline{\partial_x\Psi}|dx+\int_{\R\backslash(-a,l+a)}|(|\Psi|^2-r_0^2)\partial_x\phi| dx,\\
    \lesssim&\sqrt{L(\Psi)-E_\ka(u_{0,\ka})+K_0\mu}+\sup_{\R\backslash(-a,l+a)}|\Psi(\cdot)|\times\sqrt{l+2a}+\frac{|||\Psi|\partial\phi||_{L^2(\R\backslash(0,l))}^2}{\inf_{\R\backslash(-a,l+a)}|\Psi(\cdot)|^2}\\
    \lesssim&\sqrt{L(\Psi)-E_\ka(u_{0,\ka})+K_0\mu}+(a)^{3/2}+\frac{\delta^4}{a^2},
\end{align*}
where we used, for the last line, inequalities $|||\Psi|\partial\phi||_{L^2(\R\backslash(0,l))}^2\leq\delta^4$ and \eqref{ineq:infloc}--\eqref{ineq:suploc}. From that, the taylor expansion of $\cos(\cdot)$ near $-\pi-2j\pi$ provides
\begin{equation}\label{ineq:phaseloc}
    |1+\cos(\phi(-a))|\lesssim|\phi(-a)+\pi+2j\pi|^2\lesssim L(\Psi)-E_\ka(u_{0,\ka})+K_0\mu+a^2+\frac{\delta^8}{a^4}.
\end{equation}
Combining \eqref{ineq:infloc} and \eqref{ineq:phaseloc} in \eqref{ineq:gradcomplex} we get
\begin{equation*}
    ||\partial_x\Psi-\partial_xu_{0,\ka}||^2_{L^2(\R\backslash(0,l))}\lesssim \delta^4+L(\Psi)-E_\ka(u_{0,\ka})+K_0\mu+a^2+\frac{\delta^4}{a^2}\lesssim L(\Psi)-E_\ka(u_{0,\ka})+K_0\mu+a^2.
\end{equation*}
We can conclude, using this together with inequalities \eqref{ineq:H1mathfrak{N}} in Step~1 and \eqref{ineq:stab0l} in Step~2. Indeed we have
\begin{align*}
    d_{\boX}(\Psi,u_{c,\ka})^2=&d_{\boX}(\indicator_{(0,l)}\Psi,\indicator_{(0,l)}u_{c,\ka})^2+d_{\boX}(\indicator_{\R\backslash(0,l)}\Psi,\indicator_{\R\backslash(0,l)}u_{c,\ka})^2\\
    =&d_{\boX}(\indicator_{(0,l)}\Psi,\indicator_{(0,l)}u_{c,\ka})^2+\int_{\R\backslash(0,l)}|\partial_x\Psi-\partial_xu_{0,\ka}|^2+(|\Psi|-|u_{c,\ka}|)^2dx+\mu^2\notag\\
    \lesssim&\delta^2+L(\Psi)-E_\ka(u_{0,\ka})+K_0\mu+a^2+||\mathfrak{N}||_{L^2(\R)}^2\lesssim L(\Psi)-E_\ka(u_{0,\ka})+K_0\mu+\delta^2.
\end{align*}
Since $L(\Psi)-E_\ka(u_{0,\ka})\geq0$ by Proposition~\ref{prop:Lmin}, we have
$\delta^4\leq L(\Psi)-E_\ka(u_{0,\ka})+K_0\mu\leq K_0( L(\Psi)-E_\ka(u_{0,\ka})+\mu),$ which concludes the proof using the inequality above.
\end{proof}
Finally, we provide the proof of our main result.
\begin{proof}[Proof of Theorem~\ref{thm:stab}] Recall that $L(u_{0,\ka})=E_\ka(u_{0,\ka})$.
Take $\mu_*>0$ as in Proposition~\ref{prop:Lmin} and $0<\ve_0<\mu_*$ in Proposition~\ref{prop:lyap}. By Proposition~\ref{prop:Lmin}, we can find $K>0$ such that for all $v\in\boX(\R)$ with $d_{\boX}(v,u_{0,\ka})\leq\ve_0$, there holds $\inf_{x\in\R}|v(x)|\leq K(L(v)-E_\ka(u_{0,\ka}))^{1/2}$. Then, applying Proposition~\ref{prop:lyap} we deduce
$$\inf_{(z,\varphi)\in\R^2}d_{\boX}(v,u_{0,\ka}(\cdot+z)e^{i\varphi})\leq K(L(v)-E_\ka(u_{0,\ka})+K(L(v)-E_\ka(u_{0,\ka}))^{1/2})^{1/4}.$$ Up to taking $\ve_0$ smaller so that $L(v)-E_\ka(u_{0,\ka})<1$, and taking larger $K>0$, we obtain
\begin{align}\label{ineq:coerciLyap}
    \inf_{(z,\varphi)\in\R^2}d_{\boX}(v,u_{0,\ka}(\cdot+z)e^{i\varphi})\leq K(L(v)-E_\ka(u_{0,\ka}))^{1/8}.
\end{align}
Now let $\Psi(\cdot,t)$ be, as stated, the solution to the IVP associated with \eqref{QGP}. Since $L(\Psi(\cdot,t))=L(\Psi_0)$, we can conclude using \eqref{ineq:coerciLyap} with $v=\Psi(\cdot,t)$ and the local Lipshitz continuity of $L(\cdot)$ (see Lemma~\ref{lem:Elip} and Lemma~\ref{lem:momentuntwist}).
\end{proof}

\subsection{Explicit computation of the criterion in an example}\label{sec:examples}
In this subsection, we check the stability criterion $\partial_cP(u_{c,\ka})<0$ when $h(r)=r$ and $F(r)=(1-r)^2/2$. In a previous work,
 we obtained explicit formulas for $\partial_cP(u_{c,\ka})$ in terms of $(c,\ka)\in(0,\sqrt{2})\times(-1/2,\infty)$ (or $(c,-\ka)\in(0,\sqrt{2})\times(-\infty,1/2)$ with the notation theirein)
\begin{proposition}\label{prop:P0QGP}[Lemmas 5.5 in~\cite{deLaire2023exotic}] In the case $h(r)=r$ and $F(r)=(1-r)^2/2$ the following statements hold for all $\ka>-1/2$: Considering the function $c\mapsto P_\ka(c)\coloneqq P(u_{c,\ka})$, then $P_\ka(\cdot)\in\boC^\infty((0,\sqrt{2}))$ and
\begin{equation*}
    P_\kappa'(c)=-\frac{3c^2\ka-4\ka+1}{4\sqrt{|\ka|}}\atan\Big(\sqrt{|\ka|}\sqrt{\frac{2-c^2}{1+2\ka}}\Big)-\frac{3(2-c^2)}{4}\sqrt{\frac{1+2\ka}{2-c^2}},\quad\text{ for all }c\in(0,\sqrt{2}).
\end{equation*}
Moreover the limit $P_\ka'(0^+)$ exists and is increasing in $\ka\in(0,\infty)$. In particular, there exists $\ka_0\approx 3.636$ such that $$P_\ka'(0^+)<0,\quad\text{ for all $0<\ka<\ka_0$, and } \quad P_\ka'(0^+)>0, \quad \text{ for all $\ka>\ka_0$}.$$
\end{proposition}

We deduce the explicit description of the stability interval of parameters in that case
\begin{proof}[Proof of Proposition~\ref{prop:stabQGP}]
    Using Proposition~\ref{prop:P0QGP} yields the result for $0<\ka<\infty.$
    It remains to check that $\partial_cP(u_{c,\ka})_{\rvert c=0}<0$ for all $-1/2<\ka\leq0$, but this is a simple consequence of its incresingness with respect to $\ka$ that was established in Proposition~\ref{prop:stabkappa}.
\end{proof}
\appendix
\section{The energy and momentum in dimension one}\label{sec:topology}
This appendix aims to picture elementary properties of the energy and of the (untwisted) momentum, in particular their Lipschitz continuity with respect to $d_{\boX}$. We also compare the various distances on $\boX(\R)$ that appeared in the literature. 
\subsection{On the energy space}
We consider a quasilinear Ginzburg--Landau type energy functional 
that is formally conserved by the flow of \eqref{QGP} 
\begin{align*}
    E_\ka(v)=\int_\R\Big(|\partial_xv|^2+F(|v|^2)+\frac\ka2|\partial_xh(|v|^2)|^2\Big)dx,
\end{align*}
where $F$ satisfies $F(r_0^2)=F'(r_0^2)=0$, and $F,h\in\boC^\infty([0,\infty);\R)$. When $v=\rho e^{i\theta}$, the energy writes
\begin{equation}E(v)=\int_\R\Big((1+2\ka\rho^2h'(\rho^2)^2)(\partial_x\rho)^2+\rho^2(\partial_x\theta)^2+F(\rho^2)\Big)dx.
\end{equation}
The natural choice of "energy space" is $$\boX(\R)=\{v\in H^1_{\loc}(\R) : \partial_xv\in L^2(\R), \text{ and } |v|^2-r_0^2\in L^2(\R)\}.$$
Indeed, one can show that the energy, the mass, and the momentum are well defined on $\boX(\R).$ Besides, we believe that the domain of the energy $\text{Dom}(E_\ka)\subset H^1_{\loc}(\R)$ is complicated whenever $F(\cdot)$ or $s\mapsto 1+2\ka s^2h'(s^2)^2$ are nonpositive.

If one is interested in studying the variational properties of $E_\ka(\cdot)$, the next result explains that it is natural to add the following constraint on $v\in\boX(\R)$ to ensure $E_\ka(v)\geq0$:
$$1+2\ka |v(x)|^2h'(|v(x)|^2)^2\geq0,\quad\text{ and }F(|v(x)|^2)\geq0,\quad\text{ for a.e. }x\in\R.$$
In particular, since $|v(x)|\to r_0$ as $x\to\pm\infty$, the first inequality implies $\ka\geq-1/(2r_0^2h'(r_0^2)^2)$.
\begin{proposition}\label{prop:illvar}
    Assume that one of the following statements holds
    \begin{enumerate}
        \item\label{item:illvar1} There exists $r\geq0$ with $F(r^2)<0$.
        \item\label{item:illvar2} There exists $r\geq0$ with $1+2\ka r^2h'(r^2)^2<0$.
    \end{enumerate}
    Then for all $\gp\in\R$, there exists a sequence $(v_n)\subset \boX(\R)$ satisfying
    \begin{equation}
        P(v_n)=\gp,\quad\text{ for all }n\in\N,\quad\text{ and } \lim_{n\to\infty}E_\ka(v_n)=-\infty.
    \end{equation}
\end{proposition}
\begin{proof}Recall that for $v=\rho e^{i\theta}$ the ene
    Since $h'$ and $F$ are continuous, we can assume $r>0$ and $r\ne r_0/4$ and $r\ne r_0$ in cases~\ref{item:illvar1}--\ref{item:illvar2}.
    Let us assume that Case~\ref{item:illvar1} holds.
    Consider the sequence of twisted plateau functions
    $v_n=\rho_n e^{i\theta_n}$ for $n>1$ with
       \begin{equation*}
           \rho_n(x)=\left\{\begin{aligned}
             r_0+(r-r_0)x, \quad&\text{ if }\quad x\in(0,1),\\
        r,\quad&\text{ if }\quad x\in(1,n),\\
        (r-r_0)(n-x)+r,\quad &\text{ if }\quad x\in(n,n+1),\\
        r_0,\quad&\text{ otherwise,}
           \end{aligned}\right.\quad\text{ and }\quad \theta_n(x)=\left\{\begin{aligned}
               0,\quad&\text{ if }\quad x\in(-\infty,0), \\
            Cx,\quad&\text{ if }\quad x\in(0,1),\\
            C, \quad&\text{ if }\quad x\in(1,\infty),
           \end{aligned}\right. 
       \end{equation*}
        where $C\in\R$ is fixed but to be determined.
        We compute
       \begin{align}\label{eq:momentPlateau}
            P(v_n)&=\int_\R\partial_x\theta_n(r_0^2-\rho_n^2)=C\int_0^1 r_0^2-(r_0+(r-r_0)x)^2dx=-\frac{C}{3}(r-r_0)(4r-r_0)\ne0,
        \end{align}
        since $r\ne r_0$ and $r\ne r_0/4$. 
        Taking $C=-3\gp/((r-r_0)(4r-r_0))$ we get $P(v_n)=\gp$ for every $n>1.$
       For the energy of $v_n$, using that $\partial_x\rho_n$ is supported and uniformly bounded in $[0,1]\cup[n,n+1]$, we compute
       \begin{equation}
E(v_n)=\int_{[0,1]\cup[n,n+1]}\Big((1+2\ka\rho_n^2h'(\rho_n^2)^2)\partial_x\rho_n^2+\rho_n^2\partial_x\theta_n+F(\rho_n^2)\Big)dx+F(r^2)\int_1^ndx.
       \end{equation}
We can show that the first integral is bounded and the second diverges to $-\infty$ as $n\to\infty$, therefore
$P(v_n)=\gp$ for all $n>1$ and $\lim_{n\to\infty}E(v_n)=-\infty.$

In Case~\ref{item:illvar2}, a slightly more elaborate example works, namely $v_n=\rho_n e^{i\theta_n}$ for $n>1$ with
\begin{equation*}
           \rho_n(x)=\left\{\begin{aligned}
             r_0+(r-r_0)x, \quad&\text{ if }\quad x\in(0,1),\\
        r+\delta\sin(2\pi nx),\quad&\text{ if }\quad x\in(1,2),\\
        (r-r_0)(n-x)+r,\quad &\text{ if }\quad x\in(2,3),\\
        r_0,\quad&\text{ otherwise,}
           \end{aligned}\right.\quad\text{ and }\quad \theta_n(x)=\left\{\begin{aligned}
               0,\quad&\text{ if }\quad x\in(-\infty,0), \\
            Cx,\quad&\text{ if }\quad x\in(0,1),\\
            C, \quad&\text{ if }\quad x\in(1,\infty),
           \end{aligned}\right. 
       \end{equation*}
       where $\delta>0$ satisfies $r_0,r_0/4\notin[r-\delta,r+\delta]$, with $r-\delta>0$ and $1+2\ka s^2h'(s^2)^2<0$ for all $s\in[r-\delta,r+\delta]$.
       Clearly equation \eqref{eq:momentPlateau} still holds true so that $C=-3\gp/((r-r_0)(4r-r_0))$ yields $P(v_n)=\gp$ for all $n>1$. Next, we compute
       \begin{equation}\label{eq:Eplateau}
           E(v_n)=\int_{\R\backslash(1,2)}e_\ka(v_n)+\int_1^2\Big((2\pi n\delta)^2(1+2\ka\rho_n^2h'(\rho_n^2)^2)\cos^2(2\pi nx)+F(\rho_n^2)\Big)dx.
       \end{equation}
       As in Case~\ref{item:illvar1}, the first integral in \eqref{eq:Eplateau} is bounded, also $r-\delta\leq\rho_n\leq r+\delta$ thus
       \begin{equation*}
           \int_1^2e_\ka(v_n)dx\leq 2\pi n\delta\inf_{s\in[r-\delta,r+\delta]}(1+2\ka s^2h'(s^2)^2)\int_1^22\pi n\cos^2(2\pi n x)dx+\sup_{s\in[r-\delta,r+\delta]}|F(s)|.
       \end{equation*}
       Since $\inf_{s\in[r-\delta,r+\delta]}(1+2\ka s^2h'(s^2)^2)<0$, we get that $\lim_{n\to\infty}E_\ka(v_n)=-\infty.$
\end{proof}
\begin{remark}
    Notice that, in case~\ref{item:illvar2}, the constructed comparison map $(v_n)\subset\boX(\R)$  enjoys an uniform bound on $|||v_n|^2-r_0^2||_{L^2(\R)}$. 
\end{remark}
\subsection{On the topology of the energy space and the hydrodynamical formulation}
Several distances have been introduced on $\boX(\R)$ to investigate the local well-posedness of \eqref{QGP} and the stability of traveling-wave solutions. Namely, $d_{\boX}$, $d_{\infty}$, $d_A$ and $d_f$ given by, for all $v,w \in\boX(\R)$
\begin{align*}
    d^2_{\boX}(v,w)&=\int_\R|\partial_xv-\partial_xw|^2dx+\int_\R(|v|-|w|)^2dx+|v(0)-w(0)|^2,\\
    d^2_{\infty}(v,w)&=\int_\R|\partial_xv-\partial_xw|^2dx+\int_\R(|v|-|w|)^2dx+\sup_{x\in\R}|v(x)-w(x)|^2,\\
    d_A^2(v,w)&=\int_\R|\partial_xv-\partial_xw|^2dx+\int_\R(|v|-|w|)^2dx+\sup_{x\in[-A,A]}|v(x)-w(x)|^2,\quad\text{ where }A>0.\\
    d_f^2(v,w)&=\int_\R|\partial_xv-\partial_xw|^2dx+\int_\R(|v|-|w|)^2dx+\int_\R f(\cdot)|v-w|^2,
\end{align*}
where $f(\cdot)>0$ and $f$ is continuous and bounded (e.g. $f=r_0^2-|u_{0,\ka}|^2$ given by Theorem~\ref{thm:soli} for $\ka>\tilde\ka$). We refer to~\cite{HaidarHydro} for a comprehensive study of the relations between these metrics regarding the local and global well-posedness of the semilinear equation~\eqref{QGP} with $\ka=0$.
Morrey's inequality provides the following equivalence of distance in dimension one
\begin{proposition}\label{prop:equivdist}Let $A>0$ and $f\in\boC(\R;\R)$. Assume $f(\cdot)>0$ in $\R$ and $\int_\R(1+|x|)fdx<\infty$ and consider the distances introduced above. Then $d_f,$ $d_{\boX}$, and $d_A$ are equivalent, that is there exists $K>1$ such that for all $v,w\in\boX(\R),$ we have
    $$\frac{1}{K}d_f(v,w)\leq d_{\boX}(v,w)\leq Kd_A(v,w)\leq K^2d_f(v,w).$$ 
    Moreover there holds
    $d_{\boX}(v,w)\leq d_{\infty}(v,w)$, but the Zhidkov metric $d_{\infty}$ is not equivalent to $d_{\boX}$. Precisely, there exists $(v_n)\subset\boX(\R)$ with $d_{\boX}(v_n,1)\to0$, and $d_{\infty}(v_n,1)\to|e^i-1|\ne0$.
\end{proposition}
\begin{proof}
Let $v,w\in\boX(\R)$.
Using the Cauchy--Schwarz inequality yields
\begin{align*}
    \int_\R f(x)|v(x)-w(x)|^2dx&=\int_\R f(x)\Big|v(0)-w(0)+\int_0^x(\partial_xv-\partial_x w)dy\Big|^2dx\\&\leq2||f||_{L^1(\R)}|v(0)-w(0)|^2+2||f||_{L^1(\R,|x|dx)}\int|\partial_xv-\partial_xw|^2,
\end{align*} hence $d_f(v,w)\lesssim d_{\boX}(v,w)$. Inequality $d_{\boX}(v,w)\lesssim d_A(v,w)$ is immediate.
Next, for all $x,y\in[-A,A],$ with $A/4<|x-y|<A/2,$ we have 
\begin{align*}
    |v(x)-w(x)|^2=\Big|v(y)-w(y)+\int_y^x\partial_xv(z)-\partial_xw(z)dz\Big|^2\leq 2|v(y)-w(y)|^2+ A||\partial_xv-\partial_xw||_{L^2(\R)}^2.
\end{align*}
Multiplying by $f(y)/\inf_{[-A,A]}f(\cdot)>1,$ and integrating on $y\in[-A,A]$ with $A/4<|x-y|<A/2$, we get
\begin{align*}
    \frac{A}{4}|v(x)-w(x)|^2&\leq \int_{\{y\in[-A,A]:A/4<|x-y|<A/2\}}\frac{f(y)}{\inf_{[-A,A]}f(\cdot)}|v(y)-w(y)|^2dy+\frac{A^2}{4}||\partial_xv-\partial_xw||^2_{L^2(\R)},\\
    &\leq\frac{1}{\inf_{[-A,A]}f(\cdot)}\int_\R f(y)|v(y)-w(y)|^2dy+\frac{A^2}{2}||\partial_xv-\partial_xw||^2_{L^2(\R)}.
\end{align*}
Taking the supremum in $x\in[-A,A],$ we deduce $d_A(v,w)\leq d_f(v,w)$.

Concerning the the Zhidkov-type distance $d_\infty$, there exists a sequence $(v_n)\subset\boX(\R)$ such that $d_{\boX}(v_n,1)\to0$ but $d_{\infty}(v_n,1)\to|e^i-1|\ne0$. Precisely, letting $\chi$ be a triangle wave given by
\begin{equation}\label{eq:trianglewave}
\chi(x)=\begin{cases}
    x+1,\quad\text{ in }[-1,0],\\
    1-x, \quad\text{ in }[0,1]\\
    0,\quad\text{ for all }x\in(-\infty,-1)\cup(1,\infty).
\end{cases} 
\end{equation} Then the sequence $(v_n)\subset\boX(\R)$ given by $v_n=e^{i\chi(n+x/n)}$ for all $n\geq1$ yields the aforementioned result. Indeed, changing variables $n+x/n=y$, $dx=ndy$, we get $$d_{\boX}(v_n,1)^2=\frac{1}{n^2}\int_\R|\partial_x\chi(n+x/n)|^2dx+|e^{i\chi(n)}-1|=\frac{1}{n}\int_{0}^{2}|\partial_x\chi(y)|^2dy+|e^{i\chi(n)}-1|\to 0,$$
as $n\to\infty.$
By similar computations, we get
$$d^2_{\infty}(v_n,1)^2=\frac{1}{n^2}\int_\R|\partial_x\chi(n+x/n)|^2dx+\sup_{x\in\R}|e^{i\chi(x/n)}(x)-1|^2\to|e^i-1|^2\ne0,\quad\text{ as }n\to\infty,$$
since $0\leq\chi(n+x/n)\leq\chi(1)=1$ for all $x\in\R$. 
Therefore $d_{\boX}$ and $d_{\infty}$ are not equivalent.
\end{proof}
One can also replace $\displaystyle{ \int_\R(|v|-|w|)^2}$ in $d_{\boX}$ by $\displaystyle{ \int_\R(|v|^2-|w|^2)^2}$ as explained by the next result.
\begin{proposition}\label{prop:lipmod}
For all $v,w\in\boX(\R),$  with $ \inf_\R(|v|+|w|)>0$ there holds \begin{equation}\label{ineq:L2mod}
    \inf(|v|+|w|)^2\int_\R(|v|-|w|)^2dx\leq\int_\R(|v|^2-|w|^2)^2dx\leq\sup_{\R}(|v|+|w|)^2\int_\R(|v|-|w|)^2dx,
\end{equation}
Moreover, letting $$d(v,w)\coloneqq\Big(\int_\R|\partial_xv-\partial_xw|^2+(|v|^2-|w|^2)^2dx+|v(0)-w(0)|^2\Big)^{1/2},$$ define a distance on $\boX(\R),$ then, the application $\text{Id}:(\boX(\R),d)\to(\boX(\R),d_{\boX})$
is locally bi-Lipschitz.
\end{proposition}
\begin{proof}
    Inequalities \eqref{ineq:L2mod} are easy to obtain. 
    Let $v\in\boX(\R)$ and take $w\in\boX(\R)$ with $d_{\boX}(v,w)\leq1$. Then we can find a $L^\infty$-bound on $w$, indeed, for all $x\in\R$, we have
    $$\int_{x-1}^xw(x)dy=\int_{x-1}^x\Big(w(y)+\int_{y}^x\partial_xw(\sigma) dr\Big) dy.$$
    Thus, by the Cauchy--Schwarz inequality, we get
    $$|w(x)|\leq\Big( \int_{x-1}^x|w|\Big)+\int_{x-1}^x|x-y|^{\frac12}dy\times||\partial_xw||_{L^2(\R)},$$
    and, by multiple applications of the triangular inequality, we conclude  that 
    $$||w||_{L^\infty(\R)}\leq || |v|-|w|||_{L^2(\R)}+||v||_{L^\infty}+||\partial_xv-\partial_xw||_{L^2(\R)}+||\partial_xv||_{L^2(\R)}\coloneqq C(v,d_{\boX}(v,w)). $$
    Combining that and the right inequality in \eqref{ineq:L2mod} yields, $d(v,w)\leq (C(v, d_{\boX}(v,w))+ ||v||_{L^\infty(\R)})d_{\boX}(v,w)$.

    To show that the inverse mapping is also Lipschitz continuous, let $v\in\boX(\R)$ and take $A\geq0$ to be such that $\inf_{|x|\geq A}|v(x)|>0$. If $w\in\boX(\R)$ with $d_{\boX}(v,w)\leq1$, then, on one hand, there holds
\begin{align}
    \int_{-A}^A(|v|-|w|)^2&\leq \int_{-A}^{A}|v(x)-w(x)|^2\leq\int_{-A}^A\Big|v(0)-w(0)+\int_{0}^x(\partial_xv(z)-\partial_xw(z))dz\Big|^2dx,\\
    &\leq4A\Big(|v(0)-w(0)|^2+\sqrt{A}\int_\R|\partial_xv(z)-\partial_xw(z)|^2dz\Big)\lesssim_Ad^2(v,w).
\end{align}
On the other hand, using the left inequality in \eqref{ineq:L2mod} we get
\begin{align*}
    \int_{\R\backslash(-A,A)}(|v|-|w|)^2\leq\frac{1}{\inf_{|x|\geq A}v(x)}d^2(v,w)
\end{align*}
Combining these inequalities yields $d_{\boX}(v,w)\leq K_{v}d(v,w),$ which concludes.
\end{proof}

For functions in the so-called nonvanishing energy space
$$\boN\boX(\R)\coloneqq\big\{v\in\boX(\R): \inf_{x\in\R}|v(x)|>0\big\}$$, one can introduce the hydrodynamical variables living in vector spaces by performing the Madelung transform
\begin{align}
    \boM:(\boN\boX(\R),d_{\boX})&\to (H^1(\R;\R)\times L^2(\R;\R)\times \S^1,||\cdot||_{hy})\\
    v(\cdot)=\sqrt{\eta(\cdot)-r_0^2}e^{i\theta(\cdot)}&\mapsto \boM(v)=\left(\eta,\partial_x\theta,\frac{v(0)}{|v(0)|}\right),
\end{align}
where $||(\eta,w,s)||^2_{hy}=||\eta||^2_{H^1(\R)}+||w||_{L^2(\R)}^2+|s|^2.$
The next result shows that the mapping $\boM$ is locally bi-Lipschitz. In particular, orbital stability in terms of the hydrodynamical distance amounts to orbital stability for the energy distance $d_{\boX}$.
\begin{proposition}[\cite{Chiron-stability}]
    Let $Y=\{\eta\in H^1(\R;\R): \eta>-r_0^2\}$. Then the mapping $$\boM : (\boN\boX(\R),d_{\boX})\to (Y\times L^2(\R;\R)\times \S^1, ||\cdot||_{hy}),$$ is a homeomorphism. However, $\boM^{-1}$ is not Lipshitz continuous, that is, there exist $(v_n)\subset\boN\boX(\R)$ and $v_\infty\in\boN\boX(\R)$  with
    $$\lim_{n\to\infty}d_{\text{hy}}(\boM(v_n),\boM(v_\infty))\to 0,\quad \text{ and } \lim_{n\to\infty}\frac{d_{\text{hy}}(\boM(v_n),\boM(v_\infty))}{d_{\boX}(v_n,v_\infty)}=\infty. $$
\end{proposition}
\begin{remark}
For $v\in\boX(\R)$ with $\inf_{x\in\R}|v(x)|>0$, we have
$$\eta=|v|^2-r_0^2,\quad\text{ and }\partial_x\theta=\Im(\partial_xv\bar v)/|v|^2.$$ It is tempting---however impossible---to extend $\boM$ to $\boX(\R)$ using these formulae. Indeed, the function $v(x)=r_0(1-e^{-x^4})(\cos(1/x)+i\sin(1/x))$ provides an example of a function $v\in\boX(\R)$ with $\Im(\partial_xv\bar v)/|v|^2$ not locally integrable. The mapping
$\boM$ is also not Lipshitz continuous (see Lemma~9 in \cite{HaidarHydro}).
\end{remark}
Still, the stability in the hydrodynamical distance amounts to the stability in the energy space distance. This subtlety comes from the fact that $\boM$ and $\boM^{-1}$ are locally Lipschitz continuous in the neighborhood of functions $v$, for which $|v|^2-r_0^2$ decays rapidly, such as the solitons, as explained by the following lemma.
\begin{lemma}
    Let $v\in\boN\boX(\R)$ with $\partial_xv\in L^2(\R;\sqrt xdx)$, then there exists $\delta>0$ such that the restriction of $\boM$ to $B_{d_{\boX}}(v,\delta)$ is bi-Lipshitz continuous on its image.
\end{lemma}
The fact that the energy is locally Lipschitz continuous with respect to $d_{\boX}$ is standard.
\begin{lemma}\label{lem:Elip}
  Let $v\in\boX(\R)$ and let $A>0$. There exists $K>0$ such for all $w\in\boX(\R)$, if $d(v,w)<A$, then there holds
  $$|E_\ka(v)-E_\ka(w)|\leq K d_\boX(v,w).$$
\end{lemma}

\subsection{The momentum in dimension one}
 \begin{lemma}\cite{deLGrSm1}\cite{bethuel-black}\label{lem:momentuntwist}
           For all $\Psi\in\boX(\R),$ the limit
               \begin{equation*}
                   \boP(\Psi)=\lim_{R\to+\infty}\left(\int_{-R}^R\langle i\Psi,\partial_x\Psi\rangle dx-r_0^2\Big[\arg(\Psi(x))\Big]^{x=R}_{x=-R}\right),
               \end{equation*}
              exists in $\R/ (2\pi r_0^2\Z)$ and $\boP\in\boC(\boX(\R);\R\backslash (2\pi r_0^2\Z))$. Moreover, if $\Psi\in\boX(\R)$ satisfies $\inf_{x\in\R}|\Psi(x)|>0,$ then $\boP(\Psi)\equiv P(\Psi)\mod 2r_0^2\pi.$ Finally, the function $\boP$ is locally Lipschitz continuous, i.e. there exists $j\in\Z$ such that
              $$|\boP(u)-\boP(\tilde u)+2j\pi r_0^2|\lesssim d_{\boX}(u,\tilde{u}),\quad\text{ for all } u,\tilde u\in\boX(\R)$$
           \end{lemma}

\section{Characterization of the infimum in the orbital stability inequality}\label{sec:infimum}
We adapt the proof of J.~L.~Bona and A.~Soyeur~\cite{BonaSoyeur} to control the modulation parameters: Since the function $g(t):(z,\varphi)\mapsto d_{\boX}^2(\Psi(\cdot,t),u_{0,\ka}(\cdot+z)e^{i\varphi})$ is $\boC^2(\R^2;\R)$ for all $t\in[0,T]$, using the implicit function theorem, we can find locally unique $z(\cdot),\varphi(\cdot)\in\boC^1([0,T])$ realizing the infimum of $g$ so that $D_{z,\varphi}g(t,z(t),\varphi(t))=0$. However $d_\boX(\R)$ is not smooth in the sense that $\partial_t(D_{z,\varphi}g(t,z(t),y(t)))$ have local terms involving $\partial_t\Psi(0,t)$ which cannot be controlled with $E_\ka(\Psi)$ or $d_{\boX}(\Psi(\cdot,0),u_{0,\ka}(\cdot))$. This prevents us from bounding  $z'(t),\varphi'(t)$ in terms of $d_{\boX}(\Psi,u_{c,\ka})$. Hence, we replace $d_{\boX}$ by a smoother distance 
$d_0$ given by 
\begin{equation*}
    d_0^2(v(\cdot),w(\cdot))=||\partial_xv-\partial_xw||_{L^2(\R)}^2+\int_\R(|v|^2-|w|^2)^2dx+\int_\R(1+x^4)^{-1}|v-w|^2dx,\quad\text{ for all }v,w\in\boX(\R).
\end{equation*} The estimations on $| z'(t)|$ and on $|\varphi'(t)|$ in Theorem~\ref{thm:stab} follows readily using equivalence of distances in Proposition~\ref{prop:equivdist}, and the local bi-Lipschitz inequalities in Proposition~\ref{prop:lipmod}.
\begin{proposition}\label{prop:paramEstim}
    There exist $\ve_0>0$ and a constant $K$ independant of $T>0$, such that if $\Psi\in\boC([0,T);H^s(\R))$ with $s>5/2$ is a solution to \eqref{QGP} satisfying
    \begin{equation}\label{eq:distquotient}
    \inf_{(z,\varphi)\in\R^2}d_0(\Psi(\cdot,t),w(\cdot+z)e^{i\varphi})<\ve_0, \quad\text{ for all }t\in[0,T),  
    \end{equation}
    then, there exist $z(\cdot),\varphi(\cdot)\in\boC^1([0,T);\R)$ that reach the infimum in \eqref{eq:distquotient}, and there holds \begin{equation*}
        |z'(t)|+| \varphi'(t)|\leq K d_0(\Psi(\cdot,t),u(\cdot-z(t))e^{i\varphi(t)}).
    \end{equation*}
\end{proposition} 
\begin{remark}
    Note that the orbit of black solitons is parametrized by two real numbers, instead of one in the case of KdV-related models. Thus, the arguments based on Rolle's theorem fail to provide global uniqueness of $(z(t), \varphi(t))\in\R\times[0,2\pi)$ reaching the infimum in \eqref{eq:distquotient}.
\end{remark}
\begin{remark}
    If $1\leq s<5/2$ and LWP holds in $H^s$, one must take $(\Psi_0^{(n)})\subset H^{\frac52^+}$ converging to $\Psi_0$ in $H^s(\R)$, perform the computations on the solutions starting from $\Psi_0^{(n)}$ and pass to the limit using the continuity of the flow map.
\end{remark}
\begin{proof}
Let $\ve_0>0$ be a small quantity determined later and $s>5/2$. Assume that $\Psi\in\boC([0,T),H^s(\R))$ is a solution to \eqref{QGP} satisfying \eqref{eq:distquotient}. Recall the definition of the (real-valued) scalar product $\langle\cdot,\cdot\rangle_\C$
$$\langle\Psi_1,\Psi_2\rangle_\C=\Re(\Psi_1\overline\Psi_2).$$
We start by showing the existence of the curve $z(\cdot),\varphi(\cdot)\in\boC^1([0,T);\R)$ that reaches the infimum in \eqref{eq:distquotient} using the implicit function theorem.
Recall that $|u_{0,\ka}|^2-r_0^2,$ and all its derivatives decay exponentially fast,
hence, by the theorem of derivation under the integral sign, we deduce that $(z,\varphi)\mapsto d_{0}(\Psi(\cdot,t),u_{0,\ka}(\cdot+z)e^{i\varphi})^2$ is $\boC^2(\R^2;\R)$ for all $t\in[0,T)$. Denoting its differential by $G(t,z,\varphi)=(g_1,g_2),$ we compute 
\begin{align}
    {g_1}=&-2\int_{\R}\langle\partial_{xx}u_{0,\ka}(x+z)e^{i\varphi},\partial_x\Psi(x,t)\rangle_\C dx-2\int_{\R}\partial_x(|u_{0,\ka}(x+z)|^2)|\Psi(x,t)|^2dx\notag\\\notag&
    -2\int_\R(1+x^4)^{-1}\langle\partial_xu_{0,\ka}(x+z)e^{i\varphi},{(\Psi(x,t)-u_{0,\ka}(x+z)e^{i\varphi})}\rangle_\C dx.
\end{align}
\begin{align*}
      g_2=-2\int_{\R}\langle i\partial_{x}u_{0,\ka}(x+z)e^{i\varphi},\partial_x\Psi(x,t)\rangle_\C dx-2\int_\R (1+x^4)^{-1}\langle iu_{0,\ka}(x+z)e^{i\varphi},{\Psi}(x,t)\rangle_\C dx.
\end{align*}
Letting $\delta(x,t)=\Psi(x,t)-u(x+z)e^{i\varphi}$, we compute $D_{z,\varphi}(g_1,g_2)^T$, where
\begin{align*}
    \partial_zg_1=&-2\int_\R\langle\partial_{x}^3u_{0,\ka}(x+z)e^{i\varphi},\partial_x\Psi(x,t)\rangle_\C dx-2\int_\R\partial_{xx}(|u_{0,\ka}(x+z)|^2)|\Psi(x,t)|^2dx\\
   &-2\int_\R(1+x^4)^{-1}(\langle\partial_{xx}u_{0,\ka}(x+z)e^{i\varphi},\delta(x,t)\rangle_\C-|\partial_xu_{0,\ka}(x+z)|^2)dx,\\
    =&2\int_\R(|\partial_{xx}u_{0,\ka}(x)|^2+(\partial_{x}|u_{0,\ka}(x)|^2)^2+(1+(x-z)^4)^{-1}|\partial_xu_{0,\ka}(x)|^2)dx\\
    &-2\int_\R\langle\partial_{x}^3u_{0,\ka}(x+z)e^{i\varphi},\partial_x{\delta}(x,t)\rangle_\C dx-2\int_\R\partial_{xx}(|u_{0,\ka}(x+z)|^2)(|\Psi(x,t)|^2-|u_{0,\ka}(x+z)|^2)dx\notag\\
    &-2\int_\R(1+x^4)^{-1}\langle\partial_{xx}u_{0,\ka}(x+z)e^{i\varphi},\delta(x,t)\rangle_\C dx,\\
    \partial_zg_2 = &- 2\int_{\R}\langle i\partial_{xx}u_{0,\ka}(x+z)e^{i\varphi},\partial_{x}\Psi(x,t)\rangle_\C dx-2\int_\R (1+x^4)^{-1}\langle i\partial_xu_{0,\ka}(x+z)e^{i\varphi},{\Psi}(x,t)\rangle_\C dx,\\
    =&-2\int_\R \langle i\partial_{xx}u_{0,\ka}(x),\partial_x u_{0,\ka}(x)\rangle_\C dx-2\int_\R(1+(x-z)^4)^{-1}\langle i\partial_xu_{0,\ka}(x), u_{0,\ka}(x)\rangle_\C dx\\
    &-2\int_\R \langle i\partial_{xx}u_{0,\ka}(x+z)e^{i\varphi},\partial_x \delta(x,t)\rangle_\C-2\int_\R (1+x^4)^{-1}\langle i\partial_xu_{0,\ka}(x+z)e^{i\varphi},{\delta}(x,t))\rangle_\C dx,
\end{align*}
and,
\begin{align*}
    \partial_\theta g_1=&-2\int_\R \langle i\partial_{xx}u_{0,\ka}(x),\partial_x u_{0,\ka}(x)\rangle_\C dx-2\int_\R(1+(x-z)^4)^{-1}\langle i\partial_xu_{0,\ka}(x), u_{0,\ka}(x)\rangle_\C dx\\
    &-2\int_\R \langle   i\partial_{xx}u_{0,\ka}(x+z)e^{i\varphi},\partial_x\delta(x,t)\rangle_\C+(1+x^4)^{-1}\langle i\partial_xu_{0,\ka}(x+z)e^{i\varphi},{\delta}(x,t)\rangle_\C dx,\\
    \partial_\theta g_2=&2\int_{\R}(|\partial_xu_{0,\ka}(x)|^2+(1+(x-z)^4)^{-1}|u_{0,\ka}(x)|^2)dx\\
    &+2\int_\R\langle\partial_xu(x+z)e^{i\varphi},\partial_x \delta(x,t)\rangle_\C dx+2\int_\R
    (1+x^4)^{-1}\langle u_{0,\ka}(x+z)e^{i\varphi},\delta(x,t)\rangle_\C dx.
\end{align*}
For all $t\in[0,T]$, we can find $(z(t),\varphi(t))\in\R\times[0,2\pi)$ satisfying \eqref{eq:distquotient}: Indeed, 
there are $\theta_+,\theta_-\in\R$ such that $|\partial_xu_{0,\ka}(x+z)|+|u_{0,\ka}(x+z)-e^{i\theta_\pm}r_0|\to0$  for all $x\in\R$ as $z\to\pm\infty$. Then, we can show, using Lebesgue's dominated convergence theorem, that
for all $v\in\boX(\R)$ and all $\varphi\in\R,$ there holds
\begin{align*}
    \lim_{z\to\pm\infty}d_{0}(v,u_{0,\ka}(\cdot+z)e^{i\varphi})^2=&\int_\R|\partial_xu_{0,\ka}|^2+(|u_{0,\ka}|^2-r_0^2)^2dx\\&+\int_{\R}|\partial_xv|^2+(|v|^2-r_0^2)^2+(1+x^4)^{-1}|v-r_0e^{i\theta_\pm}|^2dx\\
    \geq&\int_\R|\partial_xu_{0,\ka}|^2+(|u_{0,\ka}|^2-r_0^2)^2dx.
\end{align*}
Using this, and the inequality in \eqref{eq:distquotient}, we deduce that for $\ve_0$ small enough, the infimum in \eqref{eq:distquotient} is attained at $(z(t),\varphi(t))\in\R\times[0,2\pi]$.
Hence we have $G(t,z(t),\varphi(t))=0$ and we can decompose $D_{z,\varphi}G(t,z,\varphi)$ as
\begin{align*}
    2\begin{pmatrix}
        a_1&a_2\\
        a_2&a_3
    \end{pmatrix}+H,
\end{align*}
where \begin{align}
    a_1&= \int_\R|\partial_{xx}u_{0,\ka}(x)|^2+(\partial_x|u_{0,\ka}(x)|^2)^2+(1+(x-z)^4)^{-1}|\partial_xu_{0,\ka}(x)|^2dx,\\
     a_2&=-\int_\R\langle i\partial_{xx}u_{0,\ka}(x)\partial_x, u_{0,\ka}(x)\rangle_\C dx-\int_\R (1+(x-z)^4)^{-1}\langle i\partial_xu_{0,\ka}(x),u_{0,\ka}(x)\rangle_\C dx.\\
    a_3&=\int_\R|\partial_xu_{0,\ka}(x)|^2+(1+(x-z)^4)^{-1}|u_{0,\ka}(x)|^2dx,
\end{align}
For all $(w_1,w_2)\in\R^2$, the Young inequality provides for all $\ve>0$ 
\begin{align*}
    2a_2w_1w_2\geq &-w_1^2\Big(\int_\R|\partial_{xx}u_{0,\ka}(x)|^2+\frac{1}{\ve}(1+(x-z)^4)^{-1}|\partial_xu_{0,\ka}(x)|^2dx\Big)\\&-w_2^2\Big(\int_\R|\partial_xu_{0,\ka}(x)|^2+\ve(1+(x-z)^4)^{-1}|u_{0,\ka}(x)|^2dx\Big).
\end{align*}
Taking $\ve=(1+\int_\R(1+(x-z)^4)^{-1}|\partial_xu_{0,\ka}(x)|^2dx/(\int_\R(1+(x-z)^4)|\partial_xu_{0,\ka}(x)|^2+(\partial_x|u_{0,\ka}|^2)^2dx))/2,$ we get $0<\ve<1$ and
$$\Big(1-\frac{1}{\ve}\Big)(1+(x-z)^4)^{-1}|\partial_xu_{0,\ka}(x)|^2\geq-\frac{(1+(x-z)^4)^{-1}|\partial_xu_{0,\ka}(x)|^2}{2(1+(x-z)^4)^{-1}|\partial_xu_{0,\ka}(x)|^2+(\partial_x|u_{0,\ka}|^2)^2}(\partial_x|u_{0,\ka}|^2)^2.$$
Hence, there exists $K>0$ such that $a_1w_1^2+2a_2w_1w_2+a_3w_2^2\geq K^{-1}|w|^2,$ for all $(w_1,w_2)\in\R^2$.
Up to taking larger $K$, we get the following bound on $H$ using the Cauchy-Schwarz inequality:  $$\langle Hw,w\rangle_{\R^2\times\R^2}\geq-Kd_{0}(\Psi(\cdot,t),u_{0,\ka}(\cdot+z)e^{i\varphi})|w|^2,\quad\text{ for all }w\in\R^2.$$ Thus, assuming $0<\ve_0<1/(4K^4)$ in \eqref{eq:distquotient}, then combining the estimate on $H$ and the coercivity estimate above, we get
\begin{equation}\label{ineq:coercimodu}
    \langle D_{z,\varphi}G(t,z(t),\varphi(t))w, w\rangle_{\R^2\times\R^2}\geq \frac{1}{2K}|w|^2,\quad\text{ for all }w\in\R^2.
\end{equation}
By the implicit function theorem, we deduce that the curve $(z(\cdot),\varphi(\cdot))$ is $\boC^1$ in $[0,T).$

To conclude we estimate $\partial_tG(t,z(t),\varphi(t))$ in terms of $d_0(\Psi(\cdot,t),u_{0,\ka}(\cdot)).$ Since $G(t,z(t),\varphi(t))=0$ for all $t\in[0,T)$  and $\Psi\in\boC^1([0,T);H^{s-2}(\R)$ with $s>2+d/2$, by derivation under the integral, we obtain the following ODE on $(z(\cdot),\varphi(\cdot))$
\begin{equation}\label{eq:ODEzvarph}
D_{z,\varphi}G(t,z,\varphi)_{\rvert(z(t),\varphi(t))}\begin{pmatrix}
       \partial_t z(t)\\\partial_t\varphi(t)
    \end{pmatrix}=2\begin{pmatrix}
        \int_\R\langle\partial_{xx}u_{0,\ka}(x+z)e^{i\varphi},\partial_{xt}\Psi(x,t)\rangle_\C dx\\
        +\int_\R\partial_{x}|u_{0,\ka}(x+z)|^2\partial_{t}|\Psi(x,t)|^2dx\\
        +\int_\R(1+x^4)^{-1}\langle\partial_xu_{0,\ka}(x+z)e^{i\varphi},\partial_t\Psi(x,t)\rangle_\C dx\\\\
        \int_\R \langle i\partial_xu_{0,\ka}(x+z)e^{i\varphi},\partial_{xt}{\Psi}(x,t)\rangle_\C dx\\
        +\int_\R (1+x^4)^{-1}\langle iu_{0,\ka}(x+z)e^{i\varphi}\partial_t\Psi(x,t)\rangle_\C dx
    \end{pmatrix}.
\end{equation} 
For conciseness we denote by $u$ the function $u_{0,\ka}(\cdot+z)e^{i\theta}$. Then writing $\Psi(x,t)=u(x,t)+\delta(x,t)$  and $|\Psi(x,t)|^2=|u(x,t)|^2+\mathfrak{N}(x,t),$ for all $(x,t)\in\R^2,$ we compute
\begin{align}
    \ka\overline\Psi h'(|\Psi|^2)\partial_{xx}h(|\Psi|^2)=&
    \ka(\overline{u+\delta})\partial_{xx}(h(|\Psi|^2)(h'(|u|^2)+\mathfrak{N}( h''(|u|^2)+R_1)),
\end{align}
where $||R_1(\cdot,t)||_{L^\infty(\R)\cap\dot{H}^1(\R)}\leq C(||\delta(\cdot,t)||_{L^\infty(\R)\cap \dot H^1(\R)})$ for some positive increasing function $C(\cdot)$.
Since equation \eqref{QGP} involves terms in $\partial_{xx}h(|\Psi|^2)$, and because $d_0$ controls only first order terms, we will integrate by part in \eqref{eq:ODEzvarph} to recover products of order one terms. This fact motivates the decomposition \begin{align}\notag
&\partial_{xx}h(|\Psi|^2)=2(\langle\partial_{xx}\Psi,\Psi\rangle_\C+|\partial_x\Psi|^2)h'(|\Psi|^2)+4\langle\partial_x\Psi,\Psi\rangle_\C^2h''(|\Psi|^2),\\\notag
    =&2\Re( \partial_{xx}u\bar u+\partial_{xx}u\delta+\partial_{xx}\delta\bar u++\partial_{xx}\delta\bar\delta+|\partial_xu|^2+2\partial_x\delta\partial_x\bar u+|\partial_x\delta|^2)\\&\times\Big(h'(|u|^2)+\mathfrak{N}(h''(|\Psi|^2)+R_2)\Big)\notag
    +4\Big(\Re^2(\partial_xu\bar u)+2\Re(\partial_xu\bar u)\Re(\partial_x\delta\bar u+\partial_xu\bar\delta+\partial_x\delta\bar\delta)\\\notag
    &+\Re^2(\partial_x\delta\bar u+\partial_xu\bar\delta+\partial_x\delta\bar\delta)\Big)
    \times(h''(|u|^2)+\mathfrak{N}(h^{(3)}(|u|^2)+R_3),\\\notag
=&\partial_{xx}h(|u|^2)+2\Re(\partial_{xx}\delta\bar\Psi+2\partial_x\delta\partial_x\bar u+|\partial_x\delta|^2)h'(|\Psi|^2)+2\Re(\partial_{xx}u\bar u)\mathfrak{N}(h''(|u|^2)+R_2)
\\\notag&+4\Big(\Re^2(\partial_x\delta\bar u+\partial_xu\bar\delta+\partial_x\delta\bar\delta)+2\Re(\partial_xu\bar u)(\Re(\partial_x\delta\bar u+\partial_xu\bar\delta+\partial_x\delta\bar\delta)\Big)h''(|\Psi|^2)\\&+4\Re^2(\partial_xu\bar u)\mathfrak{N}(h^{(3)}(|u|^2)+R_3),
\end{align}
where $||R_j(\cdot,t)||_{L^\infty(\R)\cap\dot{H}^1(\R)}\leq C_j(||\delta(\cdot,t)||_{L^\infty(\R)\cap \dot H^1(\R)})$ for some positive increasing functions $C_j(\cdot)$, with $j=2,3$.
We can now rewrite \eqref{QGP} as
\begin{align}\label{eq:QGPlin0}
-i\partial_t\bar\Psi&=\notag\partial_{xx}(\overline{u+\delta})+\bar{u}\Big(f(|u|^2)+\ka h'(|u|^2)\partial_{xx}h(|u|^2)\Big)+\bar\delta\Big(f(|\Psi|^2)+\ka h'(|\Psi|^2)\partial_{xx}h(|\Psi|^2)\Big)\\\notag
&+\overline{u}\mathfrak{N}(\ka(h''(|u|^2)+R_1)\partial_{xx}h(|\Psi|^2)+R_4)+2\ka\bar u h'(|\Psi|^2)\Big(\Re(\partial_{xx}\delta\bar\Psi+2\partial_x\delta\partial_x\bar u+|\partial_x\delta|^2)h'(|\Psi|^2)\\&+2h''(|\Psi|^2)(\Re^2(\partial_x\delta\bar u+\partial_xu\bar\delta+\partial_x\delta\bar\delta)+2\Re(\partial_xu\bar u)\Re(\partial_x\delta\bar u+\partial_xu\bar\delta+\partial_x\delta\bar\delta))\Big),
\end{align}
where $||R_4||_{L^\infty\cap\dot{H}^1(\R)}\leq C_4(||\delta||_{L^\infty\cap\dot{H}^1(\R)})$ for some positive increasing function $C_4(\cdot).$
Since $u$ solves $\text{(TW$(0,\ka)$)}$, Eq \eqref{eq:QGPlin0} vanish at the zeroth order in $\mathfrak{N},\delta$.

It remains to bound the right-hand side of \eqref{eq:ODEzvarph} in terms of $d_0(\Psi,u)$ to conclude. We bound the integral in the first line of \eqref{eq:ODEzvarph} in terms of $d_{0}(\Psi,u)$, the remaining estimates follow in the same manner. Since $u_{0,\ka}$ and its derivative are smooth and exponentially decaying, by integration by parts, and H\"older inequalities, there holds
\begin{align}
    &\Big|\int_\R\langle\partial_{xx}u,\partial_{xt}\Psi(x,t)\rangle_\C dx\Big|=\\\notag
    &\Big|i\int_\R\partial_{x}^3u\Big(\partial_{xx}\bar\delta+\bar\delta\ka h'(|\Psi|^2)\partial_{xx}h(|\Psi|^2)+\ka\bar u\mathfrak{N}(h''(|u|^2)+R_1)\partial_{xx}h(|\Psi|^2)+2\ka\bar u(h'(|\Psi|^2))^2\Re(\partial_{xx}\delta\bar\Psi)\\ \notag
    &\qquad+\bar\delta f(|\Psi|^2)+\bar u \mathfrak{N} R_4+2\ka\bar uh'(|\Psi|^2)\Big(\Re(2\partial_x\delta\partial_x\bar u+|\partial_x\delta|^2)h'(|\Psi|^2)\\\notag
    &\qquad+2h''(|\Psi|^2)(\Re^2(\partial_x\delta\bar u+\partial_xu\bar\delta+\partial_x\delta\bar\delta)+2\Re(\partial_xu\bar u)\Re(\partial_x\delta\bar u+\partial_xu\bar\delta+\partial_x\delta\bar\delta))\Big)\Big)dx\Big|,\\\notag
    &\leq\Big|i\int_\R\partial_{x}^3u\Big(\partial_{xx}\bar\delta+\bar\delta\ka h'(|\Psi|^2)\partial_{xx}h(|\Psi|^2)+\ka\bar u\mathfrak{N}(h''(|u|^2)+R_1)\partial_{xx}h(|\Psi|^2)\\\notag
    &\qquad+2\ka\bar u(h'(|\Psi|^2))^2\Re(\partial_{xx}\delta\bar\Psi)\Big)dx\Big|+C(||u||_{W^{3,\infty}(\R)\cap\dot H^3(\R)},||\delta||_{L^\infty(\R)\cap\dot H^1(\R)})\times d_0(\Psi,u),
\end{align}
with $C(\cdot,\cdot)$ positive and incresing in every variable.
To remove the second-order derivatives on $\delta$, integrate by parts the integral in the last two lines above to obtain
\begin{align}\notag
    \Big|\int_\R\langle\partial_{xx}u,\partial_{xt}\Psi(x,t)\rangle_\C dx\Big|\leq& C(||u||_{W^{3,\infty(\R)}\cap\dot H^3(\R)},||\delta||_{L^\infty(\R)\cap\dot H^1(\R)})\\&\times(||\partial_x^3u||_{H^1(\R,(1+x^4)dx)}\int_{\R}(1+x^4)^{-1}\delta^2dx+||\partial_x\mathfrak{N}||_{L^2(\R)}+d_0(\Psi,u)).\notag
\end{align}
Proceeding similarly for the other integral terms in \eqref{eq:ODEzvarph} we infer, after inverting $D_{z,\varphi}G$
\begin{equation}\label{eq:coroparamCCL}
    \left|\begin{pmatrix}
        \partial_tz\\\partial_t\varphi
    \end{pmatrix}\right|\leq C(||u||_{W^{3,\infty(\R)}\cap\dot H^3(\R)},||\delta(\cdot,t)||_{L^\infty(\R)\cap\dot H^1(\R)},||\partial_x^3u||_{H^1(\R,(1+x^4)dx)})d_0(\delta(\cdot,t),u),
\end{equation}
where $C(\cdot,\cdot,\cdot)>0$ is a positive function increasing in every variable.
The energy estimates in Lemma~\ref{lem:pointwiseEnergy} enables us, for $\ve_0>0$ small enough, to bound $||\delta(\cdot,t)||_{L^\infty(\R)\cap \dot H^1(\R)}$ with quantities in terms of $E(\Psi(\cdot,0))$ and the norms of $u_{0,\ka}$. We can conclude using this and \eqref{eq:coroparamCCL}.
\end{proof}
\begin{merci}
				The author acknowledges support from the CDP C2EMPI and is grateful to the French
State under the France-2030 programme, the University of Lille, the Initiative
d'excellence de l'Université de Lille, and the Métropole Européenne de Lille for
funding the R-CDP-24-004-C2EMPI project. He was also supported by the R\'egion Hauts-de-France. 
                He is grateful to A.~de Laire and O.~Goubet for the helpful discussions and comments on this work.
			\end{merci}			
			\bibliographystyle{abbrv}   
			\bibliography{ref}
\end{document}